\documentclass[reqno, oneside, 12pt]{amsart}
\usepackage[letterpaper]{geometry}
\usepackage{amsmath,amssymb}
\usepackage{amsthm}
\usepackage{thmtools}
\usepackage{mathtools}
\usepackage{dsfont}
\usepackage{color}
\usepackage{enumitem}
\usepackage{hyperref}

\usepackage{mathabx}
\usepackage{mathrsfs}

\usepackage{comment}

\usepackage{tikz}

\usepackage{cases}
\usepackage{multirow}

\usepackage[normalem]{ulem}

\newcommand{\Z}{{\mathbb{Z}}}
\newcommand{\Q}{{\mathbb{Q}}}
\newcommand{\R}{{\mathbb{R}}}
\newcommand{\C}{{\mathbb{C}}}

\newcommand{\HH}{{\mathbb{H}}}
\newcommand{\Mform}{{\mathcal{M}}}
\newcommand{\Scusp}{{\mathcal{S}}}
\newcommand{\Lform}{{\mathcal{L}}}
\newcommand{\LEigenform}{{\tilde{\mathcal{L}}}}

\newcommand{\U}{{\mathcal{U}}}
\newcommand{\mR}{{\mathcal{R}}}

\DeclareMathOperator{\re}{Re}
\DeclareMathOperator{\im}{Im}

\DeclareMathOperator{\rank}{rank}
\DeclareMathOperator{\spann}{span}

\newcommand{\ep}{\varepsilon}

\newcommand{\defeq}{\vcentcolon=}
\def\SL{{\rm SL}}
\def\GL{{\rm GL}}
\newcommand{\floor}[1]{\left\lfloor #1 \right\rfloor}

\newcommand{\ceil}[1]{\left\lceil #1 \right\rceil}

\renewcommand{\(}{\left(}
\renewcommand{\)}{\right)}
\newcommand{\la}{\left|}
\newcommand{\ra}{\right|}
\newcommand{\Ea}{E_{\mathfrak{a}}}
\newcommand{\Eb}{E_{\mathfrak{b}}}
\newcommand{\ma}{\mathfrak{a}}

\newcommand{\Mod}[1]{\ (\mathrm{mod}\ #1)}

\DeclareMathOperator{\ch}{ch}
\DeclareMathOperator{\sh}{sh}
\DeclareMathOperator{\sgn}{sgn}

\newtheorem{theorem}{Theorem}

\newtheorem{lemma}[theorem]{Lemma}
\newtheorem{corollary}[theorem]{Corollary}
\newtheorem{proposition}[theorem]{Proposition}

\newtheorem{definition}[theorem]{Definition}

\newtheorem{setting}[theorem]{Setting}

\theoremstyle{remark}
\newtheorem*{remark}{Remark}

\numberwithin{equation}{section}
\numberwithin{theorem}{section}
\numberwithin{lemma}{section}
\numberwithin{proposition}{section} 
\numberwithin{example}{section}
\numberwithin{definition}{section}
\numberwithin{corollary}{section}
\numberwithin{setting}{section}

\author{Qihang Sun}
\address{Department of Mathematics, University of Illinois, Urbana, IL 61801}
\email{qihangs2@illinois.edu}
\title[Uniform bounds for Kloosterman sums]{Uniform bounds for Kloosterman sums of half-integral weight with applications}

\date{\today}

\begin{document}

\begin{abstract}
Sums of Kloosterman sums have deep connections with the theory of modular forms, and their estimation has many important consequences. Kuznetsov used his famous trace formula and got a power-saving estimate with respect to $x$ with implied constants depending on $m$ and $n$. Recently, in 2009, Sarnak and Tsimerman obtained a bound uniformly in $x$, $m$ and $n$. 

The generalized Kloosterman sums are defined with multiplier systems and on congruence subgroups. Goldfeld and Sarnak bounded sums of them with main terms corresponding to exceptional eigenvalues of the hyperbolic Laplacian. Their error term is a power of $x$ with implied constants depending on all the other factors. In this paper, for a wide class of half-integral weight multiplier systems, we get the same bound with the error term uniformly in $x$, $m$ and $n$. 

Such uniform bounds have great applications. For the eta-multiplier, Ahlgren and Andersen obtained a uniform and power-saving bound with respect to $m$ and $n$, which resulted in a convergent error estimate on the Rademacher exact formula of the partition function $p(n)$.  We also establish a Rademacher-type exact formula for the difference of partitions of rank modulo $3$, which allows us to apply our power-saving estimate to the tail of the formula for a convergent error bound.

\end{abstract}
\maketitle

\section{Introduction: sum of Kloosterman sums} 
For a positive integer $c$, the standard Kloosterman sum
\[S(m,n,c)\defeq\sum_{d\Mod c^*}e\(\frac{m\overline{d}+nd}c\),\quad e(z)\defeq e^{2\pi i z},\quad \overline{d} d\equiv 1 \Mod c\]
has a trivial bound $c$ and a well-known Weil bound
\begin{equation}\label{Weilbound0.5}
	|S(m,n,c)|\leq \sigma_0(c)(m,n,c)^{\frac12}c^{\frac12}, \quad \text{where\ } \sigma_k(\ell)=\sum\nolimits_{d|\ell}d^k
\end{equation}
is the divisor function. The Weil bound implies a square-root cancellation for estimating
\begin{equation}\label{StandardWeilBound}
	\sum_{c\leq x}\frac{S(m,n,c)}c\ll \sigma_0((m,n)) x^{\frac12}\log x. 
\end{equation}
In the 1960s, Linnik \cite{Linnikconj} and Selberg \cite{selberg} pointed out the connection between such sums and modular forms. They conjectured that there should be a full cancellation, which was reformulated by Sarnak and Tsimerman in \cite{SarnakTsimerman09} as
\[\sum_{c\leq x}\frac{S(m,n,c)}c\ll_\ep | {m} {n}x|^{\ep}. \]
Kuznetsov \cite{Kuznetsov1980trFml} applied his famous trace formula which resulted in the bound
\[\sum_{c\leq x}\frac{S(m,n,c)}c\ll_{m,n} x^{\frac16}(\log x)^{\frac13}. \]
Sarnak and Tsimerman \cite{SarnakTsimerman09} obtained a bound which is uniform in $m$ and $n$ for $ {m}  {n}>0$:
\begin{equation}\label{STbound}
	\sum_{c\leq x}\frac{S(m,n,c)}c\ll_\ep \(x^{\frac16}+( {m} {n})^{\frac16}+(m+n)^{\frac18}( {m}  {n})^{\frac\theta2}\) ( {m}  {n}x)^\ep, 
\end{equation}
where $\theta$ is an admissible exponent towards the Ramanujan-Petersson conjecture for $\GL_2/\Q$. One may take $\theta=\frac7{64}$ by the work of Kim and Sarnak \cite{KimSarnak764}. 

In this paper we focus on the generalized Kloosterman sums defined by 
\[S(m,n,c,\nu)\defeq\sum_{\substack{0\leq a,d<c \\ \gamma=\begin{psmallmatrix} a&b\\ c&d  		\end{psmallmatrix} 		\in \Gamma}} \overline{\nu}(\gamma)e \(\frac{\tilde ma+\tilde nd}c\)\]
where $\Gamma$ is a congruence subgroup of $\SL_2(\Z)$ with $\begin{psmallmatrix}
	1&1\\0&1
\end{psmallmatrix}\in \Gamma$, $\nu$ is a weight $k$ multiplier system on $\Gamma$, and $\tilde{n}\defeq n-\alpha_{\nu}$ for $\alpha_{\nu}\in [0,1)$ defined by $e(-\alpha_{\nu})=\nu(\begin{psmallmatrix}
	1&1\\0&1
\end{psmallmatrix})$. These Kloosterman sums have been studied by Goldfeld and Sarnak \cite{gs} and Pribitkin \cite{pribitkin}. Sums of generalized Kloosterman sums are related to many problems in number theory. For example, if we denote $p(n)$ as the partition function, which is the number of ways to write the natural number $n$ as a sum of positive integers (e.g. $p(4)=5$), then we have Rademacher's exact formula \cite[(1.8)]{Rademacher1937pn} \cite[(1.2), (1.3)]{AAimrn}
\begin{align}\label{RademacherExactFormula}
	\begin{split}
	p(n)&=\frac1{\pi\sqrt2}\sum_{c=1}^\infty A_c(n)\sqrt c\,\frac{d}{dn}\Bigg(\frac{\sinh\(\frac\pi c\sqrt{\frac23(n-\frac1{24})}\)}{\sqrt{n-\frac 1{24}}}\Bigg)\\
	&=\frac{2\pi}{(24n-1)^{\frac 34}}\sum_{c=1}^\infty \frac{A_c(n)}c  I_{\frac32}\(\frac{4\pi\sqrt{24n-1}}{24c}\) \\
	&=\frac{2\pi e(-\frac18)}{(24n-1)^{\frac 34}}\sum_{c=1}^\infty \frac{S(1,1-n,c,\nu_\eta)}c  I_{\frac32}\(\frac{4\pi\sqrt{24n-1}}{24c}\) 
\end{split}
\end{align}
where $A_c(n)$ is in \eqref{AknAsKloostermanSum}, $\nu_\eta$ is the multiplier system of weight $\frac12$ for Dedekind's eta-function \eqref{etaMultiplier} and $I_\kappa$ is the $I$-Bessel function. Since we know Bessel functions well, bounds for the sum of Kloosterman sums
\begin{equation}\label{sumofKlstmSums}
	\sum_{c\leq x}\frac{S(m,n,c,\nu_\eta)}c
\end{equation}
will result in estimates of the tail of \eqref{RademacherExactFormula}
\begin{equation}\label{errorp}
	R_1(n,x)\defeq\frac{2\pi e(-\frac18)}{(24n-1)^{\frac 34}}\sum_{c>x}\frac{S(1,1-n,c,\nu_\eta)}c  I_{\frac32}\(\frac{4\pi\sqrt{24n-1}}{24c}\). 
\end{equation}
A bound for $R_1(n,x)$ will tell us how many terms are needed in order to approximate the integer $p(n)$ with a given accuracy. 

An important step on the asymptotics of \eqref{sumofKlstmSums} was obtained by Goldfeld and Sarnak \cite{gs}. Define 
\[\beta=\limsup_{c\rightarrow\infty} \frac{\log|S(m,n,c,\nu)|}{\log c}\text{\ \ \  and \ \ }Z(m,n,s,\nu)=\sum_{c=1}^{\infty}\frac{S(m,n,c,\nu)}{c^{2s}} \]
as Selberg's Kloosterman zeta function. Then they showed that
\begin{equation}\label{GoldfSbound}
	\sum_{c\leq x}\frac{S(m,n,c,\nu)}c=\sum_{\frac12<s_j<1}\tau_j(m,n)\frac{x^{2s_j-1}}{2s_j-1}+O\(x^{\frac \beta 3+\ep}\)
\end{equation}
where the sum runs over simple poles of $Z(m,n,\cdot,\nu)$ in $(\frac12,1)$. Here $\tau_j$ and the implied constant depend on $m,n,k,\nu$ and $\Gamma$.

We obtain a result analogous to \eqref{GoldfSbound}. Our result is uniform in $m$ and $n$ and applies to a general class of multipliers, which includes both the theta- and eta-multipliers as well as their twists by quadratic characters. 
This result has a nice application to exact formulas for arithmetic functions related to integer partitions, which we will show in the next section. 

To state the results, we first need to classify our admissible multipliers:  
\begin{definition}\label{Admissibility}
Let $(k,\nu')=(\frac12,(\frac{|D|}\cdot)\nu_\theta)$ or $(-\frac12,(\frac{|D|}\cdot)\overline{\nu_\theta})$ where $D$ is some even fundamental discriminant and $\nu_\theta$ is the multiplier for the theta function. 
We say a weight $k$ multiplier $\nu$ on $\Gamma=\Gamma_0(N)$ is \underline{admissible} if it satisfies the following two conditions:
\begin{enumerate}\label{MultplCond}
	\item[(1)] Level lifting: there exist positive integers $B$ and $M$ such that the map $z\rightarrow Bz$ gives an injection from weight $k$ automorphic eigenforms of the Laplacian on $(\Gamma_0(N),\nu)$ to those on $(\Gamma_0(M),\nu')$ and keeps the eigenvalue. 
	Here $M$ is a multiple of $4$ and $M$ depends on $B$.
	\item[(2)] Average Weil bound: for $x>y>0$ and $x-y\gg x^{\frac 23}$, we have
	\[\sum_{N|c\,\in[y,x]}\frac{|S(m,n,c,\nu)|}{c}\ll_{\nu,\ep} (\sqrt x-\sqrt y)| \tilde{m}  \tilde{n} x|^\ep \]
\end{enumerate}	
\end{definition}

\begin{remark}
The exponent $\frac23=1-\delta$ comes from a special parameter $\delta$ among the proof which is finally chosen to be $\frac13$. An individual Weil-type bound on $S(m,n,c,\nu)$ may imply the average bound specified in condition (2), but our result only needs this weaker requirement. 
\end{remark}

Let $\rho_j(n)$ denote the $n$-th Fourier coefficients of an orthonormal basis $\{v_j(\cdot)\}_j$ of $\LEigenform_{k}(N,\nu)$ (the space of square-integrable eigenforms of the weight $k$ Laplacian $\Delta_k$ with respect to $(\Gamma_0(N),\nu)$, see Section~3 for details). Here are our theorems: 

\begin{theorem}\label{mainThm}
	Suppose $\nu$ is an \textit{admissible} weight $k=\pm\frac12$ multiplier on $\Gamma_0(N)$, $  \tilde{m}>0$ and $  \tilde{n} <0$. 
	Then when $B\tilde{\ell}$ is square-free or coprime to $M$ for $\ell\in\{m,n\}$, we have
	\begin{equation}
		\sum_{N|c\leq X} \frac{S(m,n,c,\nu)}{c}=\!\!\!\sum_{r_j\in i(0,\frac14]}\!\!\! \tau_j(m,n)\frac{X^{2s_j-1}}{2s_j-1}
		+O_{\nu,\ep}\(\(|\tilde{m}  \tilde{n}|^{\frac{143}{588}}+X^{\frac16}\) |  \tilde{m}  \tilde{n}X|^\ep\),
	\end{equation}
	where the sum runs over exceptional eigenvalues $\lambda_j<\frac14$ of the hyperbolic Laplacian $\Delta_k$: $\lambda_j=\frac14+r_j^2$, $s_j=\frac12+\im r_j$, and
	\[{\tau_j(m,n)}=2^{3-2s_j}\, i^k\, \pi^{1-2s_j} \,\Gamma(2s_j-1) \overline{\rho_j(m)}\rho_j(n) | \tilde{m}  \tilde{n}|^{1-s_j}\]
	are the coefficients as in \cite{gs} (as corrected by \cite[Proposition~7]{AAAlgbraic16}). 
\end{theorem}

\begin{remark}

	This theorem, along with the subsequent theorems presented in this paper  (Theorem~\ref{mainThm2}, Theorem~\ref{mainThm3} and Theorem~\ref{mainThmLastSec}), is limited to cases where $\tilde m\tilde n<0$ in the half-integral weight. Although the statement of Theorem~\ref{mainThm} is for $\tilde m>0$ and $\tilde n<0$, when $\tilde m<0$ and $\tilde n>0$, the same result can be directly verified via \eqref{KlstmSumConj}.

	For the contrast case $\tilde m\tilde n>0$, the author has submitted another paper \cite{QihangSecondAsympt} in which we prove a similar uniform bound.  

\end{remark}

We can estimate the contribution from exceptional eigenvalues to obtain 
\begin{theorem}\label{mainThm2}
	With the same setting as Theorem~\ref{mainThm}, we have
	\begin{equation}
		\sum_{N|c\leq X}\!\!\! \frac{S(m,n,c,\nu)}{c}\ll_{\nu,\ep}\(|\tilde{m}  \tilde{n}|^{\frac{131}{588}-\frac\theta2}X^{\theta}+  |  \tilde{m}  \tilde{n} |^{\frac{143}{588}} +X^{\frac16}\) |  \tilde{m}  \tilde{n}X|^\ep. 
	\end{equation}

\end{theorem}

\begin{remark}
	We can take $\theta=\frac7{64}$ as the Kim-Sarnak bound, then the term involving $\theta$ dominates when 
\[|  \tilde{m}  \tilde{n}|^{\frac{471}{686}}<X<|  \tilde{m}  \tilde{n}|^{2-\frac 3{1078}}. \]
When $\theta\leq \frac1{47}$, the $\theta$ term never dominates.
\end{remark}

The requirement that $B\tilde{m}$ and $B\tilde{n}$ are square-free or coprime to $M$ can be removed but this results in a more complicated bound. For positive integers $\alpha$ and $\beta$, we write $\alpha|\beta^\infty$ if there exists some positive integer $K$ such that $\alpha|\beta^K$. 
\begin{theorem}\label{mainThm3}
	With the same setting as Theorem~\ref{mainThm} (except the requirement for $B\tilde{m}$ and 
	$B\tilde{n}$), factor $|B \tilde{\ell}|=t_\ell u_\ell^2 w_\ell^2$ where $t_\ell$ is square-free, $u_\ell|M^\infty$ and $(w_\ell,M)=1$ for $\ell\in\{m,n\}$. Then
	\begin{equation}
		\begin{split}
		\sum_{N|c\leq X} \frac{S(m,n,c,\nu)}{c}-\!\!\!\sum_{r_j\in i(0,\frac14]}\!\!\! \tau_j(m,n)\frac{X^{2s_j-1}}{2s_j-1}
		\ll_{\nu,\ep}\left(A_u(m,n)+X^{\frac16}\right)|  \tilde{m}  \tilde{n}X|^\ep
		\end{split}
	\end{equation}
	where
	\begin{align*}A_u(m,n)&\defeq \(\tilde{m}^{\frac{131}{294}}+u_m\)^{\frac18} \(|\tilde{n}|^{\frac{131}{294}}+u_n\)^{\frac18} |\tilde{m}\tilde{n}|^{\frac3{16}}\\
		&\ll| \tilde{m}  \tilde{n}|^{\frac{143}{588}}+\tilde{m}^{\frac{143}{588}}|\tilde{n}|^{\frac 3{16}}\,u_n^{\frac18}+\tilde{m}^{\frac 3{16}}|\tilde{n}|^{\frac{143}{588}}\,u_m^{\frac18}+| \tilde{m}  \tilde{n}|^{\frac3{16}}(u_mu_n)^{\frac18}.
	\end{align*}
\end{theorem}
The middle two terms above are not negligible. As a corollary we have
\begin{corollary}
	With the same setting as Theorem~\ref{mainThm3}, suppose $u_\ell^2\ll\tilde{\ell}^\kappa$ for $\ell\in\{m,n\}$ and some $\kappa\in[0,1]$. Then  
	\begin{equation}
		\begin{split}
			\sum_{N|c\leq X} \frac{S(m,n,c,\nu)}{c}-\!\!\!\sum_{r_j\in i(0,\frac14]}\!\!\! \tau_j(m,n)\frac{X^{2s_j-1}}{2s_j-1}
			\ll_{\nu,\ep}\left(| \tilde{m}  \tilde{n}|^{\max\left\{\frac{143}{588},\frac{3+\kappa}{16}\right\}}+X^{\frac16}\right)|  \tilde{m}  \tilde{n}X|^\ep,
		\end{split}
	\end{equation}
and
\begin{align*}
	\sum_{N|c\leq X} \frac{S(m,n,c,\nu)}{c}\ll_{\nu,\ep} \(|  \tilde{m}  \tilde{n}|^{\max\left\{\frac{131}{588},\frac \kappa4 \right\}-\frac\theta2}X^{\theta}+ | \tilde{m}  \tilde{n}|^{\max\left\{\frac{143}{588},\frac{3+\kappa}{16}\right\}}+X^{\frac16}\)|  \tilde{m}  \tilde{n}X|^\ep. 
\end{align*}

\end{corollary}
The quantities in both maxima are equal when $\kappa=\frac{131}{147}$. When $\frac{131}{147}<\kappa\leq 1$ and if we choose $\theta=\frac7{64}$, then the $\theta$ term dominates when  
\[|  \tilde{m}  \tilde{n}|^{\frac12+\frac{12}{7}(1-\kappa)}< X <|  \tilde{m}  \tilde{n}|^{\frac{27}{11}-\frac{48}{11}(1-\kappa)}. \]

We modify our estimate to get the following bound suitable for the applications in the next section. Recall that $\alpha_{\nu}\in [0,1)$ is defined as $\nu(\begin{psmallmatrix}
1&1\\0&1
\end{psmallmatrix})=e(-\alpha_\nu)$. 
\begin{theorem}\label{mainThmLastSec}
	With the same setting as Theorem \ref{mainThm}, for $\beta=\frac12$ or $\frac 32$ and $\alpha>0$, we have
	\begin{equation}\label{mainThmLastSecTheEquation}
		\sum_{N|c> \alpha\sqrt{|\tilde m\tilde{n}|}}\frac{S(m,n,c,\nu)}c \mathscr{M}_{\beta}\(\frac{4\pi \sqrt{|\tilde m\tilde{n}|}}c\)\ll_{\alpha,\nu,\ep}|\tilde m\tilde{n}|^{\frac{143}{588}+\ep},
	\end{equation}
	where $\mathscr{M}_\beta$ is the Bessel function $I_\beta$ or $J_\beta$. 
\end{theorem}
\begin{remark}
This theorem still works when we have $\tilde m<0$ and $\tilde n>0$, because we can take the conjugation with \eqref{KlstmSumConj} and $\mathscr{M}_\beta$ is a real function on $\R$. 
\end{remark}

\section{Introduction: rank of partitions and Maass-Poincar\'e series}
Suppose $\Lambda=\{\Lambda_1\geq \Lambda_2\geq \cdots\geq \Lambda_\kappa\}$ is a partition of $n$, i.e. $\sum_{j=1}^\kappa \Lambda_j=n$. Let 
\[\rank(\Lambda)\defeq\Lambda_1-\kappa\]
define the rank of this partition, $N(m,n)$ denote the number of partitions of $n$ with rank $m$, and $N(a,b;n)$ denote the number of partitions of $n$ with rank $\equiv a\Mod b$. Let $q=\exp(2\pi i z)=e(z)$ for $z\in \HH$ (the upper-half complex plane) and $w$ be a root of unity. It is known that the generating function of $N(m,n)$ is
\begin{equation}\label{rankGeneratingFunction}
	\mR(w;q)\defeq 1+\sum_{n=1}^\infty\sum_{m=-\infty}^\infty N(m,n) w^mq^n=1+\sum_{n=1}^\infty\frac{q^{n^2}}{(wq;q)_n(w^{-1}q;q)_n},
\end{equation}
where $(a;q)_n\defeq\prod_{j=0}^{n-1}(1-aq^j)$. 
For example, $\mR(1;q)=1+\sum p(n)q^n$ is the generating function for partitions. For integers $b>a>0$, denote $A(\frac ab;n)$ as
\[\mR(\zeta_b^a;q)=:1+\sum_{n=1}^\infty A\(\frac ab;n\)q^n\]
where $\zeta_b=\exp(\frac{2\pi i}b)$ is a $b$-th root of unity. The following identity is easy to get but helpful in understanding the relation between $A(\frac ab;n)$ and $N(a,b;n)$:
\begin{equation}
	bN(a,b;n)=p(n)+\sum_{j=1}^{b-1}\zeta_b^{-aj}A\(\frac jb;q\). 
\end{equation}

The function $\mR(w;q)$ has many beautiful connections and properties. When $w=-1$, it is known that $N(0,2;n)-N(1,2;n)=A(\frac12;n)$ is the Fourier coefficient of Ramanujan's third order mock theta function $f(q)$. We know that the Hardy-Ramanujan asymptotics 
\[p(n)\sim \frac1{4n\sqrt3}e^{\pi\sqrt{\frac 23n}}\]
was perfected by Rademacher's exact formula \eqref{RademacherExactFormula}. Similarly, Dragonette \cite{Dragonette1952} and Andrews \cite{Andrews1966} improved the asymptotic formula of $A(\frac12;n)$ which was conjectured by Ramanujan. The exact formula for $A(\frac12;n)$ was later proven by Bringmann and Ono:
\begin{theorem}[{\cite[Theorem~1.1]{BrmOno2006ivt}}]\label{exactFmlThmMod2}
	The Andrews-Dragonette conjecture is true: 
	\begin{equation}\label{exactFmlMod2}
		A\(\frac 12;n\)=\frac{\pi}{(24n-1)^{\frac 14}}\sum_{k=1}^\infty \frac{(-1)^{\floor{\frac{k+1}2}}A_{2k}(n-\frac{k(1+(-1)^k)}4)}k I_{\frac12}\(\frac{\pi\sqrt{24n-1}}{12k}\). 
	\end{equation}
\end{theorem}

When $w$ is a third root of unity, we find that $A(\frac13;n)=A(\frac23;n)$ is the Fourier coefficient of a Ramanujan's sixth order mock theta function denoted as $\gamma(q)$ in \cite{AndrewsMockThetaSixthOrder}. Bringmann deduced a formula \cite[Theorem~1.1]{BringmannTAMS} for $A(\frac13;n)=N(0,3;n)-N(1,3;n)$, which was used to prove the Andrews-Lewis conjecture about comparing $N(0,3;n)$ and $N(1,3;n)$: 
\[A\(\frac13;n\)=\frac{4\sqrt 3 i}{\sqrt{24n-1}}\sum_{3|c\leq \sqrt{n}}\frac{B_{1,3,c}(-n,0)}{\sqrt c} \sinh\(\frac{\pi\sqrt{24n-1}}{6c}\)  +O(n^\ep).\]
Here $B_{\ell,u,c}(n,m)$ is the notation in \cite[(1.11)]{BringmannTAMS} and we omit its definition here. 
Bringmann and Ono (\cite[Theorem~1.3]{BringmannOno2012}) stated that the above sum, when extended to infinity, gives an exact formula for $A(\frac13;n)$ up to the coefficient of a linear combination of theta functions. A complete proof would require a generalization of \cite{BrmOno2006ivt}. 

Here we provide a proof which also yields uniform estimates for the error in the approximation of this series by partial sums. 
\begin{theorem}\label{exactFmlThmMod3}
	The following equation is true: 
	\begin{equation}\label{exactFmlMod3}
		A\(\frac 13;n\)=\frac{2\pi\,e(-\frac18)}{(24n-1)^{\frac14}}\sum_{3|c>0}\frac{S(0,n,c,\,(\frac \cdot3)\overline{\nu_\eta})}{c}I_{\frac12}\(\frac{\pi\sqrt{24n-1}}{6c}\). 
	\end{equation}
\end{theorem}
\begin{remark}
The relations between our notations can be derived by \eqref{KlstmSumConj}, 
	{\cite[Remark before Sect. 4 on p. 3493]{BringmannTAMS}},
	\[B_{1,3,c}(-n,0)=e(-\tfrac 18)S(1,1-n,c,(\tfrac \cdot 3)\nu_\eta),\quad\text{and}\quad I_{\frac12}(z)=\sqrt{\tfrac2{\pi z}}\;\sinh z. \]
\end{remark}

Let $R_j(n,x)$ be the tail sum on $n>x$ of the above exact formulas for $A(\frac1j;n)$: $j=1$ for \eqref{RademacherExactFormula}, $j=2$ for \eqref{exactFmlMod2} (see \eqref{ExactFmlMod2_inKlstmSumCase}) and $j=3$ for \eqref{exactFmlMod3}. Each multiplier system $\nu$ in these sums satisfies $\alpha_{\nu}=\frac1{24}$ or $\frac{23}{24}$ with $B=24$ and $M|576$ in Definition~\ref{Admissibility}. Clearly both $24n-1$ and $24n-23$ are always coprime with $M$. The admissibility of each multiplier is proved in \cite{AAimrn}, \cite{ahlgrendunn} and Proposition~\ref{Section3Prop}, respectively. Now we can apply Theorem~\ref{mainThmLastSec} to get a power-saving with exponent $\frac14-\frac{143}{588}=\frac1{147}$ less:
\begin{theorem} For $\alpha>0$ we have
\[
R_j(n,\alpha\sqrt{n})\ll_{\alpha,\ep}
 \left\{ 
 \begin{array}{ll}
	n^{-\frac12-\frac1{147}+\ep}, &j=1;\\
	n^{-\frac1{147}+\ep}, &j=2,3.
\end{array} \right.\]
\end{theorem}

When $j=1,2$, this improves \cite[Theorem~1.4, Theorem~1.1]{ahlgrendunn} by removing the square-free requirement. Recently Andersen and Wu \cite[Theorem~1.1]{AndersenWu2022bound36} proved a stronger bound when $j=1$:
\[R_1(n,\alpha\sqrt{n})\ll_{\alpha,\ep} n^{-\frac12-\frac1{36}+\ep}. \]
Another new contribution of us is to extend to the case $j=3$.  

In Section 3, we introduce notation, and in Section 4, we prove Theorem~\ref{exactFmlThmMod3} by following the outline of \cite{BrmOno2006ivt}. Our result in Section 1 substitutes the work in \cite[\S 4]{BrmOno2006ivt} and simplifies the proof. To prove the theorems in Section 1, we require the spectral theory of Maass forms. In Section 5, we provide examples of admissible multipliers, and in Section 6, we state Kuznetsov's trace formula with a test function $\phi$. In Section 7, we estimate $\check\phi$ on the full spectrum.

In Section 8, we prove two average bounds for the coefficients of Maass forms. One bound works well when the spectral parameter is large, while the other works well for small spectral parameters. We also recall the Maass-Selberg relation to deal with the contribution from the continuous spectrum for general level $N$, which is usually dropped by positivity when there are no singular cusps (as in \cite{AAimrn,ahlgrendunn}). However, it must be considered in our general case. In Section 9, we recover the $\tau_j$ terms and complete the proofs of Theorem~\ref{mainThm}, \ref{mainThm2}, \ref{mainThm3}, and \ref{mainThmLastSec}. Finally, we include the proof of Theorem~\ref{FourierExpn_ofPoincareSeries} in the last section.

{
\section{Background and notations}

In this section we recall some basic facts about Maass forms with general weight and multiplier, which can be found in various references like \cite{Proskurin2005,pribitkin,DFI12,AAimrn,ahlgrendunn}. Let $\Gamma = \Gamma_0(N)$ for some $N\geq 1$ denote our congruence subgroup and $\HH$ denote the upper-half complex plane. Fixing the argument $(-\pi,\pi]$, for any $\gamma\in\SL_2(\R)$ and $z=x+iy\in\HH$, we define the automorphic factor
\[j(\gamma,z)\defeq \frac{cz+d}{|cz+d|}=e^{i\arg(cz+d)} \]
and the weight $k$ slash operator 
\[(f|_k\gamma)(z)\defeq j(\gamma,z)^{-k}f(\gamma z) \]
for $k\in \R$. 
We say that $\nu:\Gamma\to \C^\times$ is a multiplier system of weight $k$ if
\begin{enumerate}[label=(\roman*)]
	\item $|\nu|=1$,
	\item $\nu(-I)=e^{-\pi i k}$, and
	\item $\nu(\gamma_1 \gamma_2) =w_k(\gamma_1,\gamma_2)\nu(\gamma_1)\nu(\gamma_2)$ for all $\gamma_1,\gamma_2\in \Gamma$, where
	\[w_k(\gamma_1,\gamma_2)\defeq j(\gamma_2,z)^k j(\gamma_1,\gamma_2z)^k j(\gamma_1\gamma_2,z)^{-k}.\]
\end{enumerate}
If $\nu$ is a multiplier system of weight $k$, then it is also a multiplier system of weight $k'$ for any $k'\equiv k\pmod 2$, and its conjugate $\overline\nu$ is a multiplier system of weight $-k$. 
One can check the basic properties that
\begin{equation}\label{MultiplierSystemBasicProprety}
	\quad \nu(\gamma)\nu(\gamma^{-1})=1,\quad \nu(\gamma \begin{psmallmatrix}
		1&b\\0&1
	\end{psmallmatrix})=\nu(\gamma)\nu(\begin{psmallmatrix}
	1&1\\0&1
\end{psmallmatrix})^b. 
\end{equation}

For any cusp $\mathfrak{a}$ of $\Gamma$, let $\Gamma_{\mathfrak{a}}$ denote its stabilizer in $\Gamma$. For example, $\Gamma_\infty=\{\pm\begin{psmallmatrix}
	1&b\\0&1
\end{psmallmatrix}:b\in\Z\}$. Let $\sigma_{\mathfrak{a}}\in\SL_2(\R)$ denote a scaling matrix satisfying $\sigma_{\mathfrak{a}}\infty=\mathfrak{a}$ and $\sigma_{\mathfrak{a}}^{-1} \Gamma_{\mathfrak{a}}\sigma_{\mathfrak{a}}=\Gamma_\infty$. 
We define $\alpha_{\nu,\mathfrak{a}} \in [0,1)$ by the condition
\begin{equation}\label{AlphaDefine}
\nu\( \sigma_{\mathfrak{a}}\begin{psmallmatrix} 
	1 & 1 \\ 
	0 & 1
\end{psmallmatrix}\sigma_{\mathfrak{a}}^{-1} \) = e(-\alpha_{\nu,\mathfrak{a}}).
\end{equation}
The cusp $\mathfrak{a}$ is called singular if $\alpha_{\nu,\mathfrak{a}}=0$. When $\mathfrak{a}=\infty$ we drop the subscript and denote $\alpha_{\nu}\defeq\alpha_{\nu,\infty}$. For $n\in \Z$, define $ \tilde{n} \defeq n-\alpha_{\nu}$ and $n_{\mathfrak{a}}\defeq n-\alpha_{\nu,\mathfrak{a}}$.
The Kloosterman sums for the cusp pair $(\infty,\infty)$ with respect to $\nu$ are given by 
\begin{equation}
	\label{eq:kloos_def}
	S(m,n,c,\nu) :=\!\!\! \sum_{\substack{0\leq a,d<c \\ \gamma=\begin{psmallmatrix} a&b\\ c&d 
			\end{psmallmatrix}
			\in \Gamma}}\!\!\!  \overline\nu(\gamma) e\(\frac{ \tilde{m} a+ \tilde{n} d}{c}\)
	=\!\!\! \sum_{\substack{\gamma\in \Gamma_\infty\setminus\Gamma/\Gamma_\infty\\\gamma=\begin{psmallmatrix} a&b\\ c&d 
	\end{psmallmatrix}}} \!\!\!  \overline\nu(\gamma) e\(\frac{ \tilde{m} a+ \tilde{n} d}{c}\).
\end{equation}
They satisfy the relationships
\begin{equation}\label{KlstmSumConj}
	\overline{S(m,n,c,\nu)}=\left\{\begin{array}{ll}
		S(1-m,1-n,c,\overline \nu)&\ \  \text{if\ }\alpha_{\nu}>0,\\
		S(-m,-n,c,\overline\nu)  &\ \  \text{if\ }\alpha_{\nu}=0,
	\end{array}\right. 
\end{equation}
because 
\begin{equation}\label{tildeNConj}
	n_{\overline{\nu}}=\left\{\begin{array}{ll}
		-(1-n)_{\nu}&\ \  \text{if\ } \alpha_{\nu}>0,\\
		n  &\ \ \text{if\ }\alpha_{\nu}=0.
	\end{array}\right. 
\end{equation}

We call a function $f:\HH\rightarrow \C$ automorphic of weight $k$ and multiplier $\nu$ on $\Gamma$ if 
\[f|_k\gamma =\nu(\gamma)f \quad\text{for all }\gamma\in\Gamma.\] Let $\mathcal{A}_k(N,\nu)$ denote the linear space consisting of all such functions and $\Lform_k(N,\nu)\subset \mathcal{A}_k(N,\nu)$ denote the space of square-integrable functions on $\Gamma\setminus\HH$ with respect to the measure \[d\mu(z)=\frac{dxdy}{y^2}\]
and the Petersson inner product 
\[\langle f,g\rangle \defeq\int_{\Gamma\setminus\HH} f(z)\overline{g(z)}\frac{dxdy}{y^2} \]
for $f,g\in\Lform_k(N,\nu)$. For $k\in \R$, the Laplacian
\begin{equation}\label{Laplacian}
	\Delta_k\defeq y^2\(\frac{\partial^2}{\partial x^2}+\frac{\partial^2}{\partial y^2}\)-iky \frac{\partial}{\partial x}
\end{equation}
can be expressed as
\begin{align}
	\Delta_k&=-R_{k-2}L_k-\frac k2\(1-\frac k2\)\\
	&=-L_{k+2}R_k+\frac k2\(1+\frac k2\)
\end{align}
where $R_k$ is the Maass raising operator
\begin{equation*}
	R_k\defeq \frac k2+2iy\frac{\partial}{\partial z}=\frac k2+iy\(\frac{\partial}{\partial x}-i\frac{\partial}{\partial y}\)
\end{equation*}
and $L_k$ is the Maass lowering operator
\begin{equation*}
	L_k\defeq \frac k2+2iy\frac{\partial}{\partial \bar z}=\frac k2+iy\(\frac{\partial}{\partial x}+i\frac{\partial}{\partial y}\).
\end{equation*}
These operators raise and lower the weight of an automorphic form as
\[(R_k f)|_{k+2}\;\gamma=R_k(f|_{k}\gamma),\quad (L_k f)|_{k-2}\;\gamma=L_k(f|_{k}\gamma),\quad  \text{for\ } f\in\mathcal{A}_k(N,\nu)\]
and satisfy the commutative relations
\begin{equation}\label{RLOperators}
	R_k\Delta_k=\Delta_{k+2}R_k,\quad L_k\Delta_k=\Delta_{k-2}L_k. 
\end{equation}
Moreover, $\Delta_k$ commutes with the weight $k$ slash operator for all $\gamma\in\SL_2(\R)$. 

We call a real analytic function $f:\HH\rightarrow \C$ an eigenfunction of $\Delta_k$ with eigenvalue $\lambda\in\C$ if
\[\Delta_k f+\lambda f=0. \]
From \eqref{RLOperators}, it is clear that an eigenvalue $\lambda$ for the weight $k$ Laplacian is also an eigenvalue for weight $k\pm 2$. 
We call an eigenfunction $f$ a Maass form if $f\in \mathcal{A}_k(N,\nu)$ is smooth and satisfies the growth condition
\[(f|_k\gamma)(x+iy)\ll y^\sigma+y^{1-\sigma} \]
for all $\gamma\in\SL_2(\Z)$ and some $\sigma$ depending on $\gamma$ when $y\rightarrow +\infty$. Moreover, if a Maass form $f$ satisfies
\[\int_0^1 (f|_k\sigma_{\mathfrak{a}})(x+iy)\;e(\alpha_{\nu,\mathfrak{a}}x)dx=0\]
for all cusps $\mathfrak{a}$ of $\Gamma$, then $f\in\Lform_k(N,\nu)$ and we call $f$ a Maass cusp form. For details see \cite[\S 2.3]{AAimrn}

Let $\mathcal{B}_k(N,\nu)\subset \Lform_k(N,\nu)$ denote the space of smooth functions $f$ such that both $f$ and $\Delta_k f$ are bounded. One can show that $\mathcal{B}_k(N,\nu)$ is dense in $\Lform_k(N,\nu)$ and $\Delta_k$ is self-adjoint on $\mathcal{B}_k(N,\nu)$. If we let $\lambda_0\defeq \lambda_0(k)=\tfrac {|k|}2(1-\tfrac {|k|}2)$, then for $f\in\mathcal{B}_k(N,\nu)$,
\[\langle f,-\Delta_k f\rangle\geq \lambda_0 \langle f,f\rangle,\]
i.e. $-\Delta_k$ is bounded from below. By the Friedrichs extension theorem, $-\Delta_k$ can be extended to a self-adjoint operator on $\Lform_k(N,\nu)$. 
The spectrum of $\Delta_k$ consists of two parts: the continuous spectrum $\lambda\in [\frac14,\infty)$ and a discrete spectrum of finite multiplicity contained in $[\lambda_0,\infty)$. 

Non-zero eigenfunctions corresponding to eigenvalue $\lambda_0$ come from holomorphic modular forms. To be precise, let $M_k(N,\nu)$ denote the space of holomorphic modular forms of weight $k$ and multiplier $\nu$ on $\Gamma$. There is a one-to-one correspondence between all $f\in\Lform_k(N,\nu)$ with eigenvalue $\lambda_0$ and weight $k$ holomorphic modular forms $F$ by
\begin{equation}\label{CuspFormR0} 
	f(z)=\left\{\begin{array}{ll}
	y^{\frac k2}F(z)\quad & k\geq 0,\ F\in M_k(N,\nu),\\
	y^{-\frac k2}\overline{F(z)}\quad & k<0,\ F\in M_{-k}(N,\overline{\nu}). 
\end{array}
\right. 	\end{equation}
For the Fourier expansion $\sum_{n\in \Z} a_y(n)e(\tilde n x)$ of such $f$, we have
\[
\left\{
\begin{array}{lll}
k\geq 0 &\Rightarrow& a_y(n)=0 \text{\ for\ } \tilde n<0,\\
k<0 &\Rightarrow& a_y(n)=0 \text{\ for\ } \tilde n>0.
\end{array}
\right.\]

Let $\lambda_\Delta(G,\nu,k)$ denote the first eigenvalue larger than $\lambda_0$ in the discrete spectrum with respect to the congruence subgroup $G$, weight $k$ and multiplier $\nu$. For weight 0, Selberg showed that $\lambda_\Delta(\Gamma(N),\boldsymbol{1},0)\geq \frac3{16}$ for all $N$ \cite{selberg} and Selberg's famous eigenvalue conjecture states that $\lambda_\Delta(G,\boldsymbol{1},0)\geq\frac14$ for all $G$.  
We introduce the hypothesis $H_\theta$ as
\begin{equation}\label{HTheta}
	H_\theta:\quad \lambda_\Delta(\Gamma_0(N),\boldsymbol{1},0)\geq \tfrac14-\theta^2\ \  \text{for all\ }N. 
\end{equation}
Selberg's conjecture includes $H_0$ and the best progress known today is $H_{\frac 7{64}}$ by \cite{KimSarnak764}. 
We denote $\lambda_\Delta(G,\nu,k)$ as $\lambda_\Delta$ when $(G,\nu,k)$ is clear from context.

Let $\LEigenform_{k}(N,\nu)\subset \Lform_k(N,\nu)$ denote the subspace spanned by eigenfunctions of $\Delta_k$. For each eigenvalue $\lambda$, we write 
\[\lambda=\tfrac14+r^2=s(1-s), \quad s=\tfrac12+ir,\quad r\in i(0,\tfrac14]\cup[0,\infty). \]
So $r\in i\R$ corresponds to $\lambda<\frac14$ and any such $\lambda\in(\lambda_0,\frac 14)$ is called an exceptional eigenvalue. Set
\begin{equation}\label{FirstSpecPara}
	r_\Delta(N,\nu,k)\defeq i\cdot\sqrt{\tfrac14-\lambda_\Delta(\Gamma_0(N),\nu,k)}. 
\end{equation}

Let $\LEigenform_{k}(N,\nu,r)\subset\LEigenform_{k}(N,\nu)$ denote the subspace corresponding to the spectral parameter $r$. Complex conjugation gives an isometry 
\[\LEigenform_{k}(N,\nu,r)\longleftrightarrow\LEigenform_{-k}(N,\overline\nu,r)\]
between normed spaces. 
For each $v\in \LEigenform_{k}(N,\nu,r)$, we have the Fourier expansion
\begin{equation}\label{FourierExpMaassEigenform}
	v(z)=v(x+iy)=c_0(y)+\sum_{ \tilde{n}\neq 0} \rho(n) W_{\frac k2\sgn \tilde{n},\;ir}(4\pi| \tilde{n}|y)e( \tilde{n} x)
\end{equation}
where $W_{\kappa,\mu}$ is the Whittaker function and
\[c_0(y)=\left\{ 
\begin{array}{ll}
	0&\alpha_{\nu}\neq 0,\\
	0&\alpha_{\nu}=0 \text{ and } r\geq 0,\\
	\rho(0)y^{\frac12+ir}&\alpha_{\nu}=0 \text{ and } r\in i(0,\frac14]. 
\end{array}
\right. 
\]
Using the fact that $W_{\kappa,\mu}$ is a real function when $\kappa$ is real and $\mu\in \R\cup i\R$ \cite[(13.4.4), (13.14.3), (13.14.31)]{dlmf}, if we denote the Fourier coefficient of $f_c\defeq \bar f$ as $\rho_c(n)$, then
\begin{equation}\label{FourierCoeffConj}
	\rho_c(n)=\left\{\begin{array}{ll}
	\overline{\rho(1-n)},&	\alpha_{\nu}>0\text{ and }n\neq 0,\\
	\overline{\rho(-n)}, &\alpha_{\nu}=0. 
\end{array}
\right.\end{equation}


There are two fundamental multiplier systems of weight $\frac12$. The theta-multiplier $\nu_{\theta}$ on $\Gamma_{0}(4)$ is given by
\begin{equation}\label{thetaFunctionStandard}
	\theta(\gamma z) = \nu_{\theta}(\gamma) \sqrt{cz+d}\; \theta(z), \quad \gamma=\begin{pmatrix}
		a&b\\c&d
	\end{pmatrix}\in \Gamma_0(4)
\end{equation}
where
\[
\theta(z) \defeq \sum_{n\in\Z} e(n^2 z), \quad \nu_{\theta}(\gamma)=\(\frac cd\)\ep_d^{-1}, \quad \ep_d=\left\{ \begin{array}{ll}
	1&d\equiv 1\Mod 4,\\
	i&d\equiv 3\Mod 4, 
\end{array}\right.
\]
and $\(\frac \cdot\cdot\)$ is the extended Kronecker symbol. 
The eta-multiplier $\nu_\eta$ on $\SL_2(\Z)$ is given by
\begin{equation}
	\eta(\gamma z) = \nu_{\eta}(\gamma) \sqrt{cz+d}\; \eta(z), \quad \gamma=\begin{pmatrix}
		a&b\\c&d
	\end{pmatrix}\in \SL_2(\Z)
\end{equation}
where
\begin{equation}
	\eta(z) \defeq q^{\frac1{24}}\prod_{n=1}^\infty (1-q^n),\quad q=e(z).
\end{equation}
Let $((x))\defeq x-\floor{x}-\frac12$ when $x\in \R\setminus\Z$ and $((x))\defeq 0$ when $x\in \Z$. We have the explicit formula \cite[(74.11), (74.12)]{Rad73Book}
\begin{equation}\label{etaMultiplier}
	\nu_\eta(\gamma)=e\(-\frac18\)e^{-\pi i s(d,c)}e\(\frac{a+d}{24c}\), \quad s(d,c)\defeq\sum_{r\Mod c}\(\(\frac rc\)\)\(\(\frac{dr}c\)\),
\end{equation}
for all $c\in \Z\setminus \{0\}$ and $\nu_\eta\(\begin{psmallmatrix}
	1&b\\0&1
\end{psmallmatrix}\)=e\(\tfrac{b}{24}\)$. Another formula \cite{Knopp70Book} for $c>0$ is
\begin{equation}\label{KnoppFmlEta}
	\nu_{\eta}(\gamma)=\left\{
	\begin{array}{ll}
		\(\tfrac dc\)e\left\{\tfrac1{24}\Big((a+d)c-bd(c^2-1)-3c\Big)\right\} & \text{if } c \text{ is odd,}\\
		\(\tfrac cd\)e\left\{\tfrac1{24}\Big((a+d)c-bd(c^2-1)+3d-3-3cd\Big)\right\} & \text{if } c \text{ is even.}
	\end{array}
	\right. 
\end{equation}
The properties $\nu_{\eta}(-\gamma)=i\nu_\eta(\gamma)$ when $c>0$ and $e(\frac{1-d}8)=(\frac 2d)\ep_d$ for odd $d$ are convenient. 

A classical Kloosterman-type sum is defined as
\begin{equation}\label{AknAsKloostermanSum}
	A_c(n)\defeq \sum_{d\Mod c^*}e^{\pi is(d,c)}e\(-\frac {dn}c\). 
\end{equation}
One can check that $A_c(n)=e(-\frac 18)S(1,1-n,c,\nu_\eta)$ (see \cite[\S 2.8]{AAimrn}) and 
\begin{equation}\label{AknAsKloostermanSum2}
(-1)^{\floor{\frac{c+1}2}}A_{2c}\(n-\tfrac{c(1+(-1)^c)}4\)=e(\tfrac18)\overline{S(0,n,2c,\psi)}
\end{equation}
where $\psi$ is defined as in \eqref{Nu2} (see \cite[Lemma~3.1]{ahlgrendunn}). For our application in the next section, one can verify that $(\frac d3)\overline{\nu_\eta}$ is a weight $\frac12$ multiplier on $\Gamma_0(3)$. The fact that $w_{\pm \frac12}(\gamma_1,\gamma_2)\in \{\pm 1\}$ is helpful for this verification.

\section{Proof of Theorem~\ref{exactFmlThmMod3}}
Now we use the theorems in Section~1 to prove Theorem \ref{exactFmlThmMod3}. We follow the outline of \cite{BrmOno2006ivt} and the idea is that $q^{-\frac1{24}}\mR(w,q)$ equals the holomorphic part of a Poincar\'e series whose Fourier coefficients can be explicitly calculated.   

The following construction can be found in \cite{BrmOno2006ivt,BringmannOno2012}. Let $k\in \frac12+\Z$, $z=x+iy$ for $x,y\in \R$ and $y\neq 0$, $s\in \C$, $4|N$ is a positive integer. 
We define the weight $k$ hyperbolic Laplacian (different from the former section) by
\[\widetilde{\Delta}_k\defeq -y^2\(\frac{\partial^2}{\partial x^2}+\frac{\partial^2}{\partial y^2}\)+iky\(\frac{\partial}{\partial x}+i\frac{\partial}{\partial y}\). \]

\begin{definition}
	With the notations above, let $\chi$ be a Dirichlet character modulo $N$. A weight $k$ harmonic Maass form on $\Gamma_0(N)$ with Nebentypus $\chi$ is any smooth function $f:\HH\rightarrow \C$ satisfies: 
	
	(1) For all $\gamma\in \Gamma_0(N)$, we have $f(\gamma z)=\chi(d)\nu_\theta(\gamma)^{2k}(cz+d)^k f(z)$; 
	
	(2) $\widetilde{\Delta}_k f=0$;
	
	(3) There exists a polynomial $\mathcal{P}(z)=\sum_{n\leq 0} a^+(n) q^n$ with coefficients in $\C$ such that \[f(z)-\mathcal{P}(z)=O(e^{-Cy})\] for some $C>0$. Analogous conditions are required for all cusps. 
\end{definition}

\begin{remark}
We denote the space of such harmonic Maass forms by $H_k(N,\chi\nu_\theta^{2k})$. The polynomial $\mathcal{P}$ is called the principal part of $f$ at the cusp $\infty$, with analogous definition at other cusps. When the transformation formula in condition (1) is replaced by $f(\gamma z)=\nu(\gamma)(cz+d)^k f(z)$ for some multiplier system $\nu$, we call $f$ a weight $k$ harmonic Maass form for $(\Gamma_0(N),\nu)$. 
\end{remark}

Denote the anti-linear differential operator $\xi_k$ by
\[(\xi_k g)(z)\defeq 2iy^k\;\overline{\frac{\partial}{\partial \overline z}\,(g(z))}=R_{-k}(y^k\overline{g(z)}) \]
where $R_k$ is the Maass raising operator defined in \eqref{RLOperators}. If we let $G(z)=g(Bz)$ for some constant B, one can check that $(\xi_k G)(z)=B^{1-k}(\xi_k g)(Bz)$. 
{Let $S_k(N,\nu)$ be the space of holomorphic cusp forms on $\Gamma_0(N)$ with multiplier system $\nu$.} The following lemma is crucial in this section:
\begin{lemma}[{\cite[Proposition~3.2]{BruinierFunke},\cite[Lemma~2.2]{BringmannOno2012}}]\label{xiOperatorMapToCuspForms}
	The map 
	\[\xi_k: H_k(N,\chi\nu_\theta^{2k})\rightarrow S_{2-k}(N,\overline{\chi}\nu_{\theta}^{-2k})\]
	is a surjective map. Moreover, if $f\in H_k(N,{\chi\nu_\theta^{2k}})$ has Fourier expansion
	\[f(z)=\sum_{n\geq n_0}c_f^+(n)q^n+\sum_{n<0}c_f^-(n)\Gamma(1-k,4\pi|n|y)q^n\quad\ \text{for\ some\ }n_0\in \Z,\]
	then 
	\[(\xi_k f)(z)=-(4\pi)^{1-k}\sum_{n=1}^\infty \overline{c_f^-(-n)} n^{1-k}q^n. \]
\end{lemma}
\begin{remark}
	We denote the holomorphic part of $f$ as 
	\[f_h(z)\defeq \sum_{n\geq n_0}c_f^+(n)q^n=\mathcal{P}(z)+\sum_{n>0}c_f^+(n)q^n\]
	and the non-holomorphic part of $f$ as 
	\[f_{nh}(z)\defeq \sum_{n<0}c_f^-(n)\Gamma(1-k,4\pi|n|y)q^n,\]
	where $\Gamma(s,\beta)$ is the incomplete Gamma function defined by
	\[\Gamma(s,\beta)=\int_\beta^\infty t^{s-1}e^{-t}dt,\quad \beta>0. \]  
	We also define a mock modular form as the holomorphic part of a harmonic Maass form. 
\end{remark}
Now we construct an example of a harmonic Maass form. Let $\nu$ be an admissible multiplier system on $\Gamma_0(N)$ where $\alpha_\nu>0$ and $B$, $M$, $\nu'$ and $D$ are as in Definition \ref{Admissibility},
\[\mathcal{M}_s(y)\defeq  |y|^{-\frac k2} M_{\frac k2 \sgn y,\ s-\frac12}(|y|)\quad\text{and}\quad  \varphi_{s,k}(z)\defeq \mathcal{M}_s(4\pi y)e(x)\]
where $M_{\alpha,\beta}$ is the standard $M$-Whittaker function. One can check that $\varphi_{s,k}(z)$ is an eigenfunction of $\widetilde{\Delta}_k$ with eigenvalue $s(1-s)+\frac{k^2-2k}4$. 
We define the Maass-Poincar\'e series by
\begin{equation}\label{MaassPoincareSeries}
	P_k(s,m,N;z)\defeq \frac1{\Gamma(2-k)}\sum_{\gamma=\begin{psmallmatrix}
			a &b\\c&d
		\end{psmallmatrix} \in \Gamma_\infty\setminus \Gamma_0(N)} \overline \nu(\gamma)(cz+d)^{-k} \varphi_{s,k}(\tilde m\gamma z). 
\end{equation}
By \cite[Lemma~3.1]{BringmannOno2012}, when $\re s>1$, the above series is absolutely and uniformly convergent. 
As in \cite[Theorem~3.2 \& Remark~(1)]{BringmannOno2012}, we have the following theorem for $P_k$ (note that we have replaced their $2-k$ by $k$): 
\begin{theorem}\label{FourierExpn_ofPoincareSeries}
	With the notation above in the section, when $k\leq -\frac12$ is half-integral and $\tilde m<0$, we have
	\[P_k(1-\tfrac k2,m,N;Bz)\in H_{k} (\Gamma_0(M),(\tfrac{|D|}\cdot)\nu_\theta^{2k} ) \]
	and
	\begin{align*}
		P_k\(1-\tfrac k2,m,N;z\)&=\frac{1-k}{\Gamma(2-k)}\big(\Gamma(1-k)-\Gamma(1-k,4\pi|\tilde m|y)\big)q^{\tilde m}\\
	&+\sum_{\tilde{n}>0} \beta(n) q^{\tilde{n}}+\sum_{ \tilde{n}<  0} \beta'(n)\Gamma\(1-k,4\pi |\tilde n| y\)q^{\tilde{n}},\\
	\end{align*}
	where
	\[\beta(n)=i^{-k} 2\pi \Big|\frac{\tilde m}{\tilde n}\Big|^{\frac{1-k}2}\sum_{N|c>0} \frac{S(m,n,c,\nu)}{c}I_{1-k}\(\frac{4\pi \sqrt{|\tilde m\tilde n|}}c\)\]
	and
	\[\beta'(n)= \frac{ i^{-k}2\pi}{\Gamma(1-k)}\Big|\frac{\tilde m}{\tilde n}\Big|^{\frac{1-k}2}\sum_{N|c>0} \frac{S(m,n,c,\nu)}{c}J_{1-k}\(\frac{4\pi \sqrt{|\tilde m\tilde n|}}c\). \]
	This theorem also holds when $k=\frac12$, provided that we ensure the convergence of the formulas for $\beta(n)$ and $\beta'(n)$. We can guarantee the convergence of these formulas for any admissible multiplier $\nu$ satisfying $\alpha_{\nu}>0$. 
\end{theorem}

We postpone the proof of Theorem~\ref{FourierExpn_ofPoincareSeries} to the last section and now prove Theorem~\ref{exactFmlThmMod3}. 
Define the theta function as
\[\theta(z;h,N)\defeq \sum_{n\equiv h\Mod N}nq^{\frac{n^2}{24}}. \]
It is well known that the above theta functions are holomorphic cusp forms of weight $\frac32$ whose transformation formulas can be computed via \cite{ShimuraHalfItglW}. Moreover, Bringmann and Ono \cite{BrmOno2010} showed that some period integral of a linear combination of such theta functions can be added as a non-holomorphic part to $q^{-\frac1{24}}\mR(w;q)$ to get a harmonic Maass form. We call that combination a shadow of $q^{-\frac1{24}}\mR(w;q)$. 

When $w\neq 1$ is a root of unity, by \cite[Theorem~7.1]{ZagierBourbakiMockTheta} we know that $q^{-\frac1{24}}\mR(w;q)$ is a mock modular form of weight $\frac12$ with shadow proportional to 
\begin{equation}\label{shadow}
	\(w^{\frac12}-w^{-\frac12}\)\sum_{n\in \Z} \(\frac{12}n\)nw^{\frac n2}q^{\frac{n^2}{24}}.
\end{equation} 
Hence the shadow of $q^{-\frac1{24}}\mR(-1;q)$ is proportional to $\theta(z;1,6)$ as \cite[Remark, p.251]{BrmOno2006ivt} and a computation shows that the shadow of $q^{-\frac1{24}}\mR(e^{\frac{2\pi i}3};q)$ is proportional to $\theta(z;1,12)+\theta(z;5,12)$. Moreover, the differential operator $\xi_k$ maps a weight $k$ harmonic Maass form to its shadow. 

We take our weight $\frac12$ multiplier system $\nu=(\frac d3)\overline{\nu_\eta}$ on $\Gamma_0(3)$ to define the Maass-Poincar\'e series $P_k(s,m,3;z)$ in \eqref{MaassPoincareSeries}. This multiplier is admissible with $B=24$ and $|D|=4$, i.e. the trivial Nebentypus. Denote 
\[P(z)\defeq P_{\frac12}(\tfrac34,0,3;z),\quad  \text{so }P(24z)\in H_{\frac12}(576,\nu_\theta) \]
and we write the Fourier expansion of $P(z)$ as in Theorem~\ref{FourierExpn_ofPoincareSeries}. 

We define $M(z)$ to be the unique harmonic Maass form such that $q^{-\frac1{24}}\mR(e^{\frac{2\pi i}3};q)$ is its holomorphic part. It follows that $M(24z)$ is a weight $\frac12$ harmonic Maass form for some $\Gamma_0(M')$ and Nebentypus $\chi'$. If we can establish the equality $M(24z)=P(24z)$, then Theorem~\ref{exactFmlThmMod3} is proved using Theorem~\ref{FourierExpn_ofPoincareSeries}. The rest of this section is devoted to proving this equality.

Decompose $P(24z)$ and $M(24z)$ into holomorphic and non-holomorphic parts as 
\begin{align}\label{PrincipalPart}
	P(24z)&=P_h(24z)+P_{nh}(24z)=q^{-1}+\sum_{n=1}^\infty \beta(n)q^{24n-1}+P_{nh}(24z),\\
	M(24z)&=M_h(24z)+M_{nh}(24z)=q^{-1}\mR(e^{\frac{2\pi i}3};q^{24})+M_{nh}(24z). 
\end{align}
\begin{lemma}\label{MOnGamma0_3}
	$M(z)$ is a weight $\frac12$ harmonic Maass form for $(\Gamma_0(3),(\frac d3)\overline{\nu_\eta})$. 
\end{lemma}


\begin{proof}

We begin by investigating the shadows.  {Recall that $S_k(N,\nu)$ is the space of holomorphic cusp forms on $\Gamma_0(N)$ with multiplier system $\nu$.} By combining Lemma~\ref{xiOperatorMapToCuspForms}, Theorem~\ref{FourierExpn_ofPoincareSeries} and the definition of $\xi_{\frac12}$, the shadow of $P$ satisfies
\begin{equation}\label{shadowOfP}
P_{sha}(z)\defeq \xi_{\frac12}(P(z))=\xi_{\frac12}(P_{nh}(z))\in S_{\frac32}(3,\,(\tfrac \cdot 3)\nu_\eta). 
\end{equation} 
Direct calculations using \eqref{KnoppFmlEta} yield $P_{sha}(3z)\in S_{\frac 32}(9,\nu_\eta^3)$.  

On the other hand, since $\xi_{\frac12}$ maps $M(z)$ to the shadow of $q^{-\frac1{24}}\mR(e^{\frac{2\pi i}3};q)$, we see that $\xi_{\frac12}(M(z))$ is proportional to
\[M_{sha}(z)\defeq\theta(z;1,12)+\theta(z;5,12)=\sum_{n=1}^\infty \chi_{-36}(n)nq^{\frac{n^2}{24}},  \]
where $\chi_{-36}$ is the Dirichlet character modulo 12 induced by $(\frac{-4}\cdot)$.
One can check that $M_{sha}(3z)=\eta(z)^3-(\eta^3|U_3V_3)(z)$ where for $f=\sum_{n=1}^\infty  a_f(n)q^{\frac n{8}}$, 
\[(f|U_3)(z)\defeq \sum_{n=1}^\infty a_f(3n)q^{\frac n{8}}=\frac13\sum_{u=0}^2f\(\frac{z+8u}3\)\text{\ \ and \ \ }(f|V_3)(z)\defeq f(3z). \]
Clearly $\eta^3\in S_{\frac32}(1,\nu_\eta^3)$. With some tedious matrix calculation, we observe
\[(\eta^3|U_3V_3)(z)=\frac13\sum_{u=0}^2\eta^3\(z+\frac{8u}3\)\in S_{\frac32}(9,\nu_\eta^3).\]
Hence $M_{sha}(3z)\in  S_{\frac 32}(9,\nu_\eta^3)$, so $P_{sha}(3z)$ and $M_{sha}(3z)$ are in the same space. 

Next we prove that $P_{sha}(3z)$ and $M_{sha}(3z)$ are proportional. One can check that
\[\eta(3z)\in S_{\frac12}(9,(\tfrac \cdot 3)\nu_\eta^3) \quad \text{and} \quad f(z)\in  S_{\frac 32}(9,\nu_\eta^3) \  \Longrightarrow\ f(z)\eta(3z)^7\in S_5(9,(\tfrac \cdot 3)).\]
Here $S_5(9,(\frac \cdot 3))$ is a two-dimensional space spanned by $q-2q^4+O(q^6)$ and $q^2+q^3+O(q^4)$. Since both $P_{sha}(3z)$ and $M_{sha}(3z)$ have Fourier expansion 
\[\sum_{n\equiv 1\Mod{24}}a_1(n)q^{\frac n8},\qquad \text{ and }\quad \eta(3z)^7=\sum_{n\equiv 7\Mod{24}} a_2(n)q^{\frac n8},\]
the Fourier expansion of $P_{sha}(3z)\eta(3z)^7$ and $M_{sha}(3z)\eta(3z)^7$ both start with $Cq+O(q^4)$ for some non-zero constant $C$ (which might be different). We get \[P_{sha}(3z)\eta(3z)^7=cM_{sha}(3z)\eta(3z)^7 \quad \Rightarrow \quad P_{sha}(z)=cM_{sha}(z) \text{\ \ and\ \ } P_{nh}(z)=cM_{nh}(z)\]
for some constant $c$. 

From \eqref{shadowOfP} we conclude that $M_{sha}(z)\in  S_{\frac 32}(3,(\frac\cdot 3){\nu_\eta})$. 
By \cite[Theorem~1.2]{BrmOno2010} we know that $M(72z)$ is a harmonic Maass form on $\Gamma_1(1728)$ and by \cite[Theorem~3.4]{BrmOno2010}, $M(z)$ is an entry of a vector-valued harmonic Maass form on $\Gamma_0(1)$. To prove Lemma~\ref{MOnGamma0_3}, it suffices to check the transformation law on $\Gamma_0(3)$ and we don't need to check for the growth rate at cusps. It is known that $\Gamma_0(3)$ can be generated by $\begin{psmallmatrix}
	1&1\\0&1
\end{psmallmatrix}$ and 
\begin{equation}\label{generatorGamma0_3}
\begin{psmallmatrix}
	-1&1\\-3&2
\end{psmallmatrix}=\begin{psmallmatrix}
	0&-1\\3&0
\end{psmallmatrix}\begin{psmallmatrix}
	1&-1\\0&1
\end{psmallmatrix}\begin{psmallmatrix}
	0&-1/3\\1&0
\end{psmallmatrix}\begin{psmallmatrix}
	1&-1\\0&1
\end{psmallmatrix}.
\end{equation} 
The transformation law of $M_h(z)=q^{-\frac 1{24}}\gamma(q)$ as a sixth-order mock theta function can be found in \cite[(4.3), (5.5), p. 122]{GordonMcIntoshSurvey}. Combining $M_h$ and $M_{nh}$ and carefully comparing the notation of Mordell integrals between \cite[p. 121]{GordonMcIntoshSurvey}
\[J(\alpha)\defeq \int_0^\infty \frac{e^{-\alpha x^2}}{\ch \alpha x} dx\]
and \cite[(2.5), Theorem~2.3, (3.2), and Lemma~3.2]{BrmOno2010}
\[J(\tfrac13;\alpha)\defeq \int_0^\infty e^{-\frac 32\alpha x^2}\;\frac{\ch \alpha x+1}{\ch (3\alpha x/2)} dx,\]
it's not hard to check that, under the transform of generators of $\Gamma_0(3)$ decomposing as \eqref{generatorGamma0_3},  
\[M(\begin{psmallmatrix}
	-1&1\\-3&2
\end{psmallmatrix}z)=e\(\tfrac 5{12}\)(2-3z)^\frac12 M(z), \quad \text{where\ }(\tfrac 23)\overline{\nu_{\eta}}\begin{psmallmatrix}
-1&1\\-3&2
\end{psmallmatrix}=e\(\tfrac 5{12}\) \]
with the help of \eqref{KnoppFmlEta}.

\end{proof}


Next we show that the principal part of $M(z)$ at the cusp 0 of $\Gamma_0(3)$ is constant. We can take the scaling matrix $\sigma_0=\begin{psmallmatrix}
	&-1\\3&
\end{psmallmatrix}$. With \cite[Theorem~2.3]{BrmOno2010} we can check the image of the holomorphic part $M_h(z)=\sin \frac\pi 3\, \mathcal{N}(\frac13;q)$ (in their notation) under the slash operator $|_{\frac12}\sigma_0$. The result $\mathcal{M}(\frac13;3z)$ has principal part 0 and the Mordell integral $\sqrt{z}J(\frac13;-6\pi iz)$ is bounded when $\im z\rightarrow \infty$. 

Since the principal part of the Poincar\'e series $P(z)$ is non-constant only at $\infty$, by \eqref{PrincipalPart}, the principal part of $E(z)\defeq P(z)-M(z)$ is constant at both cusps of $\Gamma_0(3)$. Then $P_{nh}=M_{nh}$ by \cite[Lemma~2.3]{BringmannOno2012} and $E(z)=P_h(z)-M_h(z)$ is in fact a holomorphic modular form whose Fourier coefficients are supported on $n-\frac1{24}$ for $n\geq 1$. 

According to the Serre-Stark basis theorem, the space $M_{\frac12}(576,\nu_\theta)$ consists of theta functions whose Fourier coefficients are zero except those for exponents $t\ell^2$ where $t|576$ and $\ell\in \Z$. However, $E(24z)$ is in $M_{\frac12}(576,\nu_\theta)$ and has Fourier expansion  with supported exponents of the form $24n-1$. Therefore, $E(z)=0$. 

It follows that $q^{-\frac1{24}}\mR(e^{\frac{2\pi i}3};q)=P_h(z)$ as the holomorphic part of $P_{\frac12}(\frac34,0,3;z)$ whose Fourier coefficient is shown in Theorem~\ref{FourierExpn_ofPoincareSeries}.  Note that we are in the special case $k=\frac12$ and {$\tilde m=\tilde 0=-\frac 1{24}<0$} (stated at the end of Theorem~\ref{FourierExpn_ofPoincareSeries}). This finishes the proof of Theorem~\ref{exactFmlThmMod3}.


\begin{remark}
	Denote 
	\begin{equation}\label{Nu2}
		\psi(\gamma)\defeq e(\tfrac c8)(\tfrac{-1}d)^{\frac c2+1}\overline{\nu_\eta}(\gamma)
	\end{equation}
	which is a weight $\frac12$ admissible multiplier on $\Gamma_0(2)$ as in \cite[(3.4), Lemma~3.2]{ahlgrendunn}. The shadow of $q^{-\frac1{24}}\mathcal{R}(-1;q)$ is proportional to
	\[\theta(z;1,6)=\frac{\eta(z)^5}{\eta(2z)^2}\in  S_{\frac 32}(2,\overline \psi).\]
	By \eqref{AknAsKloostermanSum2} and \cite[(3.5)]{ahlgrendunn}, we rewrite the exact formula of $A(\frac12;n)$ (originally proved in \cite{BrmOno2006ivt}) as:
	\begin{equation}\label{ExactFmlMod2_inKlstmSumCase}
	A\(\frac12;n\)=	\frac{2\pi \, e(-\frac18)}{(24n-1)^{\frac 14}}\sum_{2|c>0}\frac{S(0,n,c,\psi)}c I_{\frac12}\(\frac{\pi \sqrt{24n-1}}{6c}\).
	\end{equation}
	This can also be deduced from a similar process as our proof here using Theorem~\ref{FourierExpn_ofPoincareSeries}. 
\end{remark}

	

\section{Examples of admissible multipliers and a lower bound of the spectrum}

In the remaining part of this paper we come back to prove the theorems in Section 1. Suppose $N$ is a positive integer. In this section we are going to prove the following proposition and conclude a lower bound for the exceptional spectrum in certain cases. 
\begin{proposition}\label{Section3Prop}
	If $\nu=\chi \nu_\theta$ or $\nu=\chi \nu_\eta$ where $\chi$ is quadratic character modulo $N$, then $\nu$ and its conjugate are \it{admissible}, i.e. satisfy the requirements in Definition~\ref{Admissibility}. If $\nu$ is the multiplier for a weight $\pm\frac12$ eta-quotient, then $\nu$ satisfies the condition (1) in Definition \ref{Admissibility}.  
\end{proposition}

We will verify this proposition in the next subsection while we only prove the case for weight $\frac12$. The proof for weight $-\frac12$ case with respect to the conjugate of each multiplier follows from the same process. For simplicity we recall Dirichlet's lemma:

\begin{lemma}\label{DirichletLemma}
	Every real character $\chi$ modulo $N$ can be expressed in the form
	$\chi(y)=(\tfrac dy)$
	where $d\equiv 0,1\Mod 4$ depends on $\chi$ and $N$. Every primitive real character can be expressed in the form 
	\[\chi(y)=(\tfrac Dy),\]
	where $D$ is a fundamental discriminant and $|D|$ equals the conductor of $\chi$. 
\end{lemma}



\subsection{Proof of Proposition~\ref{Section3Prop}} Suppose $\chi$ is a quadratic character modulo $N$. If $\nu=\chi\nu_\theta$, write $\chi=(\frac{D}\cdot)\mathbf{1}_{N/D}$ where $D$ is a fundamental discriminant. Since $\nu$ is assumed to be a weight $\frac12$ multiplier, we have $\chi(-1)=1$ so $D>0$. If $D\equiv 0\Mod 4$, we are done; if $D\equiv 1\Mod 4$, then $-4D$ is fundamental and $\nu$ equals $(\frac{|-4D|}\cdot)\nu_\theta$ on $\Gamma_0(N)$. Now we have proved condition (1) for $\nu=\chi\nu_\theta$.  

The individual Weil bound is known by Blomer \cite[(2.15)]{BlomerSumofHeckeEvaluesOverQP} that for $4|N|c$, we have 
\[|S(m,n,c,\nu)|\leq \sigma_0(c)(m,n,c)^{\frac 12}c^{\frac 12}, \]
where $\nu=\psi\nu_\theta$ or $\overline{\psi\nu_\theta}$ for a Dirichlet character $\psi$ modulo $N$ satisfying $\psi(-1)=1$. This proves condition (2) for $\nu=\chi\nu_\theta$ {by $\sigma_0(c)\ll_\ep c^\ep$ and}
\begin{align*}
		{\sum_{N|c\in [y,x]}\frac{S(m,n,c,\nu)}c}
		& \ll_\ep x^{\ep}\sum_{N|c\in [y,x]}(m,n,c)^{\frac 12}c^{-\frac 12}\\
		&\ll_\ep x^{\ep} \sum_{1\leq d\leq (m,n)}d^{\frac 12}\sum_{\substack{c'\in[\frac yd,\frac xd]\\(c=c'd)}} c'^{-\frac 12} d^{-\frac 12}\\
		&\ll_\ep x^{\ep} (\sqrt x-\sqrt y)\sigma_{-\frac 12}((m,n)). 
\end{align*}

For $\nu=\chi \nu_\eta$, we have a map to $\tilde{\mathcal{L}}_\frac12 \(576N,(\tfrac{12}\cdot)\chi\nu_\theta\)$:
\begin{lemma}\label{LevelLiftzAz} For $\nu=\chi\nu_\eta$ 
and for each $r$, the map $z \rightarrow 24z$ gives an injection
	\[\tilde{\mathcal{L}}_\frac12 (N,\nu,r)\rightarrow \tilde{\mathcal{L}}_\frac12 \(576N,(\tfrac{12}\cdot)\chi\nu_\theta,r\). \]                                                                              
\end{lemma}
\begin{proof}
	The proof follows from a similar process as \cite[Lemma~3.2]{ahlgrendunn} by setting $\gamma'=\begin{psmallmatrix}
		a &24b\\c/24 & d
	\end{psmallmatrix}$ when $\gamma=\begin{psmallmatrix}
	a &b\\c & d
	\end{psmallmatrix}\in\Gamma_0(576N)$, $c>0$. For any $f\in \tilde{\mathcal{L}}_\frac12 (N,\nu,r)$, define \[g(z)\defeq f(24z)=f|_{\frac12}\(\begin{psmallmatrix}
	\sqrt{24}&0\\0&1/\sqrt{24}
	\end{psmallmatrix}\).\]
	Observe that
	\[g|_{\frac12}\gamma=f|_{\frac12}\gamma'\begin{psmallmatrix}
	\sqrt{24}&0\\0&1/\sqrt{24}
\end{psmallmatrix}. \]
One can check that $(\chi\nu_\eta)(\gamma')=\chi(d)(\frac{12}d)\nu_\theta(\gamma)$ by \eqref{KnoppFmlEta} with the help of the identities $e(\frac{1-d}8)=(\frac2d)\ep_d$ and $\ep_d^2=(\frac{-1}d)$ for odd $d$. 
\end{proof}
Now $(\tfrac{12}\cdot)\chi\nu_\theta$ is a weight $\frac12$ multiplier on $\Gamma_0(M)=\Gamma_0(576N)$ and $\chi(-1)=1$. As in the beginning of this subsection, $(\tfrac{12}\cdot)\chi\nu_\theta$ equals $(\frac{D'}\cdot)\nu_\theta$ on $\Gamma_0(M)$ for some $D'$ fundamental. Finally we pick $(\frac{D'}\cdot)\nu_\theta$ or $(\frac{|-4D'|}\cdot)\nu_\theta$ {as we discussed at the beginning of this subsection}. 

Although we only need an average bound, we have an individual Weil bound for $\nu=\chi\nu_\eta$. {The proof that $\nu$ satisfies condition (2) is then similar as the proof for $\psi\nu_\theta$ above. }

\begin{lemma}\label{Weilbd} Suppose that $\nu=\psi_q \nu_\eta$ where $\psi_q$ is a Dirichlet character modulo $q$ with $q|N|c$. Write $24m-23=\alpha^2 M_1$ with $M_1$ square-free. Then we have
	\[|S(m,n,c,\nu)|\ll q^\frac32 \sigma_0((\alpha,c))\sigma_0(c)\sqrt c \cdot\((24m-23)(24n-23),c\)^{\frac12}.\]	
\end{lemma}

\begin{proof}
	Set $r=N/q$ and $s=c/N$. By \eqref{AlphaDefine} we have $\alpha_{\nu}=\alpha_{\nu_{\eta}}$ so $\tilde{n}=n-\frac{23}{24}$. 
	We have
	\begin{align*}
		S(m,n,c,\nu)&=\sum_{d\Mod c^*} \overline\psi_q(d)\overline{\nu_\eta}(\gamma) e\(\frac{\tilde{m}a+\tilde{n}d}c\)\\
		&=\sum_{d\Mod c^*} \(\sum_{\ell=1}^q \frac{\overline\psi_q(\ell)}{q}\sum_{h=1}^q e\(\frac{h(d-\ell)}q\) \)\overline{\nu_\eta}(\gamma) e\(\frac{\tilde{m}a+\tilde{n}d}c\)\\
		&=\frac 1q\sum_{\ell=1}^q \overline\psi_q(\ell)\sum_{h=1}^q e\(-\frac{h\ell}q\) \sum_{d\Mod c^*} \overline{\nu_\eta}(\gamma) e\(\frac{\tilde{m}a+(\tilde{n}+hrs)d}c\).
	\end{align*}
The proof now follows from the Weil-type bound for $S(m,n,c,\nu_\eta)$. 	By \cite[Proposition~2.1]{AAimrn} we see that
	\[|S(m,n,c,\nu)|\ll q\sigma_0((\alpha,c))\sigma_0(c)\sqrt c \cdot \max_{1\leq h\leq q}(M_1N_2,c)^{\frac12},\]
	where $N_1=24n-23$ and $N_2=24(n+hrs)-23=24n-23+\frac{24hc}q$. We finish the proof by a rough estimate
	\[(M_1N_2,c)\leq q(M_1N_2,\tfrac cq)=q(M_1N_1,\tfrac cq)\leq q(M_1N_1,c).\]
\end{proof}


It remains to prove the claim for eta-quotients in Proposition~\ref{Section3Prop}. 
\begin{lemma}\label{LevelLiftzAz2}
	If $\nu$ is the multiplier system for an eta-quotient
	\[f(z)=\prod_{\delta|L}\eta(\delta z)^{r_\delta}\]
	of weight $\frac12=\frac12\sum_{\delta|L}r_\delta$, then the map $z \rightarrow 24z$ gives an injection
	\[\tilde{\mathcal{L}}_\frac12 (N,\nu,r)\rightarrow \tilde{\mathcal{L}}_\frac12 \(576LN,\;\prod_{\delta|L}\(\frac{12\delta}\cdot\)^{r_\delta}\nu_\theta,\;r\).\]
\end{lemma}

\begin{proof}
	The proof is similar to Lemma~\ref{LevelLiftzAz}. Let $\nu_\delta$ denote the multiplier for a factor $\eta(\delta z)$. Since for $\begin{psmallmatrix}
		a&b\\c&d
	\end{psmallmatrix}\in\Gamma_0(576\delta)$,
	\[\begin{psmallmatrix}
		24\delta&0\\0&1
	\end{psmallmatrix}\begin{psmallmatrix}
	a&b\\c&d
\end{psmallmatrix}=\begin{psmallmatrix}
a&24\delta b\\c/24\delta &d
\end{psmallmatrix}\begin{psmallmatrix}
		24\delta&0\\0&1
	\end{psmallmatrix}, \]
we have
	\begin{align*}
	\nu_\delta\(\begin{psmallmatrix}
		a&24b\\c/24&d
	\end{psmallmatrix}\)&=\nu_\eta\(\begin{psmallmatrix}
		a&24\delta b\\c/24\delta&d
	\end{psmallmatrix}\)=(\tfrac{c/24\delta}{d})e(\tfrac{d-1}8)=(\tfrac \delta d)(\tfrac{12}d)\nu_\theta\(\begin{psmallmatrix}
	a&b\\c&d
\end{psmallmatrix}\)\end{align*}
	because $e(\frac{1-d}8)=(\frac 2d)\ep_d$ when $d$ is odd. We take the product of all the factors. 
\end{proof}

\begin{remark}
For the multiplier of a eta-quotient, the author doesn't know its Weil bound in general. 
\end{remark}

\subsection{Lower bound for the exceptional spectrum}
After we get a twisted theta-multiplier by level lifting, the following theorems show the relationship of eigenvalues between weight 0 and weight $\frac12$ eigenforms.

\begin{theorem}[{\cite[p. 304]{sarnakAdditive}}] \label{speclift}
Let $\chi$ be a Dirichlet character modulo $4N'$ for a positive integer $N'$, $\nu=\nu_\theta\(\frac{N'}{\cdot}\)\overline\chi$, then for each eigenvalue $\lambda=\frac14+r^2$ of $\Delta_{\frac12}$ for $(\Gamma_0(4N'),\nu)$, there is an eigenvalue $\lambda'=\frac14+4r^2$ of $\Delta_0$ for $(\Gamma_0(2N'),\chi^2)$.
\end{theorem}

Recall the definition \eqref{FirstSpecPara} of $r_\Delta(N,\nu,k)$. We have the following bound:

\begin{proposition}\label{specParaBoundWeightHalf}
	Let $\nu$ be a weight $k=\pm\frac12$ multipler of $\Gamma=\Gamma_0(N)$ satisfying condition (1) in Definition \ref{Admissibility} and assume $H_\theta$ \eqref{HTheta}. Then we have
	\[2\im r_\Delta\(N,\nu,k\)\leq\theta.  \]
\end{proposition}
\begin{proof}
	We prove the case for $k=\frac12$ and the other case follows by conjugation. Condition (1) gives the injection 
	\[z\rightarrow Bz:\quad \tilde{\mathcal{L}}_{\frac12}\(N,\nu,r\)\rightarrow\tilde{\mathcal{L}}_{\frac12}\(M,(\tfrac{|D|}\cdot)\nu_\theta,r\)\]
	where $4|N|M$.  
	We set $\chi=(\tfrac{M}\cdot)(\tfrac{|D|}\cdot)$ and apply Theorem~\ref{speclift} to get an eigenvalue $\lambda'=\frac14+r'^2$ of $\Delta_0$ for $(\Gamma_0(\frac M2),\mathbf{1})$ with eigenparameter 
	\[r'=r_\Delta(-\tfrac M2,\mathbf{1},0)=2r_\Delta(N,\nu,\tfrac12).\] 
	Assuming $H_\theta$ \eqref{HTheta} we have $\im r'\leq \theta$ and finish the proof. 
\end{proof}

\section{Kuznetsov trace formula in the mixed-sign case}

Let $\mathfrak{a}$ be a singular cusp for the weight $k$ multiplier system $\nu$ on $\Gamma=\Gamma_0(N)$. For $\re s>1$, define the Eisenstein series associated to $\mathfrak{a}$ as in \cite{Proskurin2005,DFI12} by
\begin{equation}\label{EisensteinSeries}
	E_{\mathfrak{a}}(z,s)\defeq\sum_{\gamma\in\Gamma_{\mathfrak{a}}\setminus \Gamma }\overline{\nu(\gamma)w(\sigma_{\mathfrak{a}}^{-1},\gamma)}(\im \sigma_{\mathfrak{a}}^{-1}\gamma z)^s j(\sigma_{\mathfrak{a}}^{-1}\gamma,z)^{-k}
\end{equation}
and the Poincar\'e series for $m>0$ by
\[\U_m(z,s)\defeq\sum_{\gamma\in\Gamma_\infty\setminus \Gamma }\overline\nu(\gamma) (\im \gamma z)^s j(\gamma,z)^{-k}e( \tilde{m} \gamma z).\] 
Both of the series converge absolutely and uniformly on compact subsets of the fundamental domain $\Gamma\setminus \HH$ when $\re s>1$ and both of them are automorphic functions of weight $k$ as functions of $z$. The Eisenstein series can be meromorphically extended to the entire $s$-plane and have Fourier expansions on $s=\frac12+ir$ for $r\in\R$ (\cite[(12-14)]{Proskurin2005})
\begin{equation}
\begin{alignedat}{2} \label{FourierExpEsst}
	\Ea(x+iy,s)&=\delta_{\mathfrak{a}\infty}y^s & &+\rho_{\mathfrak{a}}(0,r)y^{1-s}+\sum_{\ell\neq 0}\rho_{\mathfrak{a}}(\ell,r)W_{\frac k2\sgn \tilde{\ell},\,-ir}(4\pi | \tilde{\ell}|y)e( \tilde{\ell} x)\\
	&=\delta_{\mathfrak{a}\infty}y^s& &+\frac{\delta_{\alpha_{\nu}0}\cdot4^{1-s}\Gamma(2s-1)}{e^{\frac{\pi ik}2}\Gamma(s+\frac k2)\Gamma(s-\frac k2)}y^{1-s}\varphi_{\mathfrak{a} 0}(s)\\
	&\ & &+\sum_{\ell\neq 0}|\tilde{\ell}|^{s-1} \frac{\pi^s W_{\frac k2\sgn \tilde{\ell},\,-ir}(4\pi | \tilde{\ell}|y)}{e^{\frac {\pi ik}2}\Gamma(s+\frac k2 \sgn\tilde{\ell})}\varphi_{\mathfrak{a}\ell}(s)e( \tilde{\ell} x),
\end{alignedat}
\end{equation}
where 
\[\varphi_{\mathfrak{a}\ell}(s)=\sum_{c>0}\frac1{c^{2s}}\sum_{\substack{0\leq d<c\\\gamma=\begin{psmallmatrix}
			*&*\\c&d
	\end{psmallmatrix}\in \sigma_{\mathfrak{a}}^{-1}\Gamma}}\overline\nu(\sigma_{\mathfrak{a}}\gamma)\overline{w_k}(\sigma_{\mathfrak{a}}^{-1},\sigma_{\mathfrak{a}}\gamma)e\Big(\frac{\tilde{\ell}d}c\Big),\quad \ell\neq 0.\]
We introduce two different notations $\rho_{\mathfrak{a}}(\ell,r)$ and $\varphi_{\mathfrak{a}\ell}(s)$ for later convenience. 
The Fourier expansion of $\Ea(z,s)$ at the cusp $\mathfrak{b}$ is denoted as \cite[(2.64)]{iwaniecTopClassicalMF} \cite[p. 1551]{ADinvariants}
\begin{equation}\label{FourierExpansionEsstSeriesDiffCusps}
	(\Ea(\cdot,s)|_k\sigma_{\mathfrak b})(z)=\delta_{\mathfrak{a}\mathfrak{b}}y^s +\rho_{\mathfrak{a}\mathfrak{b}}(0,s)y^{1-s}+\sum_{\ell\neq 0}\rho_{\mathfrak{a}\mathfrak{b}}(\ell,s)W_{\frac k2\sgn{\ell_{\mathfrak{b}}},\;\frac12-s}(4\pi |{\ell_{\mathfrak{b}}}|y)e({\ell_{\mathfrak{b}}} x), 
\end{equation} 
where $\rho_{\mathfrak{a}\mathfrak{b}}(0,s)=0$ when $n_\mathfrak{b}\neq 0$. 
The Fourier expansion of the Poincar\'e series can be given by \cite[(4.5)]{AAimrn}
\begin{equation}\label{FourierExpPcl}
	\U_m(x+iy,s)=y^s e( \tilde{m}z)+y^s\sum_{\ell\in\Z}\sum_{c>0}\frac{S(m,\ell,c,\nu)}{c^{2s}}B(c, \tilde{m}, \tilde{\ell},y,s,k)e( \tilde{\ell} x). 
\end{equation}
where
\[B(c, \tilde{m}, \tilde{\ell},y,s,k)=y\int_{-\infty}^\infty e\(\frac{-\tilde{m}}{c^2 y(u+i)}\)\(\frac{u+i}{|u+i|}\)^{-k}\frac{e(-\tilde{\ell} yu )}{y^{2s}(u^2+1)^s}du\]
When $\re s>1$, we have $\U_m(\cdot,s)\in \Lform_k(N,\nu)$. More properties of these two series can be found in \cite{Proskurin2005}.

The following notations are very important in the remaining part of this paper: 
\begin{setting}\label{conditionATDelta}
	Let $a=4\pi \sqrt{|\tilde{m}\tilde{n}|}\neq 0$ and $0<T\leq \frac x3$ with $T\asymp x^{1-\delta}$ where $\delta\in(0,\frac12)$ and finally will be chosen as $\frac 13$.
\end{setting}  
We define a family of test functions $\phi\defeq \phi_{a,x,T}$ as in \cite{SarnakTsimerman09} and \cite{AAimrn}:
\begin{setting}\label{conditionphi}
	The test function $\phi:[0,\infty)\rightarrow\R$ is four times continuously differentiable and satisfies
	\begin{enumerate} 
		\item $\phi(0)=\phi'(0)=0$, and for some $\ep>0$,
		\begin{equation*}
			 \phi^{(j)}(x)\ll_\ep x^{-2-\ep}\quad (j=0,\cdots,4)\quad  \text{as}\ x\rightarrow\infty.
		\end{equation*}
		\item $\phi(t)=1$ for $\frac{a}{2x}\leq t\leq \frac{a}{x}$. 
		\item $\phi(t)=0$ for $t\leq \frac{a}{2x+2T}$ and $t\geq \frac a{x-T}$.
		\item $\phi'(t)\ll\(\frac a{x-T}-\frac ax\)^{-1}\ll \frac{x^2}{aT}$.
		\item $\phi$ and $\phi'$ are piecewise monotone on a fixed number of intevals.  
	\end{enumerate}
\end{setting}

Using the notation 
\[\xi_{\mathfrak{a}}(r,f)\defeq\int_{\Gamma\setminus\HH}f(z)\overline{E_{\mathfrak{a}}\(z,\tfrac12+ir\)}\frac{dxdy}{y^2},\]
we have the spectral theorem: 
\begin{theorem}[{\cite[Theorem~2.1]{DFI12}}]\label{SpectralTheorem} 
	Let $\{v_j(z)\}$ be an orthonormal basis of $\LEigenform_{k}(N,\nu)$. Then, any $f\in \mathcal{B}_k(N,\nu)$ has the expansion
	\begin{align*}
		f(z)&=\sum_{j} \langle f,v_j\rangle v_j(z)+\!\!\!\sum_{\mathrm{singular\ }\mathfrak{a}}\frac1{4\pi}\int_{-\infty}^\infty\xi_{\mathfrak{a}}(r,f) {E_{\mathfrak{a}}\(z,\tfrac12+ir\)} dr                                                                                
	\end{align*}
	which converges absolutely. 
\end{theorem}
We also have Parseval's identity \cite[(27)]{Proskurin2005}: for $f_1,f_2\in\mathcal{L}_k(N,\nu)$,
\begin{equation}\label{ParsevalId}
	\langle f_1,f_2\rangle=\sum_{r_j} \langle f_1,v_j\rangle \overline{\langle f_2,v_j\rangle}+\!\!\!\sum_{\mathrm{singular\ }\mathfrak{a}}\frac1{4\pi}\int_{-\infty}^\infty\xi_{\mathfrak{a}}(r,f) \overline{\xi_{\mathfrak{a}}(r,f_2)} dr. 
\end{equation}

Define 
\begin{equation}\label{checkPhiDef}
	\check{\phi}(r)\defeq \ch\pi r\int_0^\infty K_{2ir}(u)\phi(u)\frac{du}{u}
\end{equation}
which is an even function for $r\in \R$. Here we prove a Kuznetsov trace formula in the mixed-sign case:
\begin{theorem}\label{Mixed-sign case trace formula}
	Suppose $\nu$ is a multiplier system of weight $k=\pm \frac 12$ on $\Gamma=\Gamma_0(N)$. Let $\{v_j(\cdot)\}$ be an orthonormal basis of $\LEigenform_{k}(N,\nu)$. Let $\rho_j(n)$ denote the $n$-th Fourier coefficient of $v_j(\cdot)$. For each singular cusp $\mathfrak{a}$ of $(\Gamma,\nu)$, let $\Ea(\cdot,s)$ be the associated Eisenstein series. Let $\varphi_{\mathfrak{a}n}(\frac12+ir)$ and $\rho_{\mathfrak{a}}(n,r)$ be defined as in \eqref{FourierExpEsst}. Then for $\tilde{m}>0$ and $\tilde{n}<0$ we have
	\begin{equation}\label{traceFormula}
		\frac{i^{-k}}{2}\sum_{N|c>0}\frac{S(m,n,c,\nu)}{c} \phi\(\frac{4\pi\sqrt{ \tilde{m} | \tilde{n}|}}{c}\)
		=
		4\sqrt{ \tilde{m} | \tilde{n}|} \sum_{r_j}\frac{\overline{\rho_j(m)}\rho_j(n)}{\ch \pi r_j}\check{\phi}(r_j)+\!\!\!\sum_{\mathrm{singular\ }\mathfrak{a}}\!\!\!\mathcal{E}_\mathfrak{a},
	\end{equation}
	where
	\begin{align*}
		\mathcal{E}_\mathfrak{a}&= \int_{-\infty}^\infty\(\frac{ \tilde{m}}{| \tilde{n}|}\)^{-ir}\frac{\overline{\varphi_{\mathfrak{a}m}\(\frac12+ir\)}\varphi_{\mathfrak{a}n}\(\frac12+ir\) \check{\phi}(r) }{\Gamma\(\frac12+\frac k2 -ir\)\Gamma\(\frac12-\frac k2 +ir\)\ch \pi r} dr\\
		&=4\sqrt{ \tilde{m} | \tilde{n}|}\cdot\frac{1}{4\pi}\int_{-\infty}^\infty  \overline{\rho_{\mathfrak{a}}\(m,r\)}\rho_{\mathfrak{a}}\(n,r\)\frac{ \check{\phi}(r) }{\ch \pi r} dr. 
	\end{align*}
	
\end{theorem}

\begin{remark}
The last equality follows from the following identity 
\begin{equation}\label{Eis Fourier Coeff rho and phi}
\sqrt{\frac{|\tilde{n}|}{\pi}}\,\rho_\mathfrak{a}(n,r)=
\frac{e^{-\frac{\pi i k}2} \pi^{ir} |\tilde{n}|^{ir}}{\Gamma\(\frac12+ir+\frac k2\sgn  \tilde{n}\)}\varphi_{\mathfrak{a}n}\(\frac12+ir\).
\end{equation} 
\end{remark}


\begin{proof}
	The proof follows the outline of \cite[Section~4]{AAimrn}, taking into account the contribution of the continuous spectrum. When $n\neq \tilde{n}$, i.e. $\alpha_{\nu}>0$, as in \cite[Lemma~4.2, Lemma~4.3]{AAimrn}, for $\re s_1>1$ and $\re s_2>1$ we have
\begin{align*}
	I_{m,n}&(s_1,s_2)\defeq \left\langle \U_m\(\cdot,s_1,k,\nu\),\overline{\U_{1-n}\(\cdot,s_2,-k,\overline\nu\)}\right\rangle\\
	=&2^{3-s_1-s_2}\(\frac{ \tilde{m}}{| \tilde{n}|}\)^{\frac{s_2-s_1}2}\frac{i^{-k}\pi\Gamma(s_1+s_2-1)}{\Gamma\(s_1-\frac k2\)\Gamma\(s_2+\frac k2\)} \sum_{c>0}\frac{S(m,n,c,\nu)}{c^{s_1+s_2}}K_{s_1-s_2}\(\frac{4\pi\sqrt{ \tilde{m} | \tilde{n}|}}{c}\). 
\end{align*}
Setting $s_1=\sigma+\frac{it}2$ and $s_2=\sigma-\frac{it}2$ with $\sigma>1$ gives 
\begin{align}\label{Imn1}
	\begin{split}
	I_{m,n}&\(\sigma+\tfrac{it}2,\sigma-\tfrac{it}2\)\\
	=&\(\frac{ \tilde{m}}{| \tilde{n}|}\)^{-\frac{it}2}\frac{2^{3-2\sigma}i^{-k}\pi\Gamma(2\sigma-1)}{\Gamma\(\sigma-\frac k2+\frac{it}2\)\Gamma\(\sigma+\frac k2-\frac{it}2\)} \sum_{c>0}\frac{S(m,n,c,\nu)}{c^{2\sigma}} K_{it}\(\frac{4\pi\sqrt{ \tilde{m} | \tilde{n}|}}{c}\). 
	\end{split}
\end{align}
To compute the inner product in the second way we introduce the notation 
\[\Lambda(s_1,s_2,r)=\Gamma\(s_1-\tfrac12-ir\)\Gamma\(s_1-\tfrac12+ir\)\Gamma\(s_2-\tfrac12-ir\)\Gamma\(s_2-\tfrac12+ir\). \]
One has (see also \cite[(32)]{Proskurin2005})
\begin{align*}
	\xi_{\mathfrak{a}}(r,\;\U_m(\cdot,&s,k,\nu))\\	
	=&\overline{\varphi_{\mathfrak{a}m}\(\tfrac12+ir\)}(4\pi  \tilde{m})^{1-s} \tilde{m}^{-\frac12-ir}e\(\tfrac k4\)\pi^{\frac12-ir}\frac{\Gamma(s-\frac12+ir)\Gamma(s-\frac12-ir)}{\Gamma(s-\frac k2)\Gamma(\frac12+\frac k2-ir)}
\end{align*}
and
\begin{align*}
	&\overline{\xi_{\mathfrak{a}}\(r,\;\overline{\U_{1-n}(\cdot,s,-k,\overline\nu)}\)}\\
	&\qquad\qquad\ =\varphi_{\mathfrak{a}n}\(\tfrac12+ir\)(4\pi | \tilde{n}|)^{1-s}| \tilde{n}|^{-\frac12+ir}e\(-\tfrac k4\)\pi^{\frac12+ir}\frac{\Gamma(s-\frac12+ir)\Gamma(s-\frac12-ir)}{\Gamma(s+\frac k2)\Gamma(\frac12-\frac k2+ir)}. 
\end{align*}
Applying Parseval's identity \eqref{ParsevalId} we get
\begin{align*}
	I_{m,n}(s_1,s_2)&=\frac{(4\pi)^{2-s_1-s_2} \tilde{m}^{1-s_1}| \tilde{n}|^{1-s_2}}{\Gamma\(s_1-\frac k2\)\Gamma\(s_2+\frac k2\)}  \(\sum_{r_j}\overline{\rho_j(m)}\rho_j(n)\Lambda(s_1,s_2,r_j)\right.\\
	&\left.+\sum_{\mathrm{singular\ }\mathfrak{a}} \frac1{4\sqrt{ \tilde{m}| \tilde{n}|}}\int_{-\infty}^\infty\(\frac{| \tilde{n}|}{ \tilde{m}}\)^{ir}\frac{\overline{\varphi_{\mathfrak{a}m}\(\tfrac12+ir\)}\varphi_{\mathfrak{a}n}\(\tfrac12+ir\) \Lambda\(s_1,s_2,r\)}{\Gamma\(\frac12+\frac k2 -ir\)\Gamma\(\frac12-\frac k2 +ir\)}dr\),
\end{align*}
and for $s_1=\sigma+\frac{it}2$ and $s_2=\sigma-\frac{it}2$,
\begin{align}
	\label{Imn2}
	\begin{split}
	&I_{m,n}\(\sigma+\tfrac{it}2,\sigma-\tfrac{it}2\)\\
	&=\frac{(4\pi)^{2-2\sigma}| \tilde{m}  \tilde{n}|^{1-\sigma}\(\tilde{m}/| \tilde{n}|\)^{-\frac{it}2}}{\Gamma\(\sigma-\frac k2+\frac{it}2\)\Gamma\(\sigma+\frac k2-\frac{it}2\)} \(\sum_{r_j}\overline{\rho_j(m)}\rho_j(n)\Lambda\(\sigma+\tfrac{it}2,\sigma-\tfrac{it}2,r\)\right.\\
	&\left. +\!\!\!\!\sum_{\mathrm{singular\ }\mathfrak{a}}\frac1{4\sqrt{ \tilde{m}| \tilde{n}|}} \int_{-\infty}^\infty\(\frac{| \tilde{n}|}{ \tilde{m}}\)^{ir} \frac{\overline{\varphi_{\mathfrak{a}m}\(\tfrac12+ir\)}\varphi_{\mathfrak{a}n}\(\tfrac12+ir\)\Lambda\(\sigma+\frac{it}2,\sigma-\frac{it}2,r\)}{\Gamma\(\frac12+\frac k2 -ir\)\Gamma\(\frac12-\frac k2 +ir\)}dr\). 
\end{split}
\end{align}

When $\alpha_{\nu}=0(=\alpha_{\overline\nu})$, we define
\[I'_{m,n}(s_1,s_2)\defeq \left\langle \U_m\(\cdot,s_1,k,\nu\),\overline{\U_{-n}\(\cdot,s_2,-k,\overline\nu\)}\right\rangle\]
and the same process above shows that $I'_{m,n}(s_1,s_2)$ equals the right hand side of both \eqref{Imn1} and \eqref{Imn2}. 

The expressions \eqref{Imn1} and \eqref{Imn2} are equal when $\sigma>1$ and we justify their equality when $\sigma=1$. The first expression \eqref{Imn1} involving $K_{it}$ converges absolutely uniformly for $\sigma\in[1,2]$ because of \cite[(10.45.7)]{dlmf}:
\[K_{it}(x)\ll (t\sh \pi t)^{-\frac 12}\quad \text{as\ }x\rightarrow 0\]
and condition (2) in Definition \ref{Admissibility}. By the basic inequalities
\[\left|\Gamma(\sigma-\tfrac12+iy)\right|=\frac{|\Gamma(\sigma+\frac12+iy)|}{|\sigma-\frac12+iy|}\leq 2\left|\Gamma(\sigma+\tfrac12+iy)\right|,\]
$\Lambda(\sigma+\frac{it}2,\sigma-\frac{it}2,y)>0$, and $|\rho\rho'|\leq |\rho|^2+|\rho'|^2$ for all $y\in \R$ and $\rho,\rho'\in \C$,
the second expression \eqref{Imn2} involving $\Lambda$ also converges absolutely uniformly for $\sigma\in[1,2]$ as a result of its absolute convergence for $\sigma>1$. 

With $\sigma=1$, we set \eqref{Imn1} and \eqref{Imn2} equal, cancel their common factors, multiply by
\[2\pi \sqrt{ \tilde{m} | \tilde{n}|}\cdot\frac2{\pi^2}\;t\sh \pi t \int_{0}^\infty K_{it}(u)\phi(u)\frac{du}{u^2}\]
and integrate over $t$. The following equations are helpful for getting \eqref{traceFormula}: 

\noindent (1) Kontorovich-Lebedev transform {\cite[(35)]{Proskurin1982}}
	\[\frac{2}{\pi^2}\int_0^\infty K_{it}(x)t\sh \pi t\int_{0}^\infty K_{it}(u)\phi(u)\frac{du}{u^2}\;dt=\frac{\phi(x)}x;\]
\noindent (2)  $\Lambda(1+\frac{it}2,1-\frac{it}2,r)\;{\ch \pi(\frac t2+r)\ch \pi(\frac t2-r)}={\pi^2}$;

\noindent (3) {\cite[(39)]{Proskurin1982}}
	\[\int_0^\infty \frac{t\sh \pi t}{\ch \pi(\frac t2+r)\ch \pi(\frac t2-r)}\int_{0}^\infty K_{it}(u)\phi(u)\frac{du}{u^2}\;dt=\frac{2}{\ch \pi r}\check{\phi}(r). \]
Theorem~\ref{Mixed-sign case trace formula} now follows from this integration. 

\end{proof}


\section{Estimating \texorpdfstring{$\check{\phi}$}{checkphi} for the full spectrum and a special  \texorpdfstring{${\phi}$}{phi}}

We focus on the case $k=\pm\frac 12$ and $\nu$ admissible as in Definition \ref{Admissibility}. Recall the notations $\Gamma=\Gamma_0(N)$, $a$, $T$, $\delta$, $\phi$ in Settings \ref{conditionATDelta} and \ref{conditionphi} and $\check \phi$ in \eqref{checkPhiDef}. Recall that we write an eigenvalue $\lambda$ of $\Delta_{k}$ for $(\Gamma_0(N),\nu)$ as $\lambda=\frac14+r^2$ where $r\in i(0,\frac14]\cup [0,\infty)$. Bounds are known for $\check{\phi}(r)$ when $r\geq 1$ and here we give bounds for $r\in i(0,\frac14]\cup[0,1]$. For simplicity, we omit the dependence of the implied constants on $N$, $\nu$ and $\ep$ in this section. 

\subsection{For \texorpdfstring{$r\in i(0,\frac14]$}{rini0frac14}}\label{imaginaryR}

Suppose $r=it$ for $t\in (0,\frac14]$.  By \cite[(10.27.3), \S 10.37]{dlmf}, for fixed $u$, $K_{-t}(u)=K_t(u)>0$ is increasing as a function of $t>0$. 
By \cite[(10.7.7), (10.27.8)]{dlmf} we have
\[K_{2t}(u)\ll \frac{\Gamma(2t)2^{2t}}{u^{2t}}, \quad  u\leq 1. \]
As for the discrete spectrum, there is a lower bound $\underline t$ for $t\in(0,\frac14]$, hence an upper bound for $\Gamma(2t)$ depending on $N$. Thus,
\[K_{2t}(u)\ll \frac{1}{u^{2t}}, \quad  u\leq 1. \]
By \cite[(10.25.3)]{dlmf}, we also have
\[K_{2t}(u)\ll e^{-u},\quad u>1.\]

Let $[\alpha,\beta]=\varnothing$ when $\beta<\alpha$. We get
\begin{align}\label{RisImBefore}
	\begin{split}
		\check{\phi}(it)&=\cos \pi t \int_0^\infty K_{2t}(u)\phi(u)\frac{du}{u}\\
		&\ll \int_{[\frac{3a}{8x},1]} \frac{du}{u^{1+2t}}+O\(\int_{[1,\frac{3a}{2x}]} e^{-u}du\)\\
		&\ll \(\frac {x}a\)^{2t}+O(1).
	\end{split}
\end{align}

In addition, assuming $H_\theta$ \eqref{HTheta}, by condition (1) in Definition~\ref{Admissibility} and Proposition~\ref{specParaBoundWeightHalf}, there are only two cases for $r=it$: $t=\frac 14$ or $2t\leq \theta$. In the second case, we have
\begin{equation}\label{RisIm}
	\check{\phi}(r)\ll \(\frac {x}a\)^{\theta}+O(1), \quad r\neq \frac i4.  
\end{equation}

\subsection{For real \texorpdfstring{$|r|\in[0,1)$}{rin01}}
We cite \cite[9.6.1, 9.8.5, 9.8.6]{HandbookMathFuncts} for numerical estimations of $K_0$:
\begin{enumerate}
	\item $K_0(u)>0$ for $u>0$,
	\item $K_0(u)\ll -\log \(\frac u2\)$ for $0<u\leq 2$,
	\item $K_0(u)\ll u^{-\frac 12}e^{-u }$ for $u\geq 2$.  
\end{enumerate}
Then we have 
\begin{align}\label{Ris0}
	\begin{split}
		\check{\phi}(0)&=\int_0^\infty K_{0}(u)\phi(u)\frac{du}{u}\\
		&\ll \int_{[\frac{3a}{8x},2]} -\log \(\frac u2\)\frac{du}u+\int_{[2,\frac{3a}{2x}]}u^{-\frac 32}e^{-u}du\\
		&\ll \(\log\frac{3a}{16x}\)^2+e^{-2}\ll (ax)^{\ep}.
	\end{split}
\end{align}
The last inequality is due to a positive lower bound for $a=4\pi\sqrt{|\tilde m\tilde n|}\geq 4\pi \min(\alpha_\nu,1-\alpha_{\nu})$ when $\alpha_{\nu}>0$ and $a\geq 4\pi$ when $\alpha_{\nu}=0$, as $\tilde m\tilde n \neq 0$. 

When $r\in(0,1)$, by \cite[(10.32.9)]{dlmf}
\begin{align*}
	|K_{ir}(u)|\leq \int_{0}^\infty e^{-u\ch w}dw=K_0(u).
\end{align*}
It follows from \eqref{Ris0} that
\begin{equation}\label{RLessThan1}
	\check{\phi}(r)\ll (ax)^\ep,\quad r\in[0,1). 
\end{equation}

\subsection{For \texorpdfstring{$r\geq 1$}{rgeq1}} These bounds are recorded in \cite[Theorem~5.1]{ADasymptotic}. The first bound corrects \cite[Theorem~6.1]{AAimrn} and \cite[Proposition~6.2]{ahlgrendunn} (but later estimates in their paper are not affected). 
\begin{equation}\label{Rlarge}
	\check{\phi}(r)\ll\left\{
	\begin{array}{ll}
		e^{-\frac r2}	&\text{ for\ \ }1\leq r\leq \frac{a}{8x}, \\
		r^{-1}						&\text{ for\ \ }\max\(1,\frac a{8x}\)\leq r\leq \frac{a}{x}, \\
		\min\big(r^{-\frac32},r^{-\frac52}\frac xT\big)	&\text{ for\ \ }r\geq \max\(\frac{a}{x},1\). 
	\end{array}
	\right. 
\end{equation}

\subsection{A special test function}

Here we choose a special test function $\phi$ satisfying Setting~\ref{conditionphi} to compute the terms corresponding to the exceptional spectrum $r\in i(0,\frac14]$ in Theorem~\ref{mainThm}.

For a general weight $k>0$ and $ \tilde{m}>0$, $ \tilde{n}<0$ with exceptional eigenvalue $\lambda<\frac14$, we set $\lambda=s(1-s)$ for $s\in(\frac12,1)$ and 
\[t=\im r=\sqrt{\tfrac14-\lambda}=\sqrt{\tfrac14-s(1-s)}=s-\tfrac12.\]
In \eqref{RisImBefore} the exponent is $2t=2s-1$. Let the lower bound for $t>0$ be $\underline t$ depending on $N$ and $0<T'\leq T\leq \frac x3$ be $T'\defeq Tx^{-\delta}\asymp x^{1-2\delta}$. 

\begin{setting}\label{conditionPhiNew5.4}
In addition to the requirement in Definition \ref{conditionphi}, when $\frac a{x-T}\leq 1.999$, we pick $\phi$ as a smoothed function of this piecewise linear one 

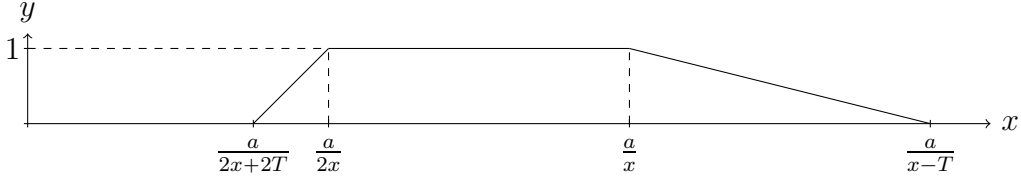
\begin{center}
	\begin{tikzpicture}
		\draw[->] (-0.05,0) -- (12.8, 0) node[right] {$x$};
		\draw[->] (0, -0.05) -- (0, 1.2) node[above] {$y$};
		\draw (3,0.05) -- (3,-0.05) node[below] {$\frac{a}{2x+2T}$};
		\draw (4,0.05) -- (4,-0.05) node[below] {$\frac{a}{2x}$};
		\draw (8,0.05) -- (8,-0.05) node[below] {$\frac{a}{x}$};
		\draw (12,0.05) -- (12,-0.05) node[below] {$\frac{a}{x-T}$};
		\draw[dashed] (0,1) -- (4,1);
		\draw (-0.05, 1) -- (0.05,1) node[left] {$1$};
		\draw[dashed] (4,0) -- (4,1);
		\draw[dashed] (8,0) -- (8,1);
		\draw[scale=2, domain=1.5:2, smooth, variable=\x, black] plot ({\x}, {\x-1.5});
		\draw[scale=2, domain=2:4, smooth, variable=\x, black] plot ({\x}, {0.5});
		\draw[scale=2, domain=4:6, smooth, variable=\x, black] plot ({\x}, {-0.25*\x+1.5});
	\end{tikzpicture}
\end{center}
where
\begin{equation}\label{PhiDerivative}
\left\{
	\begin{array}{ll}
		\phi'(u)=\tfrac{2x(x+T)}{aT}\quad  &u\in(\tfrac{a}{2x+2T-2T'},\;\tfrac{a}{2x+2T'}),\\
		\\
		\phi'(u)=-\tfrac{x(x-T)}{aT}\quad & u\in(\tfrac{a}{x-T'},\;\tfrac{a}{x-T+T'}), \\
		\\
		0\leq \phi'(u)\leq \tfrac{4x(x+T)}{aT}\quad & u\in (\tfrac{a}{2x+2T},\;\tfrac{a}{2x+2T-2T'})\cup (\tfrac{a}{2x+2T'},\;\tfrac{a}{2x}),\\
		\\
		0\geq \phi'(u)\geq -\tfrac{2x(x-T)}{aT}\quad & u\in (\tfrac{a}{x},\;\tfrac{a}{x-T'})\cup (\tfrac{a}{x-T+T'},\;\tfrac{a}{x-T}),\vspace{10px}\\
		\phi'(u)=0 \quad& \text{otherwise}.\tfrac{}{}
	\end{array}
	\right.
\end{equation}
\end{setting}
The above choice for $\phi'$ is possible because there is no requirement for $\phi''(u)$ when $u\leq 2$ but for $u\rightarrow \infty$ in Setting \ref{conditionphi}.

Derived from \cite[(10.25.2), (10.27.4), (10.37.1)]{dlmf}, for $r=it$ and $2t\in[2\underline{t},\frac12]$, we have
\[K_{2t}(u)=2^{2t-1}\Gamma(2t)u^{-2t}+O((\tfrac u2)^{2t})\quad \text{uniformly for }|u|\leq 1.999\]
 and 
\[|K_{2t}(u)|\leq |K_\frac12(u)|\ll u^{-\frac12}e^{-u}\quad\ \text{uniformly for }u\geq 1. \]
Thus, for $r=it\in i[\underline t,\frac14]$, 
\begin{align}\label{Integral1.999}
	\begin{split}
		\frac1{\cos \pi t}\check{\phi}(r)
		&=2^{2t-1}
		\Gamma(2t)\int_{0}^{1.999} \phi(u)u^{-2t}\frac{du}u+O\(\int_{0}^{1.999} \phi(u)u^{2t}\frac{du}u\)+O\(1\)\\
		&=2^{2t-1}
		\Gamma(2t)\int_{0}^{1.999} \phi(u)u^{-2t}\frac{du}u+O\(1\). \\
	\end{split}
\end{align}
\begin{lemma}\label{mainphiLemma}
	With the choice of $\phi$ in Setting \ref{conditionPhiNew5.4}, when $r=it\in i(0,\frac14]$,
	\begin{align}\label{mainphi}
		\begin{split}
			\frac1{\cos \pi t}\check{\phi}(r)
			&=2^{2t-1}\Gamma(2t)\int_{\frac{a}{2x}}^{\frac ax}u^{-2t-1}du+O\(x^{2t-\delta}a^{-2t}+1\)\\
			&=\frac{2^{2t-1}(2^{2t}-1)}{2t}\Gamma(2t) \(\frac xa\)^{2t}+O\(x^{2t-\delta}a^{-2t}+1\). 
		\end{split}
	\end{align}
\end{lemma} 
Roughly speaking, this means that the integral on $u\in(\tfrac a{2x},\tfrac ax)$ contributes the main term when $x$ is large. 

\begin{proof}[Proof of Lemma~\ref{mainphiLemma}]
	When $1.999<\frac{a}{x-T}\leq \frac{3a}{2x}$, we get $x\ll a$ and $\check{\phi}(r)=O(1)$ by \eqref{RisImBefore}, so the lemma is true in this case. It suffices to prove that when $\frac{a}{x-T}\leq  1.999$, the integral on $u<\tfrac a{2x}$ and $u>\tfrac a{x}$ is $O(x^{2t-\delta}a^{-2t})$. As $T\asymp x^{1-\delta}$ where $\delta>0$, we apply
	\[
	\(1\pm \alpha x^{-\delta}\)^{2t-1}=1\pm (2t-1)\alpha x^{-\delta}+O(x^{-2\delta})=1+O(x^{-\delta})
	\]
	where the implied constants are absolute since $2t-1\in(-1,-\frac12]$. Combining the lower bound $t\geq \underline{t}$ depending on $N$, for the left part $u\in[\frac{a}{2x+2T},\frac{a}{2x}]$, we use \eqref{PhiDerivative} and the above formula to compute the integral in \eqref{Integral1.999}:
		\begin{align*}
		\int_{\frac{a}{2x+2T}}^{\frac{a}{2x}}u^{-2t-1}\phi(u)du
		&=-\frac{2^{2t}}{2t}\(\frac xa\)^{2t}+\frac{2x(x+T)}{2taT}\int_{\frac{a}{2x+2T}}^{\frac{a}{2x}}u^{-2t}du\\
		&+O\(\frac{x^{1+\delta}}{at}\)\(\int_{\frac{a}{2x+2T}}^{\frac{a}{2x+2T-2T'}}+\int_{\frac{a}{2x+2T'}}^{\frac{a}{2x}}\)u^{-2t}du\\
		&=-\frac{2^{2t}}{2t}\(\frac xa\)^{2t}+\frac{2x(x+T)}{2taT(1-2t)}\(\frac a{2x}\)^{1-2t}\((1-2t)\frac Tx+O(x^{-2\delta})\)\\
		&+O\(\frac{x^{1+\delta}}{at}\)\(\(\frac a{2x+2T}\)^{1-2t}+\(\frac a{2x}\)^{1-2t}\)O(x^{-2\delta})\\
		&=-\frac{2^{2t}}{2t}\(\frac xa\)^{2t}+\frac{2^{2t}(1+\frac Tx)}{2t}\(\frac xa\)^{2t}+O\({x^{2t-\delta}}{a^{-2t}}\)\\
		&=O\({x^{2t-\delta}}{a^{-2t}}\).
	\end{align*}
	For $u\in[\frac{a}{x},\frac{a}{x-T}]$, a similar process gives the same conclusion. 
\end{proof}

\section{Two estimates for the coefficients of Maass forms}
Our main results depend on two estimates for the Fourier coefficients of Maass forms. These estimates were recorded in \cite[Section~4]{ahlgrendunn} but only for the coefficients of Maass cusp forms. Here we also require estimates for the coefficients of Eisenstein series. 

Recall our notations in Settings \ref{conditionATDelta} and \ref{conditionphi}.
In \cite{ADinvariants} an estimate for the coefficients of Maass cusp forms was given under the hypothesis that for some $\beta\in (\frac12,1)$,
\begin{equation}\label{AbsSumSnn}
	\sum_{N|c>0}\frac{|S(n,n,c,\nu)|}{c^{1+\beta}}\ll_{\ep,\nu} | \tilde{n}|^\ep. 
\end{equation}
Here we prove 
\begin{proposition}\label{AdsDukeEsstBd}
	Suppose that $\nu$ is a multiplier on $\Gamma=\Gamma_0(N)$ of weight $k=\pm\frac12$ which satisfies \eqref{AbsSumSnn}. Let $\rho_j(n)$ denote the Fourier coefficients of an orthonormal basis $\{v_j(\cdot)\}$ of $\LEigenform_{k}(N,\nu)$. For each singular cusp $\mathfrak{a}$ of $(\Gamma,\nu)$, let $\Ea(\cdot,s)$ be the associated Eisenstein series. Let $\varphi_{\mathfrak{a}n}(\frac12+ir)$ and $\rho_{\mathfrak{a}}(n,r)$ be defined as in \eqref{FourierExpEsst}. Then for $x>0$ we have 
	\begin{align*}
		x^{k \sgn  \tilde{n}}| \tilde{n}|\(\sum_{x\leq r_j\leq 2x}|\rho_j(n)|^2e^{-\pi r_j}+\right.  \sum_{\mathrm{singular\ }\mathfrak{a}} &\left. \int_{|r|\in[x, 2x]} |\rho_{\mathfrak{a}}(n,r)|^2 e^{-\pi |r|}dr\) \\
		& \ll_{\ep, N}x^2+| \tilde{n}|^{\beta+\ep}x^{1-2\beta}\log^\beta x.
	\end{align*} 
\end{proposition} 

\begin{remark}
Since we focus on an admissible multiplier, condition (2) in Definition~\ref{Admissibility} allows us to choose $\beta=\frac12+\ep$ when applying this proposition. 
\end{remark}

\begin{proof}
	First we suppose $\tilde{n}>0$. 
	The proof follows the same argument as \cite[Section 4]{ADinvariants}. We note that the coefficients of the Eisenstein series are not normalized correctly in their Lemma~4.2 (check with \cite[Lemma~3]{Proskurin2005}). This doesn't affect \cite[Theorem~4.1]{ADinvariants} since these terms were dropped by positivity. 
	
	With the assumptions of Proposition~\ref{AdsDukeEsstBd}, for any $t\in \R$, using \eqref{FourierExpEsst} and \eqref{Eis Fourier Coeff rho and phi}, the correct lemma is
	\begin{align*}
		\frac{2\pi^2 \tilde{n}}{|\Gamma(1-\frac k2+it)|^2}&\(\sum_{r_j}\frac{|\rho_j(n)|^2}{\ch 2\pi r_j+\ch 2\pi t}+\frac1{4\pi }\sum_{\mathrm{singular\ }\mathfrak{a}}\int_{-\infty}^\infty\frac{|\rho_{\mathfrak{a}}(n,r)|^2dr}{\ch 2\pi r+\ch 2\pi t}\)\\
		&=\frac1{4\pi}+\frac{2\tilde{n}}{i^{k+1}}\sum_{N|c>0}\frac{S(n,n,c,\nu)}{c^2}\int_L K_{2it}\(\frac{4\pi \tilde{n}}c q\)q^{k-1}dq,
	\end{align*}
	where $L$ is the semicircle contour $|q|=1$ with $\re q>0$ from $-i$ to $i$. Let $K$ be a large positive real number. Using \cite[Lemma~4.3]{ADinvariants} we get the full version of \cite[(4.3)]{ADinvariants}
	\begin{align}\label{AdsDukeEsstProof}
		\begin{split}
		\tilde{n}\(\sum_{r_j}|\rho_j(n)|^2 h_K(r_j)\right.&\left.+\sum_{\mathrm{singular}\ \mathfrak{a}}\int_{-\infty}^\infty  |\rho_{\mathfrak{a}}(n,r)|^2 h_K(r)dr\) \\
		&\ll K+\sum_{N|c>0}\frac{|S(n,n,c,\nu)|}{c}\;\Big|M_k\(K,\frac{2\pi \tilde{n}}c\)\Big|
	\end{split} 
	\end{align}
	where
	\[h_K(r)\defeq\int_{-\infty}^\infty \frac{e^{-(t/K)^2}-e^{-(2t/K)^2}}{|\Gamma(1-\frac k2+it)^2| (\ch 2\pi r+\ch 2\pi t)}dt\]
	and
	\[M(K,\alpha)=\int_{-\infty}^\infty \({e^{-(t/K)^2}-e^{-(2t/K)^2}}\)\int_{(\xi)}\frac{\sin (\pi s-\frac {\pi k}2)}{s-\frac k2}\Gamma(s+it)\Gamma(s-it)\alpha^{1-2s}dsdt.\]
	The right hand side of \eqref{AdsDukeEsstProof} is estimated in \cite[Section~4]{ADinvariants} where we get
\begin{align}\label{AdsDukeEsstProof Include h_K(r)}
\begin{split}
		\tilde{n}\(\sum_{r_j}|\rho_j(n)|^2 h_K(r_j)\right.&\left.+\sum_{\mathrm{singular}\ \mathfrak{a}}\int_{-\infty}^\infty  |\rho_{\mathfrak{a}}(n,r)|^2 h_K(r)dr\)\\
		&=O_{\ep, N}\(x+ \tilde{n}^{\beta+\ep}x^{-2\beta}\log^\beta x\).
\end{split}
	\end{align}
	Observe that 
	$h_K(r)$ is
	even as a function of $r$ and $h_x(r)\gg e^{-\pi |r|}x^{k-1}$ when $|r|\asymp x$. 	
	This proves Proposition~\ref{AdsDukeEsstBd} when $\tilde n>0$. 
	
	The $\tilde{n}<0$ case follows from conjugation by \eqref{KlstmSumConj}, \eqref{tildeNConj} and \eqref{FourierCoeffConj}, which is similar to \cite[Section~4]{ahlgrendunn}.
\end{proof}

We also require a generalization of \cite[Theorem~4.3]{ahlgrendunn} which includes the contribution from Eisenstein series.

\begin{proposition}\label{AhgDunEsstBd}
	Let $M$ be a positive integer which is a multiple of $4$. Let
	\[(k,\nu')=(\tfrac12,(\tfrac{|D|}\cdot )\nu_\theta)\text{\ \ or\ \ }(-\tfrac12,(\tfrac{-|D|}\cdot )\nu_\theta)=(-\tfrac12,(\tfrac{|D|}\cdot )\overline{\nu_\theta}), \]
	where $D$ is an even fundamental discriminant dividing $M$. Suppose that $\nu$ is a weight $k$ admissible (Definition~\ref{Admissibility}) multiplier on $\Gamma=\Gamma_0(N)$ with $M$, $D$, $\nu'$ above for some integer $B>0$. Let $\rho_j(n)$ denote the Fourier coefficients of an orthonormal basis $\{v_j(\cdot)\}$ of $\LEigenform_{k}(N,\nu)$. For each singular cusp $\mathfrak{a}$, let $\rho_{\mathfrak{a}}(n,r)$ be defined as in \eqref{FourierExpEsst} corresponding to the Eisenstein series on $(\Gamma_0(N),\nu)$. Suppose $x\geq 1$. 
	
	For $n\neq 0$ square-free or coprime to $M$ we have
	\[x^{k \sgn  {\tilde n}}|{\tilde n}|\(\sum_{|r_j|\leq x} \frac{|\rho_j(n)|^2}{\ch \pi r_j} +\sum_{\mathrm{singular\ }\mathfrak{a}}\int_{-x}^{x} \frac{|\rho_{\mathfrak{a}}(n,r)|^2}{\ch \pi r}dr\)  \ll_{\nu,\ep} | {\tilde n}|^{\frac{131}{294}+\ep} x^{3}. \]
	In general, for $ n\neq 0$ we factor $B\tilde n=t_nu_n^2w_n^2$ where $t_n$ is square-free, $u_n|M^\infty$ and ${(w_n,M)=1}$. Then we have
	\[x^{k \sgn  {\tilde n}}| {\tilde n}|\(\sum_{|r_j|\leq x} \frac{|\rho_j(n)|^2}{\ch \pi r_j} +\sum_{\mathrm{singular\ }\mathfrak{a}}\int_{-x}^{x} \frac{|\rho_{\mathfrak{a}}(n,r)|^2}{\ch \pi r}dr\)  \ll_{\nu,\ep} \( |{\tilde n}|^{\frac{131}{294}}+u_n\) x^{3}| \tilde {n}|^\ep. \]
\end{proposition}

The proof of Proposition~\ref{AhgDunEsstBd} uses Iwaniec's averaging method as in \cite{ahlgrendunn}. One important property is the relationship between the Fourier coefficients in different levels. This is not hard via the inner product for Maass cusp forms, but not clear for the continuous spectrum. Here we apply arguments in \cite{Young2019JNT} for the calculations.

For the remaining part of this section we identify the levels. Let $\langle \cdot,\cdot\rangle_{(N)}$ denote the Petersson inner product over the fundamental domain $\Gamma_0(N)\setminus \HH$. For integer $q\geq 1$, let $w_q\defeq \begin{psmallmatrix}
	\sqrt{q}&0\\0&1/\sqrt q
\end{psmallmatrix}\in \SL_2(\R)$.

	Suppose that $\nu^{({S})}$ is a weight $k=\pm\frac12$ multiplier on $\Gamma_0(S)$ and $\nu^{(T)}$ is a weight $k$ multiplier on $\Gamma_0(T)$. Suppose that there exist positive integers $q$ and $T$ such that
	\begin{equation}\label{Level lifting q S T defined in mathcal A}
		f(z)\in \mathcal{A}_k(S,\nu^{(S)})\quad \Rightarrow \quad f(qz)=(f|_kw_q)(z)\in \mathcal{A}_k(T,\nu^{(T)}).
	\end{equation}
	Note that this relation implies $qS|T$. 
	
	For $L\in \{S,T\}$, let $\rho_j^{(L)}(n)$ denote the Fourier coefficients of an orthonormal basis $\{v_j^{(L)}(\cdot)\}$ of $\LEigenform_{k}(L,\nu^{(L)})$. For each singular cusp $\mathfrak{a}$ of $(\Gamma_0(L),\nu^{(L)})$, let $\Ea^{(L)}(\cdot,s)$ be the associated Eisenstein series. Let $\varphi_{\mathfrak{a}n}^{(L)}(\frac12+ir)$ and $\rho_{\mathfrak{a}}^{(L)}(n,r)$ be defined as in \eqref{FourierExpEsst}. 

Let $\mathcal{V}_j^{(T)}(z)=v_j^{(S)}(qz)$ and $\mathcal{E}_\mathfrak{a}^{(T)}(z,s)=\Ea^{(S)}(qz,s)$. So $\mathcal{V}_j^{(T)}$ and $\mathcal{E}_\mathfrak{a}^{(T)}(\cdot,\frac12+ir)$ are eigenfunctions corresponding to the discrete and continuous spectrum of $\Delta_k$, respectively. Let $n_{(L)}\defeq n-\alpha_{\nu^{(L)}}$ for $n\in \Z$ and suppose $\alpha_{\nu^{(T)}}=0$. Then $qn_{(S)}\in \Z$ and 
\begin{equation}\label{Level lifting: relation of Fourier coeff}
	\rho_j^{(S)}(n)=\mathcal{P}_j^{(T)}(qn_{(S)}) \quad \text{and}\quad \rho_\mathfrak{a}^{(S)}(n,r)=\mathcal{P}_\mathfrak{a}^{(T)}(qn_{(S)},r),
\end{equation}
where $\mathcal{P}_j^{(T)}(n)$ and $\mathcal{P}_\mathfrak{a}^{(T)}(n,r)$ are the Fourier coefficients of $\mathcal{V}_j^{(T)}$ and $\mathcal{E}_\mathfrak{a}^{(T)}(\cdot,\frac12+ir)$ as in {\eqref{FourierExpMaassEigenform}} and \eqref{FourierExpEsst}, respectively. 

Since $d\mu(z)=\frac{dxdy}{y^2}$ is invariant under $\GL_2^+(\R)$, we can denote $I(S,T)$ as the normalizing constant such that 
\[\langle f(q\cdot),g(q\cdot)\rangle_{(T)}=I(S,T)\langle f,g\rangle_{(S)},\quad \text{for all } f,g\in\Lform_k(S,\nu^{(S)}).   \]
So $I(S,T)$ is the index $[\Gamma_0(S):\Gamma_0(T)]$. The set 
\[\left\{I(S,T)^{-\frac12}\mathcal{V}_j^{(T)}(\cdot): r_j\text{ of }\Gamma_0(S)\right\}\]
is an orthonormal subset in $\LEigenform_k(T,\nu^{(T)})$ and can be expand to a orthonormal basis of $\LEigenform_k(T,\nu^{(T)})$ as
\begin{equation}\label{Alternative basis discrete spec}
	\left\{\frac{\mathcal{V}_j^{(T)}(\cdot)}{I(S,T)^{\frac12}}: r_j\text{ of }\Gamma_0(S)\right\}\bigcup \left\{w_j^{(T)}(\cdot): r_j\text{ of }\Gamma_0(T)\right\},
\end{equation}
where each $w_j^{(T)}$ is a linear combination of $v_j^{(T)}$ from the standard basis. Let the Fourier coefficients of $w_j$ be denoted as $\rho_j^{\mathrm{comp}}(n)$, which is the corresponding linear combination of $\{\rho_j^{(T)}(n)\}_j$.

For the continuous spectrum, as in \cite[(8.1)]{Young2019JNT} we let $\mathscr{E}_r(L)$ be the finite dimensional space
\[\mathscr{E}_r(L)\defeq \spann \{\Ea^{(L)}(\cdot,\tfrac12+ir):\ \text{singular }\mathfrak{a} \text{ of } (\Gamma_0(L),\nu^{(L)})\}. \]
This $\mathscr{E}_r(L)$ is also the subspace of eigenfunctions in the continuous spectrum of $\Delta_k$ with eigenvalue $\lambda=\frac14+r^2$. We define the formal inner product
\begin{equation}
	\langle \cdot,\cdot\rangle_{(L)}^{\mathrm{Eis}}:\quad \left\langle \Ea^{(L)}(\cdot,\tfrac12+ir),\ \Eb^{(L)}(\cdot,\tfrac12+ir)\right \rangle_{(L)}^{\mathrm{Eis}}=4\pi \delta_{\mathfrak{a}\mathfrak{b}}^{(L)},
\end{equation}
where $\delta_{\mathfrak{a}\mathfrak{b}}^{(L)}=1$ if cusps $\mathfrak{a}$ and $\mathfrak{b}$ are $\Gamma_0(L)$-equivalent and  $\delta_{\mathfrak{a}\mathfrak{b}}^{(L)}=0$ otherwise. This inner product is extended sesquilinearly as a inner product on $\mathscr{E}_r(L)$, which means it is conjugate linear at the first entry and linear at the second entry.

Recall \eqref{FourierExpansionEsstSeriesDiffCusps} for the Fourier expansion of Eisenstein series at cusps. Since $\mathcal{E}_{\mathfrak{a}}^{(T)}(\cdot,\frac12+ir)=\Ea^{(S)}(q\cdot,\frac12+ir)\in \mathscr{E}_r(T)$ where $\mathfrak{a}$ is a singular cusp of $\Gamma_0(S)$, we can write 
\begin{equation}\label{expand Eq in level T}
	\mathcal{E}_{\mathfrak{a}}^{(T)}(z,\tfrac12+ir)=\sum_{\substack{\mathrm{singular\ }\mathfrak{b}\\\mathrm{of\ }\Gamma_0(T)}}c_{\mathfrak{a}}(\mathfrak{b})\Eb^{(T)}(z,\tfrac12+ir). 
\end{equation}
Let 
\[I(S,T,\mathfrak{a})\defeq \sum_{\substack{\mathrm{singular\ }\mathfrak{b}\\\mathrm{of\ }\Gamma_0(T)}} |c_{\mathfrak{a}}(\mathfrak{b})|^2,\]
then we have 
\[\left\langle \mathcal{E}_{\mathfrak{a}}^{(T)}(\cdot,\tfrac12+ir),\ \mathcal{E}_{\mathfrak{a}}^{(T)}(\cdot,\tfrac12+ir)\right \rangle_{(T)}^{\mathrm{Eis}}=4\pi I(S,T,\mathfrak{a}). \]
On the other hand, we know that $\mathcal{E}_{\mathfrak{a}}^{(T)}(\cdot,\frac12+ir)=\Ea^{(S)}(\cdot,\frac12+ir)|_kw_q$. Then the Fourier expansion of $\mathcal{E}_{\mathfrak{a}}^{(T)}(\cdot,\frac12+ir)$ at the cusp $\mathfrak{b}$ has a non-zero $y^{\frac12+ir}$ term, if, and only if, the Fourier expansion of $\Ea^{(S)}(\cdot,\frac12+ir)$ at the cusp $w_q\mathfrak{b}$ has a non-zero $y^{\frac12+ir}$ term. Since $\Eb^{(T)}(\cdot,\frac12+ir)$ only has non-zero $y^{\frac12+ir}$ term at cusps equivalent to $\mathfrak{b}$ on $\Gamma_0(T)$, we can rewrite \eqref{expand Eq in level T} as
\begin{equation}\label{Expand Eaqz on Level T}
	\mathcal{E}_{\mathfrak{a}}^{(T)}(z,\tfrac12+ir)=\sum_{\substack{\mathrm{singular\ }\mathfrak{b}\mathrm{\ of\ }\Gamma_0(T)\\w_q\mathfrak{b}\mathrm{\ equivalent\ to\ } \mathfrak{a}\mathrm{\ on\ }\Gamma_0(S)}}c_{\mathfrak{a}}(\mathfrak{b})\Eb^{(T)}(z,\tfrac12+ir). 
\end{equation}
The above sum is well defined. In fact, if two cusps $\ma_1$ and $\ma_2$ are nonequivalent on $\Gamma_0(S)$, then $w_q^{-1}\ma_1$ and $w_q^{-1}\ma_2$ are nonequivalent on $\Gamma_0(T)$. This is easily verified with \eqref{Level lifting q S T defined in mathcal A} by $qS|T$ and
\[\gamma^{(T)}\in \Gamma_0(T)\quad \Rightarrow \quad w_q\gamma^{(T)}w_q^{-1}\in \Gamma_0(S). \]
Then the sums in \eqref{Expand Eaqz on Level T} for $\mathcal{E}_{\ma_1}^{(T)}$ and $\mathcal{E}_{\ma_2}^{(T)}$ are on disjoint singular cusps of $\Gamma_0(T)$. 

Therefore, by the orthogonality in $\mathcal{E}_r(T)$ with respect to $\langle\cdot,\cdot\rangle_{(T)}^{\text{Eis}}$, 
\begin{equation}
	\left\langle \mathcal{E}_{\mathfrak{a}_1}^{(T)}(\cdot,\tfrac12+ir),\ \mathcal{E}_{\mathfrak{a}_2}^{(T)}(\cdot,\tfrac12+ir)\right \rangle_{(T)}^{\mathrm{Eis}}=4\pi I(S,T,\ma_1) \delta_{\ma_1\ma_2}^{(S)}.  
\end{equation}
Now we can expand the set 
\[\left\{I(S,T,\mathfrak{a})^{-\frac12}\mathcal{E}_{\mathfrak{a}}^{(T)}(\cdot,\tfrac12+ir): \text{ singular }\ma \text{ of }\Gamma_0(S)\right\}\]
to an orthonormal basis of $\mathcal{E}_r(T)$ with respect to $\langle\cdot,\cdot\rangle_{(T)}^{\text{Eis}}$ as
\begin{equation}\label{Alternative basis cont spec}
\left\{\frac{\mathcal{E}_{\mathfrak{a}}^{(T)}(\cdot,\tfrac12+ir)}{I(S,T,\ma)^{\frac12}}: \text{ singular }\ma \text{ of }\Gamma_0(S)\right\}\bigcup \left\{ F_{\ma}^{(T)}(\cdot,\tfrac12+ir): \text{ singular }\ma \text{ of }\Gamma_0(T)\right\},
\end{equation}
where each $F_{\ma}^{(T)}$ is some linear combination of $\Ea^{(T)}$. We denote the Fourier coefficients of $F_{\ma}^{(T)}$ as $\rho_{\mathfrak{a}}^{\mathrm{comp}}(n,r)$ and $\varphi_{\mathfrak{a}n}^{\mathrm{comp}}(\tfrac12+ir)$ as \eqref{FourierExpEsst}, which are corresponding linear combinations of $\rho_{\mathfrak{a}}^{(T)}(n,r)$ or $\varphi_{\mathfrak{a}n}^{(T)}(\tfrac12+ir)$.
 
Recall the standard expansion for $h\in \mathcal{B}_k(T,\nu^{(T)})$ \cite[Theorem~2.1]{DFI12}
\begin{align}\label{Level T decomposition original}
	\begin{split}
	h(z)&=\!\!\!\sum_{r_j\text{ of }\Gamma_0(T)}\!\!\!\langle h,v_j^{(T)}\rangle_{(T)} v_j^{(T)}+\!\!\!\sum_{\substack{\mathrm{singular\ }\mathfrak{a}\\ \mathrm{\ of\ }\Gamma_0(T)}}\!\!\!\frac1{4\pi}\int_{-\infty}^\infty \left\langle h,\Ea^{(T)}(\cdot,\tfrac12+ir)\right\rangle_{(T)} \Ea^{(T)}(\cdot,\tfrac12+ir) dr\\
	&=:h_{\mathrm D}(z)+h_{\mathrm C}(z). 
	\end{split}
\end{align}
For the discrete spectrum, we have an alternative orthonormal basis \eqref{Alternative basis discrete spec} hence another expansion for $h_{\mathrm D}(z)$. For the continuous spectrum, \cite[Proposition~8.2]{Young2019JNT} ensures that the above expansion fo $h_{\mathrm C}(z)$ is invariant with an alternative basis \eqref{Alternative basis cont spec} (where we write Young's notation $\langle F,F\rangle_{\mathrm{Eis}}=4\pi$ explicitly here to be consistent with our notations). Now we can deduce another expansion for $h$: 

\begin{align}\label{Level T Decomposition alternative}
\begin{split}
h(z)&=h_{\mathrm D}(z)+h_{\mathrm C}(z)\\
&=\!\sum_{r_j \mathrm{\ of\ }\Gamma_0(S)}\left\langle h,\frac{v_j(q\cdot)}{I(S,T)^{\frac12}}\right\rangle_{(T)}\frac{v_j(qz)}{I(S,T)^{\frac12}}
+\!\!\!\sum_{r_j \mathrm{\ of\ }\Gamma_0(T)}\left\langle h,w_j\right\rangle_{(T)}w_j(z)\\
&+\!\sum_{\substack{\mathrm{singular\ }\mathfrak{a}\\ \mathrm{\ of\ }\Gamma_0(S)}}\frac1{4\pi }\int_{-\infty}^\infty \left\langle h, \frac{E_{\mathfrak{a}}^{(S)}\(q\cdot, \frac12+ir\)}{I(S,T,\ma)^{\frac12}}\right\rangle_{(T)} \frac{E_{\mathfrak{a}}^{(S)}\(qz, \frac12+ir\)}{I(S,T,\ma)^{\frac12}}dr\\
&+\!\sum_{\substack{\mathrm{singular\ }\mathfrak{a}\\ \mathrm{\ of\ }\Gamma_0(T)}}\frac1{4\pi }\int_{-\infty}^\infty \left\langle h, F_{\mathfrak a}^{(T)}(\cdot,\tfrac12+ ir)\right\rangle_{(T)} \;F_{\mathfrak a}^{(T)}(z,\tfrac12+ ir)dr. 
\end{split}
\end{align}

We now show that $I(S,T,\ma)=I(S,T)$. Let $h\in \mathcal{B}_{k}(S,\nu^{(S)})$ be orthogonal to the discrete spectrum, i.e. $h_{\mathrm D}=0$. The standard spectral expansion of $h$ at level $S$ gives
\[h(z)=h_{\mathrm C}(z)=\sum_{\substack{\mathrm{singular\ }\mathfrak{a}\\ \mathrm{\ of\ }\Gamma_0(S)}}\frac1{4\pi }\int_{-\infty }^{\infty}\left\langle h, E_{\mathfrak{a}}^{(S)}\(\cdot, \tfrac12+ir\)\right\rangle_{(S)} E_{\mathfrak{a}}^{(S)}\(z, \tfrac12+ir\)dr.  \]
Especially, 
\begin{equation}\label{Expand h(qz) to prove I(S,T,a)=I(S,T)}
	h(qz)=\sum_{\substack{\mathrm{singular\ }\mathfrak{a}\\ \mathrm{\ of\ }\Gamma_0(S)}}\frac1{4\pi }\int_{-\infty }^{\infty}\left\langle h, E_{\mathfrak{a}}^{(S)}\(\cdot, \tfrac12+ir\)\right\rangle_{(S)} E_{\mathfrak{a}}^{(S)}\(qz, \tfrac12+ir\)dr.  
\end{equation}
However, $H(\cdot)=h(q\cdot)$ is in $\mathcal{B}_{k}(T,\nu^{(T)})$ (and is still orthogonal to the discrete spectrum) with spectral expansion $H_{\mathrm C}(z)$ as \eqref{Level T Decomposition alternative}. As we have shown the orthogonality of \eqref{Alternative basis cont spec} in $\mathscr{E}_r(T)$ under $\langle\cdot,\cdot\rangle_{\mathrm{Eis}}^{(T)}$, the spectral expansion of $H(\cdot)$ has to be unique on a subset of the basis \eqref{Alternative basis cont spec}. Comparing \eqref{Expand h(qz) to prove I(S,T,a)=I(S,T)} and \eqref{Level T Decomposition alternative} we have
\[\left\langle h(q\cdot), \frac{E_{\mathfrak{a}}^{(S)}\(q\cdot, \frac12+ir\)}{I(S,T,\ma)}\right\rangle_{(T)}=\langle h, E_{\mathfrak{a}}^{(S)}\(\cdot, \tfrac12+ir\)\rangle_{(S)}\quad \Rightarrow\quad I(S,T,\ma)=I(S,T).  \]

Now we are ready to state the formula connecting the Fourier coefficients of different level eigenforms. 
For $\sigma=\re s>1$, $t\in \R$ and $n>0$, we compute the inner product
\begin{equation}\label{product of two Poincare series}
	\left\langle U_{qn_{(S)}}^{(T)}(\cdot,\sigma+\tfrac{it}2), U_{qn_{(S)}}^{(T)}\(\cdot,\sigma-\tfrac{it}2\) \right \rangle_{(T)} 
\end{equation}
as \cite[Lemma~2]{Proskurin2005}. The results are the same if we apply the above decomposition \eqref{Level T Decomposition alternative} for $h(z)=U_{qn_{(S)}}^{(T)}(z,\sigma\pm\tfrac{it}2)$ and if we apply the standard spectral decomposition \eqref{Level T decomposition original}. Recall the notation
\begin{align*}
\Lambda(\sigma+it,\sigma-it,r)=\la\Gamma\(\sigma-\tfrac12+i(t+r)\)\ra^2\,\la\Gamma\(\sigma-\tfrac12+i(t-r)\)\ra^2.
\end{align*}

For $\sigma>1$, we can get two results of \eqref{product of two Poincare series}, as on the right hand side of \cite[(30)]{Proskurin2005}, by the two different expansions mentioned above. The following equation is the identity between such two results, where we recall \eqref{Level lifting: relation of Fourier coeff} for the relation in Fourier coefficients: 
\begin{align}\label{Level Shifting before with phi and Lambda}
\begin{split}
	\sum_{r_j \text{ of } \Gamma_0(S)}&I(S,T)^{-1}
	|\rho_j^{(S)}(n)|^2 \Lambda\(\sigma+\tfrac{it}2,\sigma-\tfrac{it}2,r_j\)\\
	+\sum_{\substack{\mathrm{singular\ }\mathfrak{a}\\\mathrm{of\ }\Gamma_0(S)}} &I(S,T)^{-1}
	\int_{-\infty}^\infty \frac{|\varphi_{\mathfrak{a}n}^{(S)}(\tfrac12+ir)|^2\Lambda\(\sigma+\frac{it}2,\sigma-\frac{it}2,r\)dr}{4n\,\Gamma\(\frac12+\frac k2 -ir\)\Gamma\(\frac12+\frac k2 -	ir\)}\\
	&+\sum_{r_j \text{ of } \Gamma_0(T)}
	|\rho_j^{\mathrm{comp}}(qn_{(S)})|^2 \Lambda\(\sigma+\tfrac{it}2,\sigma-\tfrac{it}2,r_j\)\\
	&+\sum_{\substack{\mathrm{singular\ }\mathfrak{a}\\\mathrm{of\ }\Gamma_0(T)}} 
	\int_{-\infty}^\infty \frac{|\varphi_{\mathfrak{a},qn_{(S)}}^{\mathrm{comp}}(\tfrac12+ir)|^2\Lambda\(\sigma+\frac{it}2,\sigma-\frac{it}2,r\)dr}{4qn_{(S)}\,\Gamma\(\frac12+\frac k2 -ir\)\Gamma\(\frac12+\frac k2 -	ir\)}\\
	= \sum_{r_j \text{ of } \Gamma_0(T)}& |\rho_j^{(T)}(qn_{(S)})|^2\Lambda\(\sigma+\tfrac{it}2,\sigma-\tfrac{it}2,r_j\)\\
	&+\sum_{\substack{\mathrm{singular\ }\mathfrak{a}\\\mathrm{of\ }\Gamma_0(T)}} 
	\int_{-\infty}^\infty \frac{|\varphi_{\mathfrak{a},qn_{(S)}}^{(T)}(\tfrac12+ir)|^2\Lambda\(\sigma+\frac{it}2,\sigma-\frac{it}2,r\)dr}{4qn_{(S)}\,\Gamma\(\frac12+\frac k2 -ir\)\Gamma\(\frac12+\frac k2 -ir\)}. 
	\end{split}
\end{align}
With the help of \eqref{Eis Fourier Coeff rho and phi} on the notations, we have proved a lemma regarding the shifting of levels: 
\begin{lemma}\label{Level Shifting lemma with ch 2 pi t + ch 2 pi r }
Given $T>S>0$ and $S|T$, with the notations above, for all $t\in \R$ and $\sigma>1$ we have \eqref{Level Shifting before with phi and Lambda}. 
In addition, with \eqref{Eis Fourier Coeff rho and phi} we can also write the terms involving $\varphi_{\mathfrak{a}\ell}(\frac12+ir)$ as those with $\rho_{\mathfrak{a}}(\ell,r)$ and we omit the duplicated formula here. 
\end{lemma}

\subsection{Proof of Proposition~\ref{AhgDunEsstBd}}
	We still use superscript $\cdot^{(N)}$ to identify the level and should be careful on it. The notations $\rho_j$ and $\rho_{\mathfrak{a}}$ in the statement of Proposition~\ref{AhgDunEsstBd} are on level $N$ and among the proof we utilize two more different levels. When $B\tilde n$ is square-free, \cite[Theorem~4.3]{ahlgrendunn} gives the bound 
	\begin{equation}\label{originalBoundinAhgDun}
		x^{k \sgn  \tilde {n}}|  \tilde {n}|\sum_{|r_j|\leq x} \frac{|\rho_j^{(N)}(n)|^2}{\ch \pi r_j}\ll_{\nu,\ep} |\tilde {n}|^{\frac{131}{294}+\ep} x^{3}. 
	\end{equation}
Our proposition generalizes the above bound. It suffices to prove the general case involving $u_n$ with $\tilde n>0$ and $k=\pm \frac12$, where the $\tilde n<0$ case follows from conjugation by \eqref{tildeNConj} and \eqref{FourierCoeffConj}.

Following the notation in \cite[\S 5.2]{ahlgrendunn}, we can take the fundamental discriminant $D$ to be even and $M \equiv 0\Mod 8$ as a positive integer with $D | M$. 
	Let $P$ be a positive parameter (chosen later to be $n^{\frac 17}$) and 
	\[\mathcal{Q}=\mathcal{Q}(n,M,P)\defeq\{pM:\ p\text{ prime},\ P<p\leq 2P,\ \text{and } p\nmid 2nM\}. \]
	We take any $pM$ in $\mathcal{Q}$. In \cite[p.1698]{ahlgrendunn}, they require the property that when $\{v_j^{(M)}\}$ is an orthonormal subset of $\LEigenform_{k}(M,\nu')$, then $\{[\Gamma_0(M):\Gamma_0(pM)]^{-\frac12}v_j^{(M)}\}$ is an orthonormal subset of $\LEigenform_{k}(pM,\nu')$. This is easily verified by the inner product of Maass cusp forms, but we cannot take the inner product of two Eisenstein series. We will use the discussion above in this section, especially Lemma~\ref{Level Shifting lemma with ch 2 pi t + ch 2 pi r }, to interpret the estimates between level $M$ and level $pM$ involving Eisenstein series in detail. 
	
	The following lines sketch the proof in \cite[Section~5]{ahlgrendunn}. Let 
	\[\Phi(u)\defeq \frac18\sqrt{\frac \pi 2}u^{-\frac12}J_{\frac 92}(u),\quad u\geq 0.\]
	where $J_s$ is the $J$-Bessel function. We have $\Phi(0)=\Phi'(0)=0$. For $s\in \C$, define
	\[\widetilde{\Phi}(s)\defeq \int_0^\infty J_s(u)\Phi(u)\frac{du}u\]
	and
	\[\widehat{\Phi}(r)\defeq \frac{i|\Gamma(\frac{1+k}2+ir)|^2}{2\pi^2\sh \pi r}\(\widetilde{\Phi}(2ir)\cos \pi(\tfrac k2+ir)-\widetilde{\Phi}(-2ir)\cos \pi\(\tfrac k2-ir\) \). \]
	As in \cite[above (5.13)]{ahlgrendunn}, $\widehat{\Phi}(r)>0$ for $r\in \R\cup i(0,\frac14]$. 
	At level $L=M$ or $pM$, define 
	\[\Lform_{\widehat{\Phi}}^{(L)}(n,n)\defeq 4\pi |n|\sum_{j} |\rho_{j}^{(L)}(n)|^2\frac{\widehat{\Phi}(r_j)}{\ch \pi r_j}\]
	where the sum runs over the discrete spectrum of $\Delta_k$ on $\Gamma_0(L)$
	and
	\[\Mform_{\widehat{\Phi}}^{(L)}(n,n)\defeq 4\pi |n|\sum_{\mathfrak{a}} \frac1{4\pi}\int_{-\infty}^\infty |\rho_{\mathfrak{a}}^{(L)}(n,t)|^2\frac{\widehat{\Phi}(t)}{\ch \pi t}dt\]
	where the sum runs over singular cusps of $\Gamma_0(M)$. At level $L=N$, we define $\Lform_{\widehat{\Phi}}^{(N)}$ and $\Mform_{\widehat{\Phi}}^{(N)}$ with $|n|$ changed to $|\tilde n|$ because $\alpha_{\nu}$ might be non-zero. Equation \cite[before (5.14)]{ahlgrendunn} (also \cite[Theorem~2.5]{DFI12} as the original reference)
	\begin{equation} \label{5.14AhgDun}
		\Lform_{\widehat{\Phi}}^{(pM)}(n,n)+\Mform_{\widehat{\Phi}}^{(pM)}(n,n)=e(-\tfrac k4)\mathcal K_{{\Phi}}^{(pM)}(n,n)-\mathcal N_{\widecheck{\Phi}}^{(pM)}(n,n)
	\end{equation}
	was used to conclude 
	\[\Lform_{\widehat{\Phi}}^{(pM)}(n,n)\leq e(-\tfrac k4)\mathcal K_{{\Phi}}^{(pM)}(n,n)-\mathcal N_{\widecheck{\Phi}}^{(pM)}(n,n)\]
	by dropping the positive term $\Mform_{\widehat{\Phi}}^{(pM)}(n,n)$ and \cite[Theorem~4.3]{ahlgrendunn} was proved by estimating the average of the right hand side. Here we must retain this term. Recall the index $[\Gamma_0(M):\Gamma_0(pM)]\leq p+1\ll P$: 
\begin{proposition}\label{EstmtEisensteinMincluded}
	With the notations above in this subsection, for $pM\in \mathcal{Q}$ we have
	\begin{equation}\label{EisensteinSeriesOrthogonality}
		\Lform_{\widehat{\Phi}}^{(pM)}(n,n)+\Mform_{\widehat{\Phi}}^{(pM)}(n,n)\geq  \frac{\Lform_{\widehat{\Phi}}^{(M)}(n,n)+\Mform_{\widehat{\Phi}}^{(M)}(n,n)}{[\Gamma_{0}(M):\Gamma_0(pM)]}\gg \frac{\Lform_{\widehat{\Phi}}^{(M)}(n,n)+\Mform_{\widehat{\Phi}}^{(M)}(n,n)}{P}. 
	\end{equation}
\end{proposition}

\begin{proof}[Proof of Proposition~\ref{EstmtEisensteinMincluded}]

First we apply \eqref{Level Shifting before with phi and Lambda} in Lemma~\ref{Level Shifting lemma with ch 2 pi t + ch 2 pi r } with levels $M$ and $pM$. Here we take $q=1$, $\nu^{(S)}=\nu^{(T)}=(\frac{|D|}\cdot)\nu_\theta^{2k}$ for $k=\pm \frac12$ and $I(M,pM)=[\Gamma_0(M):\Gamma_0(pM)]$ to get
\begin{align}\label{Level Shifting before with phi and Lambda, M to pM}
	\begin{split}
		\frac1{[\Gamma_0(M):\Gamma_0(pM)]}\Bigg(&\sum_{r_j \text{ of } \Gamma_0(M)}
		|\rho_j^{(M)}(n)|^2 \Lambda\(\sigma+\tfrac{it}2,\sigma-\tfrac{it}2,r_j\)\\
		+&\sum_{\substack{\mathrm{singular\ }\mathfrak{a}\\\mathrm{of\ }\Gamma_0(M)}} 
		\int_{-\infty}^\infty \frac{|\varphi_{\mathfrak{a}n}^{(M)}(\tfrac12+ir)|^2\Lambda\(\sigma+\frac{it}2,\sigma-\frac{it}2,r\)dr}{4n\,\Gamma\(\frac12+\frac k2 -ir\)\Gamma\(\frac12+\frac k2 -	ir\)}\Bigg)\\
		+&\sum_{r_j \text{ of } \Gamma_0(pM)}
		|\rho_j^{\mathrm{comp}}(n)|^2 \Lambda\(\sigma+\tfrac{it}2,\sigma-\tfrac{it}2,r_j\)\\
		+&\sum_{\substack{\mathrm{singular\ }\mathfrak{a}\\\mathrm{of\ }\Gamma_0(pM)}} 
		\int_{-\infty}^\infty \frac{|\varphi_{\mathfrak{a}n}^{\mathrm{comp}}(\tfrac12+ir)|^2\Lambda\(\sigma+\frac{it}2,\sigma-\frac{it}2,r\)dr}{4n\,\Gamma\(\frac12+\frac k2 -ir\)\Gamma\(\frac12+\frac k2 -	ir\)}\\
		= \sum_{r_j \text{ of } \Gamma_0(pM)} |\rho_j^{(pM)}&(n)|^2\Lambda\(\sigma+\tfrac{it}2,\sigma-\tfrac{it}2,r_j\)\\
		+&\sum_{\substack{\mathrm{singular\ }\mathfrak{a}\\\mathrm{of\ }\Gamma_0(pM)}} 
		\int_{-\infty}^\infty \frac{|\varphi_{\mathfrak{a}n}^{(pM)}(\tfrac12+ir)|^2\Lambda\(\sigma+\frac{it}2,\sigma-\frac{it}2,r\)dr}{4n\,\Gamma\(\frac12+\frac k2 -ir\)\Gamma\(\frac12+\frac k2 -ir\)}. 
	\end{split}
\end{align}
Following Proskurin, we multiply a function of $t$ defined by \cite[(53)]{Proskurin2005} on both sides of the above formula, integrate $t$ from $0$ to $\infty$, and pass to the limit $\sigma\rightarrow 1^+$. In addition we take the test function $\varphi$ in \cite[(34)]{Proskurin2005} to be our $\Phi$ here. What we get simplifies to (see \cite[(83)]{Proskurin2005})
\begin{align*}
	\frac{1}{[\Gamma_{0}(M):\Gamma_0(pM)]}\Bigg(&\Lform_{\widehat{\Phi}}^{(M)}(n,n)+\Mform_{\widehat{\Phi}}^{(M)}(n,n)+4\pi n\!\!\!\!\!\!\sum_{r_j\text{ of }\Gamma_0(pM)} \!\!\!\!\!\!|b_{j}^{\mathrm{comp}}(n)|^2\frac{\widehat{\Phi}(r_j)}{\ch \pi r_j}\\
	&+\;	n\sum_{\mathfrak{a}\text{ of }\Gamma_0(pM)}\int_{-\infty}^\infty |b_{\mathfrak{a}}^{\mathrm{comp}}(n,r)|^2\frac{\widehat{\Phi}(r)}{\ch \pi r}dr\Bigg)\\
	&=\Lform_{\widehat{\Phi}}^{(pM)}(n,n)+\Mform_{\widehat{\Phi}}^{(pM)}(n,n). 
\end{align*}
Our notations are consistent with $\varphi_{\mathfrak{a}\ell}$ in \cite{Proskurin2005}, $b_\mathfrak{a}(\ell,r)$ in \cite{DFI12}, and $\widehat{\cdot}$ in both the references. Since \cite[below (5.13)]{ahlgrendunn}
\[\widehat{\Phi}(r)>0\quad \text{for}\quad r\in \R\cup i(0,\tfrac14],\]
we can drop the extra terms with superscript “comp” by positivity to get the desired inequality. 
\end{proof} 

With Proposition~\ref{EstmtEisensteinMincluded}, since $|\mathcal{Q}|\asymp\frac{P}{\log P}$, summing \eqref{5.14AhgDun} over $\mathcal{Q}$ gives
\begin{equation}\label{EstmtRightHandSideKN}
	\frac1{\log P}\(\Lform_{\widehat{\Phi}}^{(M)}(n,n)+\Mform_{\widehat{\Phi}}^{(M)}(n,n)\)\ll \sum_{pM\in \mathcal{Q}}\Big|\mathcal K_{{\Phi}}^{(pM)}(n,n)\Big|+\sum_{pM\in \mathcal{Q}}\Big|\mathcal N_{\widecheck{\Phi}}^{(pM)}(n,n)\Big|. 
\end{equation}
When $n$ is square-free, it was shown in \cite[\S 5.3-5.5]{ahlgrendunn} that the right hand side of \eqref{EstmtRightHandSideKN} is bounded by $O({n}^{\frac{131}{294}+\ep}x^3)$, where $P=n^{\frac 17}\Rightarrow \log P\ll n^\ep$.

Next, we will prove the bound on the right hand side of \eqref{EstmtRightHandSideKN} when $n$ is not square-free. The estimates in \cite[\S 5.3-5.4]{ahlgrendunn} depend on their Proposition 5.2, which is the only place that requires $n$ to be square-free. That proposition is a special case of \cite[(19)]{Waibel2017FourierCO}, so we apply the general estimate from Waibel's paper here. For $\mu\in\{-1,0,1\}$, $n\in \mathbb N$ and $x\geq 1$, define
	\[K_{\mu}^{(N)}(n,x)\defeq\sum_{N|c\leq x}\frac{S(n,n,c,\nu)}ce\(\frac{2\mu n}c\). \]
	\vspace{-10px}
\begin{proposition}[{\cite[(19)]{Waibel2017FourierCO}}]
	Suppose that $N\equiv 0\Mod 8$, that $\mu\in\{-1,0,1\}$, that $n>0$ is factorized as	$n=tu^2w^2$ where $t$ is square-free, $u|N^\infty$ and $(w,N)=1$, then 
	\[\sum_{Q\in \mathcal Q}|K_{\mu}^{(Q)}(n,x)|\ll_{N,\ep}\(xP^{-\frac12}+xun^{-\frac12}+(x+n)^{\frac 58}\(x^{\frac14}P^{\frac 38}+n^{\frac 18}x^{\frac 18}P^{\frac 14}\)\)(nx)^\ep. \]
\end{proposition}
Note that in the proof, Waibel chose $P$ to be $n^{\frac17}$. By using the above proposition in each place of \cite[\S 5.3-5.4]{ahlgrendunn} where \cite[Proposition~5.2]{ahlgrendunn} was applied, we obtain new estimates that are recorded here:
	\begin{flalign*}
		&\text{\cite[(5.19)]{ahlgrendunn}}& &\ll \(\ell^{-\frac12}n^{\frac 37}+\ell^{-\frac14}n^{\frac{23}{56}}+\ell^{-2}u\)(\ell n)^\ep.&& \\
		&\text{\cite[(5.22)]{ahlgrendunn}}& &\ll \(\ell^{\frac{11}6}n^{\frac 37}+\ell^{\frac{25}{12}}n^{\frac{23}{56}}+\ell^{\frac13} u\)(\ell n)^\ep.&&\\
		&\text{\cite[(5.24)]{ahlgrendunn}}& &\ll n^{\frac 37+\frac 56\beta+\ep}+n^{\frac{23}{56}+\frac{13}{12}\beta+\ep}+ u n^{-\frac 23\beta+\ep}.&&\\
		&\text{\cite[after balancing (5.26)]{ahlgrendunn}}& &\ll n^{\frac{137}{294}+\ep}+ u n^{\frac1{147}+\ep}. &&\\
		&\text{\cite[(5.28)]{ahlgrendunn}}&\!\!\!\!\!\!\!\!\!\!\!\!\!\!\!\!\!\!\!\!\!\!\!\!\!\!\!\!\!\!\!\!\sum_{pM\in \mathcal{Q}}\Big|\mathcal K_{{\Phi}}^{(pM)}(n,n)\Big|&\ll n^{\frac{131}{294}+\ep}+ u n^{-\frac2{147}+\ep}. &&\\
		&\text{\cite[(5.29)]{ahlgrendunn}}&\!\!\!\!\!\!\!\! \!\!\!\!\!\!\!\!\!\!\!\!\!\!\!\!\!\!\!\!\!\!\!\!\sum_{pM\in \mathcal{Q}}\Big|\mathcal N_{\widecheck{\Phi}}^{(pM)}(n,n)\Big| &\ll n^{\frac{3}7+\ep}+ u n^{\ep}. &&
	\end{flalign*}
Based on the last two estimates and \eqref{EstmtRightHandSideKN}, we derive
\begin{equation}\label{Estmt Right Hand Side K+N new bound include u}
	\Lform_{\widehat{\Phi}}^{(M)}(n,n)+\Mform_{\widehat{\Phi}}^{(M)}(n,n)\ll (n^{\frac{131}{294}}+u)n^\ep.
\end{equation}

Finally we transfer the bound to level $N$. Apply Lemma~\ref{Level Shifting lemma with ch 2 pi t + ch 2 pi r } again with level $N$ and level $M$, where we have $\nu^{(N)}=\nu$,  $\nu^{(M)}=\nu'=(\frac{|D|}\cdot)\nu_\theta^{2k}$, $q=B$ and $qn_{(N)}=B\tilde n$. For $\ell\in \{m,n\}$, we factor $|B\tilde \ell|=t_\ell u_\ell^2m_\ell^2$ in the statement of Proposition~\ref{AhgDunEsstBd}. Here
\[\rho_{j}^{(N)}(n)=\rho_{j}^{(M)}(B\tilde n)\quad \text{and}\quad \rho_{\mathfrak{a}}^{(N)}(n,r)=\rho_{\mathfrak{a}}^{(M)}(B\tilde n,r)\]
for $r_j$ a spectral parameter of $\Delta_k$ on $\Gamma_0(N)$ and $\mathfrak{a}$ a singular cusp of $(\Gamma_0(N),\nu)$. As in the proof of Proposition~\ref{EstmtEisensteinMincluded} above, we integrate \eqref{Level Shifting before with phi and Lambda} to a result involving $\widehat{\Phi}$, drop the extra terms as $\widehat{\Phi}(r)>0$ for $r\in \R\cup i(0,\frac14]$, and get
\begin{align}
	\begin{split}
		&|\tilde n|\Bigg(\sum_{r_j\text{ of }\Gamma_0(N)} \frac{|\rho_j^{(N)}(n)|^2}{\ch \pi r_j} \widehat{\Phi}(r_j) + \sum_{\substack{\mathrm{singular\ }\mathfrak{a}\\\mathrm{of\ }\Gamma_0(N)}}\frac1{4\pi}\int_{-\infty}^{\infty} \frac{|\rho_{\mathfrak{a}}^{(N)}(n,r)|^2}{\ch \pi r}\widehat{\Phi}(r)dr\Bigg)\\
	 &=\Lform_{\widehat{\Phi}}^{(N)}(n,n)+\Mform_{\widehat{\Phi}}^{(N)}(n,n)\ll_\nu\Lform_{\widehat{\Phi}}^{(M)}(B\tilde n,B\tilde n)+\Mform_{\widehat{\Phi}}^{(M)}(B\tilde n,B\tilde n)\\
		&\ll_{\nu,\ep}\(|B\tilde{n}|^{\frac{131}{294}}+u_n \) |B\tilde n|^\ep\\
		&\ll_{\nu,\ep}\(|\tilde{n}|^{\frac{131}{294}}+u_n\)|\tilde n|^\ep. 
	\end{split}
	\end{align}

Following from the same argument as \cite[\S 5.5, (5.31-33)]{ahlgrendunn}, when $x\geq 1$, $k=\pm\frac12$ and $\tilde n>0$ we have
\[\widehat{\Phi}(r)^{-1}\ll x^{3-k}\quad {\text{for\ }} |r|\leq  x \]
and get Proposition~\ref{AhgDunEsstBd}. When $n<0$ it follows from the relationship \eqref{tildeNConj} and \eqref{FourierCoeffConj}. 

\section{Proof of theorems on sums of Kloosterman sums}

\subsection{Proof of Theorems~\ref{mainThm}-\ref{mainThm3}}

We prove Theorem~\ref{mainThm3} first. The other two are deduced from Theorem~\ref{mainThm3} and their proofs are given at the end of this subsection. 
For simplicity let 
\[A(m,n)\defeq\(\tilde{m}^{\frac{131}{294}}+u_m\)^{\frac12} \(|\tilde{n}|^{\frac{131}{294}}+u_n\)^{\frac12}\ll| \tilde{m}  \tilde{n}|^{\frac{131}{588}}
+  \tilde{m}^{\frac{131}{588}} u_n^{\frac12}
+ |\tilde{n}|^{\frac{131}{588}} u_m^{\frac12}+
(u_mu_n)^{\frac12} \]
and
\begin{align}\label{AuDef}
	\begin{split}
		A_u(m,n)&\defeq A(m,n)^{\frac14}  | \tilde{m}  \tilde{n}|^{\frac3{16}}\\
		&\ll | \tilde{m}  \tilde{n}|^{\frac{143}{588}}+\tilde{m}^{\frac{143}{588}}|\tilde{n}|^{\frac 3{16}}\,u_n^{\frac18}+\tilde{m}^{\frac 3{16}}u_m^{\frac18}|\tilde{n}|^{\frac{143}{588}}+| \tilde{m}  \tilde{n}|^{\frac3{16}}(u_mu_n)^{\frac18}.
	\end{split}
\end{align}
Moreover, all implicit constants for  bounds in this section depend on $\nu$ and $\ep$ and we drop the subscripts unless specified. Recall the notations in Settings \ref{conditionATDelta} and \ref{conditionphi}. For the exceptional spectrum $r_j\in i(0,\frac14]$ of the Laplacian $\Delta_k$ on $\Gamma=\Gamma_0(N)$, we have $2 \im r_\Delta\leq \theta$ assuming $H_\theta$ \eqref{HTheta} by Proposition~\ref{specParaBoundWeightHalf} and $\im r_j$ has a positive lower bound $\underline t>0$ depending on $N$. 

\begin{proposition}\label{PrereqPropGEN}
	With the same setting as Theorem~\ref{mainThm3}, when $2x\geq A_u(m,n)^2$, we have 
	\begin{align}\label{PrereqPropGENequation}
		\begin{split}
			\sum_{\substack{x< c\leq 2x\\N|c}} \frac{S(m,n,c,\nu)}{c} -\sum_{r_j\in i (0,\frac14]}(2^{2s_j-1}-1)&\tau_j(m,n)\frac{x^{2s_j-1}}{2s_j-1}\\
			&\ll \(x^{\frac 16}+A_u(m,n)\)|\tilde{m}\tilde{n}x|^\ep.
		\end{split}
	\end{align}
\end{proposition}

We first prove that Proposition~\ref{PrereqPropGEN} implies Theorem~\ref{mainThm3}. For each $j$, let $\rho_j(n)$ denote the coefficients of an orthonormal basis $\{v_j(\cdot)\}$ of $\LEigenform_{k}(N,\nu)$. For each singular cusp $\mathfrak{a}$ of $\Gamma=\Gamma_0(N)$, let $\Ea(\cdot,s)$ be the associated Eisenstein series and $\rho_{\mathfrak{a}}(n,r)$ be defined as in \eqref{FourierExpEsst}.

Recall the definition of $\tau_j(m,n)$ in Theorem~\ref{mainThm} and $2\im r_j=2s_j-1 \in(0,\frac12]$ and $\underline{t}>0$ as the lower bound of $\im r_j$ depending on $\nu$. 
The sum to be estimated is 
\begin{equation}\label{ToEstmtThm1.4}
	\sum_{\substack{N|c\leq X}} \frac{S(m,n,c,\nu)}{c} -\sum_{r_j\in i (0,\frac14]}\tau_j(m,n)\frac{X^{2s_j-1}}{2s_j-1}. 
\end{equation} 
For $r_j\in i(0,\frac14]$, by Proposition~\ref{AhgDunEsstBd} we have
\begin{equation}\label{Bound for tau_j}
	\frac{\tau_j(m,n)}{2s_j-1}\ll |\rho_j(m)\rho_j(n)||\tilde{m}\tilde{n}|^{1-s_j}\ll A(m,n)|\tilde{m}\tilde{n}|^{\frac12-s_j+\ep}.
\end{equation}
When $X\ll A_u(m,n)^2$, since $A(m,n)\leq 2|\tilde{m}\tilde{n}|^{\frac14}$, 
\begin{align}\label{taujestGEN}
\begin{split}
	\tau_j(m,n)\frac{X^{2s_j-1}}{2s_j-1}\ll  A(m,n)&|\tilde{m}\tilde{n}|^{\frac12-s_j+\ep}A_u(m,n)^{4s_j-2}\\
	&= A(m,n)^{s_j+\frac12}|\tilde{m}\tilde{n}|^{\frac18-\frac14s_j}\ll A_u(m,n)  . 
\end{split}
\end{align}
So in this case we get Theorem~\ref{mainThm3} where the $\tau_j$ terms are absorbed in the errors. 

When $X\geq A_u(m,n)^2$, the segment for summing Kloosterman sums on $1\leq c\leq A_u(m,n)^2$ contributes a $O_{\nu,\ep}(A_u(m,n)|\tilde{m}\tilde{n}|^{\ep})$ by condition (2) of Definition \ref{Admissibility}. The segment for $A_u(m,n)^2\leq c\leq X$ can be broken into no more than $O(\log X)$ dyadic intervals $x<c\leq 2x$ with $A_u(m,n)^2\leq x\leq \frac X2$ and we use Proposition~\ref{PrereqPropGEN} for both the Kloosterman sum and the $\tau_j$ terms. In summing dyadic intervals, for each $r_j\in i(0,\frac14]$, we get
\begin{align*}
\sum_{\ell=1}^{\ceil{\log_2 \(X/A_u(m,n)^2\)}}&\frac{(2^{2s_j-1}-1)\tau_j(m,n)}{2s_j-1}
\(\frac{X}{2^\ell}\)^{2s_j-1}\\
&\qquad =\frac{\tau_j(m,n)}{2s_j-1}X^{2s_j-1}
\( 1 - 2^{(1-2s_j)\ceil{\log_2 \(X/A_u(m,n)^2\)}}\).
\end{align*}
The difference between the above quantity and $\tau_j(m,n)\dfrac{X^{2s_j-1}}{2s_j-1}$ in \eqref{ToEstmtThm1.4} is
\begin{equation}\label{differenceBetweenEstmtedTauAndTrueTauGEN}
	{\tau_j(m,n)}\frac{X^{2s_j-1}}{2s_j-1}\cdot 2^{(1-2s_j)\ceil{\log_2 \(X/A_u(m,n)^2\)}}\ll \frac{\tau_j(m,n)}{2s_j-1}A_u(m,n)^{4s_j-2}\ll A_u(m,n). 
\end{equation}
by \eqref{Bound for tau_j}. In conclusion, for $X\geq A_u(m,n)^2$ we get
\begin{align*}
	&\sum_{\substack{N|c\leq X}}\frac{S(m,n,c,\nu)}{c} -\sum_{r_j\in i (0,\frac14]}\tau_j(m,n)\frac{X^{2s_j-1}}{2s_j-1}\\
	&=\sum_{\substack{A_u(m,n)^2<c\leq X}} \frac{S(m,n,c,\nu)}{c}-\sum_{r_j\in i (0,\frac14]}\tau_j(m,n)\frac{X^{2s_j-1}}{2s_j-1}+O(A_u(m,n)|\tilde{m}\tilde{n}|^{\ep})\\
	&=\sum_{\ell=1}^{\ceil{\log_2 \(X/A_u(m,n)^2\)}}\(\sum_{\substack{\frac{X}{2^\ell}<c\leq \frac{X}{2^{\ell-1}}}} \frac{S(m,n,c,\nu)}{c} -\sum_{r_j\in i (0,\frac14]}\frac{(2^{2s_j-1}-1)\tau_j(m,n)}{2s_j-1}\(\frac{X}{2^\ell}\)^{2s_j-1}\)\\
	&\quad + O(A_u(m,n)|\tilde{m}\tilde{n}|^{\ep})\\
	&\ll \(X^{\frac 16}+A_u(m,n)\)|\tilde{m}\tilde{n}X|^\ep,
\end{align*}
where the second equality follows from \eqref{differenceBetweenEstmtedTauAndTrueTauGEN} and the last inequality is by Proposition~\ref{PrereqPropGEN}. 

It remains to prove Proposition~\ref{PrereqPropGEN}. For $r_j\in i(0,\frac14]$, by Proposition~\ref{AhgDunEsstBd} we have
\[\sqrt{|\tilde{m}\tilde{n}|}\; \overline{\rho_j(m)}\rho_j(n)\ll A(m,n)|\tilde{m}\tilde{n}|^{\ep}. \]
Applying Lemma \ref{mainphiLemma} where $2t_j=2\im r_j=2s_j-1$, recalling the definition of $\tau_j$ in Theorem~\ref{mainThm} and $a=4\pi \sqrt{|\tilde m \tilde n|}$ in Setting \ref{conditionATDelta}, we get
\begin{align}\label{taujReference5.4GEN}
	\begin{split}
	2i^k\cdot &4\sqrt{|\tilde{m}\tilde{n}|}\;\frac{\overline{\rho_j(m)}\rho_j(n)}{\ch \pi r_j}\check{\phi}(r_j)\\
	&=(2^{2s_j-1}-1)\tau_j(m,n)\frac{x^{2s_j-1}}{2s_j-1}+O\(A(m,n)|\tilde{m}\tilde{n}|^\ep\(|\tilde{m}\tilde{n}|^{-t_j}x^{2t_j-\delta}+1\)\). 
	\end{split}\end{align}
The error term is $O(A(m,n)|\tilde{m}\tilde{n}|^{\ep})$ when $2t_j\leq \delta$ and is $O(x^{\frac12-\delta}|\tilde m\tilde n|^\ep)$ when $t_j=\frac 14$. Thanks to Proposition~\ref{specParaBoundWeightHalf} we can choose $\delta>\theta$ ($\delta=\frac13>\frac7{64}$ in the end) and $t_j<\frac14$ implies $2t_j\leq \theta<\delta$. With the help of \eqref{taujReference5.4GEN} we break up the left hand side of \eqref{PrereqPropGENequation} as  
\begin{align}\label{mainDiffrenceGEN}
	\begin{split}
		&\left|\sum_{\substack{x<c\leq 2x\\N|c}} \frac{S(m,n,c,\nu)}{c} -\sum_{r_j\in i (0,\frac14]}(2^{2s_j-1}-1)\tau_j(m,n)\frac{x^{2s_j-1}}{2s_j-1} \right|\\
		\leq& \left|\sum_{\substack{x<c\leq 2x\\N|c}} \frac{S(m,n,c,\nu)}{c} -\sum_{N|c>0} \frac{S(m,n,c,\nu)}{c}\phi\(\frac ac\)\right| +O\(\(x^{\frac12-\delta}+A(m,n)\)|\tilde{m}\tilde{n}|^{\ep}\)\\
		&+ \left|\sum_{N|c>0} \frac{S(m,n,c,\nu)}{c}\phi\(\frac ac\)-8i^k\sqrt{|\tilde{m}\tilde{n}|}\sum_{r_j\in i (0,\frac14]}\frac{\overline{\rho_j(m)}\rho_j(n)}{\ch \pi r_j}\check{\phi}(r_j)\right|\\
		=:&\;S_1+O\(\(x^{\frac12-\delta}+A(m,n)\)|\tilde{m}\tilde{n}|^{\ep}\)+S_2. 
	\end{split}
\end{align}
Recall $T\asymp x^{1-\delta}$. The first sum $S_1$ above can be estimated by condition (2) of Definition \ref{Admissibility} as
\begin{align}\label{TraceSmoothingGEN}
	\begin{split}
		S_1\leq\sum_{\substack{x-T\leq c\leq x\\2x\leq c\leq 2x+2T\\N|c}}\frac{|S(m,n,c,\nu)|}{c}\ll_{\delta,\ep} x^{\frac12-\delta}|  \tilde{m}  \tilde{n}x|^\ep.
	\end{split}
\end{align}
We then prove a bound for $S_2$. Following from the trace formula \eqref{traceFormula}, 
\[S_2=8\sqrt{|\tilde{m}\tilde{n}|}\left|\sum_{r_j\geq 0}\frac{\overline{\rho_j(m)}\rho_j(n)}{\ch \pi r_j}\check{\phi}(r_j)+\!\!\!\sum_{\mathrm{singular\ }\mathfrak{a}}\!\!\!\frac{1}{4\pi}\int_{-\infty}^\infty  \overline{\rho_{\mathfrak{a}}\(m,r\)}\rho_{\mathfrak{a}}\(n,r\)\frac{ \check{\phi}(r) }{\ch \pi r} dr \right|. \]

When estimating $S_2$, we focus on the discrete spectrum $r_j\geq 0$, because the bounds provided by Proposition~\ref{AdsDukeEsstBd} and Proposition~\ref{AhgDunEsstBd} for $r_j\in I$ for any interval $I$ are the same as those provided for for $|r|\in I$ in the continuous spectrum. 
For $r\in [0,1)$, we apply Proposition~\ref{AhgDunEsstBd}, \eqref{Ris0} and \eqref{RLessThan1} to get
\begin{equation}\label{GeneralTraceRisSmall}
	\sqrt{ \tilde{m} | \tilde{n}|}\sum_{r\in [0,1)} \left|\frac{\overline{\rho_j(m)}\rho_j(n)}{\ch \pi r_j} \check{\phi}(r_j)\right| 
	\ll  A(m,n) |  \tilde{m}  \tilde{n}|^\ep.
\end{equation}

For $r\in[1,\frac a{8x})$, we apply Proposition~\ref{AhgDunEsstBd} and \eqref{Rlarge} with $\check{\phi}(r)\ll e^{-\frac r2}$.  
Since
\begin{align}
	\begin{split}
		S(R)\defeq \sqrt{ \tilde{m} | \tilde{n}|}\sum_{r\in [1,R]} \left|\frac{\overline{\rho_j(m)}\rho_j(n)}{\ch \pi r_j} \right|
		\ll  A(m,n)R^3|\tilde{m} \tilde{n}|^{\ep}
	\end{split}
\end{align}
by Cauchy-Schwarz, we have
\begin{align}\label{GeneralTraceRlarge1}
	\begin{split}
		\sqrt{ \tilde{m} | \tilde{n}|}\sum_{r\in [1,\frac a{8x})} \left|\frac{\overline{\rho_j(m)}\rho_j(n)}{\ch \pi r_j} \check{\phi}(r_j)\right|
		&\ll\sqrt{ \tilde{m} | \tilde{n}|}\sum_{r\in [1,\frac a{8x})} \left|\frac{\overline{\rho_j(m)}\rho_j(n)}{\ch \pi r_j} \right|e^{-\frac {r_j}2}\\
		&\ll e^{-\frac r2}S(r)\Big|_{r=1}^{\frac a{8x}}+\int_1^{\frac{a}{8x}}S(r) e^{-\frac r2}dr\\
		&\ll A(m,n) |  \tilde{m}  \tilde{n}x|^\ep \(1+\int_1^{\frac{a}{8x}} e^{-\frac r2}r^3dr\)\\
		&\ll A(m,n) |  \tilde{m}  \tilde{n}x|^\ep.
	\end{split}
\end{align}

For $r\in[\frac{a}{8x},\frac{a}{x})$, we apply Proposition~\ref{AdsDukeEsstBd} on $\tilde{m}$, Proposition~\ref{AhgDunEsstBd} on $\tilde{n}$ and \eqref{Rlarge} with $\check{\phi}(r)\ll \frac 1r\ll \frac xa$ to get
\begin{align}\label{GeneralRLarge2Step1}
	\begin{split}
		\sqrt{ \tilde{m} | \tilde{n}|}&\sum_{\frac a{8x}\leq r< \frac ax} \left|\frac{\overline{\rho_j(m)}\rho_j(n)}{\ch \pi r_j} \check{\phi}(r_j)\right|\\
		&\ll \(\frac ax+\tilde{m}^{\frac14}\)\(\frac ax\)^{\frac12}\(|\tilde{n}|^{\frac{131}{294}}+u_n\)^{\frac12}|\tilde{m}\tilde{n}x|^\ep\\
		&\ll \(A(m,n)\(\frac ax\)^{\frac32}+\tilde{m}^{\frac14}\(|\tilde{n}|^{\frac{131}{294}}+u_n\)^{\frac12}  \(\frac ax\)^{\frac 12}\)|\tilde{m}\tilde{n}x|^\ep.	 
	\end{split}
\end{align}
Exchanging the propositions applied on $\tilde{m}$ and $\tilde{n}$ gives a symmetric estimate. These two estimates conclude
\begin{align}\label{GeneralRLarge2Step2}
	\begin{split}
		&\ \ \sqrt{ \tilde{m} | \tilde{n}|}\sum_{\frac a{8x}\leq r< \frac ax} \left|\frac{\overline{\rho_j(m)}\rho_j(n)}{\ch \pi r_j} \check{\phi}(r_j)\right|\\
		&\ll \(\frac ax\)^{\frac12}\left\{A(m,n)\frac ax+\min\(\tilde{m}^{\frac14}\(|\tilde{n}|^{\frac{131}{294}}+u_n\)^{\frac12},|\tilde{n}|^{\frac14}\(\tilde{m}^{\frac{131}{294}}+u_m\)^{\frac12}\)\right\}|\tilde{m}\tilde{n}x|^\ep\\
		&\ll \(A(m,n)\(\frac ax\)^{\frac32}+|\tilde{m}\tilde{n}|^{\frac18}A(m,n)^{\frac12}\(\frac ax\)^{\frac12}\)|\tilde{m}\tilde{n}x|^\ep, 
	\end{split}
\end{align}
where in the last inequality we applied $\min(B,C)\leq \sqrt{BC}$ and the definition of $A(m,n)$ at the beginning of this subsection.

Let 
\[P(m,n)\defeq 2|\tilde{m}\tilde{n}|^{\frac18}A(m,n)^{-\frac12}\geq 1.\] 
Divide $r\geq \max(\frac ax,1)$ into two parts: 
 $\max\(\frac ax,1\)\leq r< P(m,n)$
  and 
  $r \geq  \max\(\frac ax,1,P(m,n)\)$. 
We apply Proposition~\ref{AhgDunEsstBd} on the first range and \eqref{Rlarge} with $\check{\phi}(r)\ll r^{-\frac 32}$ to get
\begin{align}\label{TraceRlargeGeneralFirstPartEstmt}
	\begin{split}
	\sqrt{ \tilde{m}| \tilde{n}|} \sum_{\max(\frac ax,1)\leq r_j<P(m,n)} &\left| \frac{\overline{\rho_j(m)} \rho_j(n)}{\ch \pi r_j}\check{\phi}(r_j)\right|\\
	&\ll\sqrt{ \tilde{m}| \tilde{n}|} \sum_{\max(\frac ax,1)\leq r_j<P(m,n)} \left| \frac{\overline{\rho_j(m)} \rho_j(n)}{\ch \pi r_j}\right|r^{-\frac 32}\\
	&\ll r^{-\frac 32}S(r)\Big|_{r=\max(\frac ax,1)}^{P(m,n)} +\int_{\max(\frac ax,1)}^{P(m,n)}r^{-\frac  52}S(r)dr\\
	&\ll|  \tilde{m}  \tilde{n}|^{\frac3{16}}A(m,n)^{\frac14}| \tilde{m}  \tilde{n}x|^\ep
	\end{split}
\end{align}
by partial summation. We divide the second range into dyadic intervals $C\leq r_j< 2C$ and apply Proposition~\ref{AdsDukeEsstBd} and \eqref{Rlarge} with $\check{\phi}(r)\ll \min(r^{-\frac 32},r^{-\frac 52}\frac xT)$ to get
\begin{align}\label{TraceRlargeGeneralSecondEstmtDyadic}
	\begin{split}
		\sqrt{ \tilde{m}| \tilde{n}|}& \sum_{C\leq r_j< 2C} \left| \frac{\overline{\rho_j(m)} \rho_j(n)}{\ch \pi r_j}\check{\phi}(r_j)\right|\\
		&\ll \min\(C^{-\frac32},C^{-\frac 52}\frac xT\)\(C^2+ (\tilde{m}^{\frac14} +|\tilde{n}|^{\frac14})C+ | \tilde{m}  \tilde{n}|^{\frac14}\)|  \tilde{m}  \tilde{n}x|^\ep\\
		&\ll \(\min\(C^{\frac12},C^{-\frac 12}\frac xT\)+ (\tilde{m}^{\frac14} +|\tilde{n}|^{\frac14})C^{-\frac12}+ | \tilde{m}  \tilde{n}|^{\frac14}C^{-\frac 32}\)|  \tilde{m}  \tilde{n}x|^\ep.
	\end{split}
\end{align}
Next we sum over dyadic intervals. For the first term $\min(C^{\frac12},C^{-\frac 12}\frac xT)$, when
\[\min\(C^{\frac12},C^{-\frac 12}\frac xT\)=C^{\frac12}:\quad \sum_{\substack{j\geq 1:\ 2^jC=\frac xT\\ C\geq P(m,n)}} C^{\frac12}\leq \sum_{j=1}^\infty 2^{-\frac j2}\(\frac xT\)^{\frac12}\ll \(\frac xT\)^{\frac12}, \]
and when
\[\min\(C^{\frac12},C^{-\frac 12}\frac xT\)=C^{-\frac12}\frac xT:\quad \sum_{j\geq 0:\ C=2^{j}\frac xT} C^{-\frac12}\frac xT\leq \sum_{j=0}^\infty 2^{-\frac j2}\(\frac xT\)^{\frac12}\ll \(\frac xT\)^{\frac12} . \]
So after summing up from \eqref{TraceRlargeGeneralSecondEstmtDyadic} and recalling $T\asymp x^{1-\delta}$ in Setting \ref{conditionATDelta}, we have
\begin{align}\label{TraceRlargeGeneralSecondEstmt}
	\begin{split}
		\sqrt{ \tilde{m}| \tilde{n}|}& \sum_{r_j\geq \max(\frac ax,1, P(m,n))} \left| \frac{\overline{\rho_j(m)} \rho_j(n)}{\ch \pi r_j}\check{\phi}(r_j)\right|\\
		&\ll \(\(\frac xT\)^{\frac12}+ (\tilde{m} +|\tilde{n}|)^{\frac14}| \tilde{m}  \tilde{n}|^{-\frac1{16}}A(m,n)^{\frac14}+ | \tilde{m}  \tilde{n}|^{\frac1{16}}A(m,n)^{\frac 34}\)|  \tilde{m}  \tilde{n}x|^\ep\\
		&\ll \( x^{\frac\delta2} + |  \tilde{m}  \tilde{n}|^{\frac3{16}}A(m,n)^{\frac14}\) |  \tilde{m}  \tilde{n}x|^\ep,
	\end{split}
\end{align}
where the last inequality is by $|  \tilde{m}  \tilde{n}|^{\frac14}\gg A(m,n)$. Combining \eqref{TraceRlargeGeneralFirstPartEstmt} and \eqref{TraceRlargeGeneralSecondEstmt} we have
\begin{equation}\label{TraceRlarge3General}
	\sqrt{ \tilde{m}| \tilde{n}|} \sum_{r\geq \max(\frac ax,1)} \left| \frac{\overline{\rho_j(m)} \rho_j(n)}{\ch \pi r_j}\check{\phi}(r_j)\right|\\
	\ll \( x^{\frac\delta2} + A_u(m,n)\) |  \tilde{m}  \tilde{n}x|^\ep. 
\end{equation}

\begin{proof}[Proof of Proposition~\ref{PrereqPropGEN}]
	Clearly $A_u(m,n)\geq A(m,n)$. Combining \eqref{mainDiffrenceGEN}, \eqref{TraceSmoothingGEN}, \eqref{GeneralTraceRisSmall}, \eqref{GeneralTraceRlarge1}, \eqref{GeneralRLarge2Step2}, and \eqref{TraceRlarge3General} we get 
	\begin{align}
		\begin{split}
			&\sum_{\substack{x<c\leq 2x\\N|c}}\frac{S(m,n,c,\nu)}{c} -\sum_{r_j\in i (0,\frac14]}(2^{2s_j-1}-1)\tau_j(m,n)\frac{x^{2s_j-1}}{2s_j-1} \\
			&\quad\ll \(x^{\frac12-\delta}+A_u(m,n)+A(m,n)\(\frac ax\)^{\frac32}+|\tilde{m}\tilde{n}|^{\frac18}A(m,n)^{\frac12}\(\frac ax\)^{\frac12}+x^{\frac\delta 2}\)|\tilde{m}\tilde{n}x|^\ep. 
		\end{split}
	\end{align}
	Since $2x\geq A_u(m,n)^2$ by assumption, we have
	\[\frac ax\ll |\tilde{m}\tilde{n}|^{\frac18}A(m,n)^{-\frac12},\]
	which implies both
	\begin{equation}\label{WhenXLargeAu}
		A(m,n)\(\frac ax\)^{\frac32}\ll A_u(m,n)\quad \text{and}\quad  |\tilde{m}\tilde{n}|^{\frac18}A(m,n)^{\frac12}\(\frac ax\)^{\frac12}\ll A_u(m,n).
	\end{equation} 
Taking $\delta =\frac 13$ we get the desired bound. 
	
\end{proof}

\begin{proof}[Proof of Theorem~\ref{mainThm}]
	Taking $u_m=u_n=1$ we get the theorem from Theorem~\ref{mainThm3}. In fact, when $u_m$ and $u_n$ are $O_{\nu}(1)$, we can still get the desired theorem. 
\end{proof}

\begin{proof}[Proof of Theorem~\ref{mainThm2}]
Assuming $H_\theta$ \eqref{HTheta}, for the range $r\in i(0,\theta]$, we apply Cauchy-Schwarz, Proposition~\ref{AhgDunEsstBd}, and \eqref{RisIm} to get
\begin{equation}\label{GeneralTraceRisIm}
	\sqrt{ \tilde{m} | \tilde{n}|}\sum_{r\in i(0,\theta]} \left|\frac{\overline{\rho_j(m)}\rho_j(n)}{\ch \pi r_j} \check{\phi}(r_j)\right| 
	\ll A(m,n)|  \tilde{m}  \tilde{n}|^{-\frac\theta2} x^{\theta}|  \tilde{m}  \tilde{n}x|^\ep. 
\end{equation}
	Taking this into account in the proof of Theorem~\ref{mainThm3} (especially modifying \eqref{mainDiffrenceGEN}) and by Proposition~\ref{specParaBoundWeightHalf} we get the desired bound. 
\end{proof}

\subsection{Proof of Theorem~\ref{mainThmLastSec}}

By  Theorem~\ref{mainThm2}, we claim that
	\begin{align}\label{mainThm2BoundEditedAlphaNuPositive}
		s(m,n,X)\defeq \sum_{N|c\leq X}\!\!\! \frac{S(m,n,c,\nu)}{c}\ll_{\nu,\ep}\(|\tilde{m}  \tilde{n}|^{\frac{131}{588}-\frac\theta2}X^{\theta}+  |  \tilde{m}  \tilde{n} |^{\frac{143}{588}} +X^{\frac16}\) |  \tilde{m}  \tilde{n}X|^\ep.  
	\end{align}
It suffices to show $\tau_j(m,n)=0$ for $r_j=\frac i4$. Recall \eqref{CuspFormR0} for the Fourier expansions of Maass forms with eigenvalue $\lambda_0=\frac3{16}$. For weight $k=\frac 12$, we have $\tilde n<0\Rightarrow n\leq 0\Rightarrow \rho_j(n)=0$ for $r_j=\frac i4$. Similarly, for weight $k=-\frac 12$ we have $\rho_j(m)=0$ for $r_j=\frac i4$. Thus we ensure the bound above.

To prove Theorem~\ref{mainThmLastSec}, we need the following bounds on Bessel functions. By \cite[(10.30.1)]{dlmf}, for fixed $\beta, \ell>0$ and $0\leq z\leq \ell$, 
	\[I_{\beta}(z)\ll_{\beta,\ell}z^{\beta}. \]
	Recall our notation $a=4\pi\sqrt{|\tilde m\tilde{n}|}$ and let $L\defeq h\sqrt{|\tilde m\tilde{n}|}\asymp_{h}a $. Then $I_{\beta}(a/L)\ll_{h} 1$. When $t\geq L$, 
	by \cite[Lemma~7.3]{ahlgrendunn} and \cite[(10.3), Lemma~10.1]{AAimrn}, 
	\[\frac{d}{dt}I_{\frac 12}\(\frac at\)\ll_{h } a^{\frac52} t^{-\frac72}+a^{\frac12}t^{-\frac32},\quad \frac{d}{dt}I_{\frac 32}\(\frac at\)\ll_{h} a^{\frac32} t^{-\frac52}. \]
With \eqref{mainThm2BoundEditedAlphaNuPositive}, for $\beta=\frac12$ or $\frac 32$ we have
		\begin{align*}
			\sum_{N|c>L}\frac{S(m,n,c,\nu)}cI_{\beta}\(\frac ac\)
			&=-s(m,n,L)I_{\beta}\(\frac aL\)-\int_L^\infty s(m,n,t)\frac d{dt}I_{\beta}\(\frac at\)dt\\
			&\ll_{h ,\ep} |\tilde m\tilde{n}|^{\frac{143}{588}+\ep}
		\end{align*}

	A similar process for the $J$-Bessel functions follows from 
	\cite[(10.14.4)]{dlmf} $|J_{\beta}(x)|\leq_\beta x^\beta$, \cite[(10.6.1)]{dlmf}  
	\begin{equation}
		\frac{d}{dt}J_{\frac 32}\(\frac at\)=-\frac a{2t^2}\(J_{\frac12 }\(\frac at\)-J_{\frac52 }\(\frac at\)\)\ll a^{\frac 32}t^{-\frac52}+a^{\frac72}t^{-\frac 92}, 
	\end{equation}
	and
\begin{equation}\label{DerivativeJHalf}
		\frac{d}{dt}J_{\frac 12}\(\frac at\)=-\frac a{2t^2}\(J_{-\frac12 }\(\frac at\)-J_{\frac32 }\(\frac at\)\)\ll a^{\frac12}t^{-\frac32}+a^{\frac 52}t^{-\frac72}.
\end{equation}

{

}

\section{Proof of Theorem~\ref{FourierExpn_ofPoincareSeries}}

The only thing left is to prove the Fourier expansion of Theorem~\ref{FourierExpn_ofPoincareSeries}. Recall the notations in Section 4 for $\tilde m<0$ and let $M_{\alpha,\beta}$ and $W_{\alpha,\beta}$ denote the $M$- and $W$-Whittaker functions, respectively. For $y>0$, by \cite[(13.18.4)]{dlmf},
\begin{align}\label{SpecialWhittakerM}
	\begin{split}
		\mathcal M_{1-\frac k2}(-y)&=y^{-\frac k2}M_{-\frac k2,\,\frac12-\frac k2}(y)\\
		&=(1-k)(\Gamma(1-k)-\Gamma(1-k,y))e^{\frac{y}2},
	\end{split}
\end{align} 
and by \cite[(13.18.2)]{dlmf},
\begin{equation}\label{SpecialWhittakerW}
	W_{-\frac k2,\,\frac12-\frac k2}(y)= y^{\frac k2}e^{\frac y2}\Gamma(1-k,y),\quad 
	W_{\frac k2,\,\frac12-\frac k2}(y)= y^{\frac k2}e^{-\frac y2}. 
\end{equation}
The contribution to $P_k(1-\frac k2,m,N;z)$ from $c=0$ in \eqref{MaassPoincareSeries} equals
\[\frac1{\Gamma(2-k)}\varphi_{1-\frac k2,\,k}(\tilde mz)=\frac{1-k}{\Gamma(2-k)}(\Gamma(1-k)-\Gamma(1-k,4\pi|\tilde m|y))e^{2\pi \tilde mz}. \]
The contribution to $P_k(1-\frac k2,m,N;z)$ from some $c>0$ equals
\begin{align}\label{ContributionSingleC}
	\begin{split}
	\frac1{\Gamma(2-k)}&\sum_{\ell \in \Z}\sum_{\substack{d(c)^*\\0<a<c,\;ad\equiv 1(c)}}\overline{\nu}
		\begin{psmallmatrix}
		a&*\\c&d+\ell c
	\end{psmallmatrix}
(cz+d+\ell c)^{-k} \\
&\cdot\mathcal{M}_{1-\frac k2}\(\frac{4\pi \tilde{m}y}{|cz+d+\ell c|^2}\)e\(\frac{\tilde ma}c- \re\(\frac{\tilde m }{c(cz+d+\ell c)}\) \)\\
=\ &\frac1{\Gamma(2-k)}c^{-k}\sum_{\substack{d(c)^*}}\overline{\nu}\begin{psmallmatrix}
	a&b\\c&d
\end{psmallmatrix}e\(\frac{\tilde ma}c\) \sum_{\ell\in \Z}  e(\ell \alpha_\nu) \(z+\frac dc+\ell\)^{-k} \\
&\cdot\mathcal{M}_{1-\frac k2}\(\frac{4\pi \tilde{m}y}{c^2|z+\frac dc+\ell|^2}\)e\(-\frac{\tilde m } {c^2} \re\(\frac1{z+\frac dc+\ell}\) \),
\end{split}
\end{align}
where we used \eqref{MultiplierSystemBasicProprety}: $\nu\begin{psmallmatrix}
a&b+\ell a\\c&d+\ell c
\end{psmallmatrix}=\nu\begin{psmallmatrix}
a&b\\c&d
\end{psmallmatrix}\nu\begin{psmallmatrix}
1&\ell\\0&1
\end{psmallmatrix}=\nu\begin{psmallmatrix}
a&b\\c&d
\end{psmallmatrix}e(-\ell\alpha_{\nu})$ for all $\ell\in \Z$. 
Let
\[f(z)\defeq \sum_{\ell\in \Z} \frac{e(\ell \alpha_\nu)}{(z+\ell)^{k} } \;\mathcal{M}_{1-\frac k2}\(\frac{4\pi \tilde{m}y}{c^2|z+\ell|^2}\)e\(-\frac{\tilde m } {c^2} \re\(\frac1{z+\ell}\) \).\]
Then $f(z)e(\alpha_{\nu} x)$ has period $1$ and $f$ has Fourier expansion
\begin{equation}\label{fzdc}
f(z)=\sum_{n\in \Z}a_y(n)e(\tilde n x),\quad f\(z+\frac dc\)=\sum_{n\in \Z} a_y(n)e\(\frac {\tilde n d}c\)e(\tilde n x), 
\end{equation}
where by \eqref{SpecialWhittakerM}
\begin{align*}
a_y(n)&=\int_\R z^{-k} \mathcal{M}_{1-\frac k2}\(\frac{4\pi \tilde{m}y}{c^2|z|^2}\)e\(-\frac{\tilde m x} {c^2|z|^2} -nx+\alpha_{\nu}x\)dx\\
&=\frac{c^k}{|4\pi \tilde m y|^{\frac k2}}\int_\R \(\frac{x-iy}{x+iy}\)^{\frac k2} M_{-\frac k2,\frac12-\frac k2}\(\frac{4\pi |\tilde{m}|y}{c^2|z|^2}\)e\(-\frac{\tilde m x} {c^2|z|^2} -\tilde nx\)dx. 
\end{align*}
By substituting $x=yu$, 
\begin{align*}
a_y(n)=\frac{i^{-k}yc^{k}}{|4\pi \tilde m y|^{\frac k2}}\int_\R \(\frac{1+iu}{1-iu}\)^{\frac k2} M_{-\frac k2,\frac12-\frac k2}\(\frac{4\pi |\tilde{m}|}{c^2y(u^2+1)}\)e\(\frac{|\tilde m| u} {c^2y(u^2+1)} -\tilde nyu\)du. 
\end{align*}
This integral is evaluated by \cite[p.32-33]{BruinierBookBorcherds}. We get

\begin{align*}
	a_y(n)=\frac{i^{-k}c^k\Gamma(2-k)}{|4\pi \tilde m y|^{\frac k2}c}
	\cdot\left\{ \begin{array}{ll}
		\dfrac{2\pi \sqrt{|\tilde m/\tilde n|}}{\Gamma(1-k)}\, W_{-\frac k2,1-\frac k2}(4\pi |\tilde n| y) J_{1-k}\(\dfrac{4\pi\sqrt{|\tilde m\tilde n|}}c\),& \tilde n<0;\\
		\dfrac{4\pi^{2-\frac k2} |\tilde m|^{1-\frac k2} c^{k-1} y^{\frac k2}}{(1-k)\Gamma(1-k)},& \tilde n=0;\\
		{2\pi \sqrt{|\tilde m/\tilde n|}}\, W_{\frac k2,1-\frac k2}(4\pi \tilde n y) I_{1-k}\(\dfrac{4\pi\sqrt{|\tilde m\tilde n|}}c\),& \tilde n>0.\\
	\end{array}
	\right. 
\end{align*}
Applying \eqref{SpecialWhittakerW}, substituting \eqref{fzdc} in \eqref{ContributionSingleC}, interchanging the finite sum on $d$ and sum on $n$, and summing over $N|c>0$ we get
\begin{align*}
	&P_k(1-\tfrac k2,m,N;z)=\frac{1-k}{\Gamma(2-k)}(\Gamma(1-k)-\Gamma(1-k,4\pi|\tilde m|y))e^{2\pi \tilde mz}
	+\sum_{n\in \Z} e^{2\pi i \tilde nz }\\
	&\ \cdot \left\{
	\begin{array}{ll}
		{\displaystyle
		 \frac{i^{-k} 2\pi\Gamma(1-k,4\pi|\tilde n|y)}{\Gamma(1-k)} \Big|\frac{\tilde m}{\tilde n}\Big|^{\frac{1-k}2}\sum_{N|c>0}\dfrac{S(m,n,c,\nu)}c J_{1-k}\(\frac{4\pi\sqrt{|\tilde m\tilde n|}}c\),}& \tilde n<0;\\
		 {\displaystyle
		 	\frac{i^{-k} (2\pi)^{2-k}|\tilde m|^{1-k}}{\Gamma(2-k)} \sum_{N|c>0} \frac{S(m,0,c,\nu)}{c^{2-k}}  }& \tilde n=0;\\
		 {\displaystyle
		 {i^{-k} 2\pi} \Big|\dfrac{\tilde m}{\tilde n}\Big|^{\frac{1-k}2}\sum_{N|c>0}\dfrac{S(m,n,c,\nu)}c I_{1-k}\(\dfrac{4\pi\sqrt{|\tilde m\tilde n|}}c\),}& \tilde n>0.
	\end{array}
\right.
\end{align*}
Since we assumed $\alpha_\nu>0$, we do not have the term for $\tilde n=0$. 

It remains to prove the convergence of Fourier coefficients when $k=\frac12$ and $\alpha_{\nu}>0$. For $\tilde n>0$, the convergence follows from Theorem~\ref{mainThmLastSec}. For $\tilde n<0$ we have the same sign case $\tilde m\tilde n>0$. We cite the generalized result by Goldfeld and Sarnak \cite{gs} where $s_j=\frac12+\im r_j$: 
\[s(x)\defeq\sum_{N|c\leq x}\frac{S(m,n,c,\nu)}c=\sum_{r_j\in i(0,\frac14]}\tau_j(m,n)\frac{x^{2s_j-1}}{2s_j-1}+O_{\nu,m,n,\ep}\(x^{\frac13+\ep}\). \]
The exponent $\frac13$ is from the trivial bound $|S(m,n,c,\nu)|\ll_{m,n} c$. 

We first show $\alpha_{\nu}>0$ implies $\tau_j(m,n)=0$ for $r_j=\frac i4$. Recall \eqref{CuspFormR0} for the Fourier expansion of $v_j(\cdot)\in \LEigenform_{\frac12}(N,\nu,\frac i4)$. When $\tilde m,\tilde n<0$, we have $\rho_j(m)=\rho_j(n)=0$ and $\tau_j(m,n)=0$ for $r_j=\frac i4$. Then the exceptional spectral parameter $r_j=\frac i4$ contributes $0$ to $s(x)$. The contribution from the other exceptional eigenvalues are in $r_j\in i(0,\frac 7{128}]$ with exponent ${2s_j-1=2\im r_j<\frac13}$ by $H_\frac 7{64}$ \eqref{HTheta} and Proposition~\ref{specParaBoundWeightHalf}. Finally we get
\[s(x)\ll_{N,\nu,m,n,\ep} x^{\frac13+\ep}. \]
Now we check for the convergence. Recall our notation $a= 4\pi \sqrt{\tilde m\tilde n}$, by \eqref{DerivativeJHalf}, 
\begin{align*}
	\sum_{N|c,\ x<c\leq y}\frac{S(m,n,c,\nu)}c J_{\frac12}\(\frac ac\)&=s(t)J_{\frac12}\(\frac at\)\Big|_{t=x}^{t=y}+O\(\int_x^y s(t)\(\frac{a^{\frac12}}{t^{\frac32}}+\frac{a^{\frac 52}}{t^{\frac72}}\)dt\)\\
	&\ll a^{\frac12}x^{-\frac16+\ep}+a^{\frac12}y^{-\frac16+\ep} + a^{\frac52}x^{-\frac{13}6+\ep}+a^{\frac52}y^{-\frac{13}6+\ep}.  
\end{align*} 
By Cauchy's convergence test we are done. For the $I$-Bessel function we have a similar bound. 

\section*{Acknowledgement}
The author extends sincere thanks to the referee for their careful review and valuable comments. Special gratitude is also owed to Professor Scott Ahlgren for his plenty of helpful discussions and suggestions. The author also thanks Nick Andersen and Alexander Dunn for their enlightening comments. 

	\bibliographystyle{alpha}
	\bibliography{allrefs}

\end{document}